\documentclass[reqno]{amsart}
\usepackage{amsmath}
\usepackage{amssymb}
\usepackage{amsfonts}
\usepackage{amsthm}
\usepackage{mathtools}
\usepackage{leftindex}
\usepackage{stmaryrd}
\usepackage{tikz} 
\usepackage{soul,xcolor}
\usetikzlibrary{tqft,calc}
\usetikzlibrary{cd}
\usepackage{enumitem}
\usepackage{graphicx}
\usepackage[left=2.5cm,right=2.5cm,top=3cm,bottom=3cm]{geometry}
\usepackage{tikz-cd}
\usepackage[hidelinks=true]{hyperref}
\usepackage{cite}

\newcommand{\C}{\IC}

\newcommand{\tr}{\mathrm{tr}}
\newcommand{\g}{\mathfrak{g}}

\newcommand{\h}{\mathfrak{h}}

\newcommand{\greg}{\mathfrak{g}_{\mathrm{reg}}}

\newcommand{\fsl}{\mathfrak{sl}}

\newcommand{\Id}{\mathrm{Id}}
\newcommand{\F}{\mathcal F}
\newcommand{\ad}{\mathrm{ad}}
\newcommand{\Ad}{\mathrm{Ad}}

\newcommand{\gdreg}{\mathfrak{g}^*_{\mathrm{reg}}}

\makeatletter
\@namedef{subjclassname@1991}{2020 Mathematics Subject Classification}
\makeatother

\numberwithin{equation}{section}

\newtheorem{theorem}{Theorem}[section]
\newtheorem{proposition}[theorem]{Proposition}

\newtheorem{lemma}[theorem]{Lemma}

\newtheorem{conjecture}[theorem]{Conjecture}
\newtheorem{mtheorem}{Main Theorem}

\theoremstyle{definition}
\newtheorem{definition}[theorem]{Definition}
\newtheorem{example}[theorem]{Example}
\newtheorem{remark}[theorem]{Remark}

\newcommand {\IC}{\mathbb{C}}

\newcommand {\reg}{_{\operatorname{reg}}}

\newcommand {\G}{\mathcal G}

\newcommand {\Hom}{\operatorname{Hom}}

\newcommand{\SL}{\operatorname{SL}}
\newcommand{\tto}{\;\substack{\longrightarrow\\[-9pt] \longrightarrow}\;}
\newcommand{\MT}{\mathbf{MT}}
\newcommand{\AMT}{\mathbf{AMT}}

\renewcommand{\H}{\mathcal{H}}
\newcommand{\too}{\longrightarrow}
\newcommand{\Spec}{\operatorname{Spec}}
\newcommand{\s}{\subseteq}
\newcommand{\mto}{\mapsto}
\newcommand{\mtoo}{\longmapsto}
\newcommand{\sll}[1]{\mkern-4mu\mathbin{/\mkern-5mu/}_{\mkern-4mu{#1}}}

\newcommand{\defn}[1]{\textbf{\textit{#1}}}

\newcommand{\im}{\operatorname{im}}

\newcommand{\htimes}{\times^{\mathrm{h}}}
\newcommand{\stimes}{\times^{\mathrm{s}}}
\newcommand{\obj}{^{0}}
\newcommand{\arr}{^{1}}

\newcommand{\sss}{\mathsf{s}}
\newcommand{\ttt}{\mathsf{t}}
\newcommand{\mmm}{\mathsf{m}}
\newcommand{\iii}{\mathsf{i}}
\newcommand{\uuu}{\mathsf{1}}

\newcommand{\pr}{\mathrm{pr}}
\newcommand{\I}{\mathcal{I}}
\renewcommand{\L}{\mathcal{L}}
\newcommand{\M}{\mathcal{M}}

\newcommand{\K}{\mathcal{K}}

\newcommand{\Cob}{\mathbf{Cob}}
\newcommand{\WS}{\mathbf{WS}_1}
\newcommand{\WSQ}{\mathbf{WS}"_{\mkern-11mu 1}}
\newcommand{\AWSQ}{\mathbf{AWS}"_{\mkern-11mu 1}}
\newcommand{\A}{\mathcal{A}}

\newcommand{\fp}[2]{\leftindex_{#1}\times_{#2}\,}
\newcommand{\aff}{^{\mathrm{Aff}}}

\begin{document}
	
	\title[The Moore--Tachikawa conjecture via shifted symplectic geometry]{The Moore--Tachikawa conjecture via\\
		shifted symplectic geometry}
	
	\author[Peter Crooks]{Peter Crooks}
	\author[Maxence Mayrand]{Maxence Mayrand}
	\address[Peter Crooks]{Department of Mathematics and Statistics\\ Utah State University \\ 3900 Old Main Hill \\ Logan, UT 84322, USA}
	\email{peter.crooks@usu.edu}
	\address[Maxence Mayrand]{D\'{e}partement de math\'{e}matiques \\ Universit\'{e} de Sherbrooke \\ 2500 Bd de l’Universit\'{e} \\ Sherbrooke, QC, J1K 2R1, Canada}
	\email{maxence.mayrand@usherbrooke.ca}
	
	
	\begin{abstract}
		We use shifted symplectic geometry to construct the Moore--Tachikawa topological quantum field theories (TQFTs) in a category of Hamiltonian schemes. Our new and overarching insight is an algebraic explanation for the existence of these TQFTs, i.e. that their structure comes naturally from three ingredients: Morita equivalence, as well as multiplication and identity bisections in abelian symplectic groupoids. Using this insight, we generalize the Moore--Tachikawa TQFTs in two directions. 

		The first generalization concerns a $1$-shifted version of the Weinstein symplectic category $\WS$. Each  abelianizable quasi-symplectic groupoid $\G$ is shown to determine a canonical $2$-dimensional TQFT $\eta_{\G}:\Cob_2\longrightarrow\WS$. We recover the open Moore--Tachikawa TQFT and its multiplicative counterpart as special cases. 
		
		Our second generalization is an affinization process for TQFTs. We first enlarge Moore and Tachikawa's category $\MT$ of holomorphic symplectic varieties with Hamiltonian actions to $\AMT$, a category of affine Poisson schemes with Hamiltonian actions of affine symplectic groupoids.
		We then show that if $\G \tto X$ is an affine symplectic groupoid that is abelianizable when restricted to an open subset $U \s X$ satisfying Hartogs' theorem, then $\G$ determines a TQFT $\eta_\G : \Cob_2 \too \AMT$.
		In more detail, we first devise an affinization process sending 1-shifted Lagrangian correspondences in $\WS$ to Hamiltonian Poisson schemes in $\AMT$.
		The TQFT is obtained by composing this affinization process with the TQFT $\eta_{\G|_U} : \Cob_2 \too \WS$ of the previous paragraph.
		Our results are also shown to yield new TQFTs outside of the Moore--Tachikawa setting.

	\end{abstract}
	
	
	\maketitle
	\setcounter{tocdepth}{2}
	\tableofcontents
	\vspace{-10pt}

	
	\section{Introduction}
	
	\subsection{Context}
	The Moore--Tachikawa conjecture \cite{moo-tac:11} inspires important, topical, and ongoing research at the interface of geometric representation theory, Lie theory, low-dimensional topology, holomorphic symplectic geometry, and theoretical physics. It arises in a string-theoretic context, from a conjectural class of 6-dimensional superconformal quantum field theories (SQFTs). One may associate such a class to each connected complex simple Lie group $G$. The Higgs branch of this conjectural class would be a 2-dimensional topological quantum field theory (TQFT) $\eta_G:\Cob_2\longrightarrow\MT$, i.e. a symmetric monoidal functor, valued in a category of holomorphic symplectic varieties with Hamiltonian actions. The objects of $\MT$ are complex semisimple groups, and a morphism from $G_1$ to $G_2$ is a holomorphic symplectic variety with a Hamiltonian action of $G_1 \times G_2$. The composition of $M \in \Hom(G_1, G_2)$ and $N \in \Hom(G_2, G_3)$ is the Hamiltonian reduction $(M \times N) \sll{} G_2$. Properties of this TQFT would necessarily include $\eta_G(S^1)=G$ and $\eta_G(\begin{tikzpicture}[
		baseline=-2.5pt,
		every tqft/.append style={
			transform shape, rotate=90, tqft/circle x radius=4pt,
			tqft/circle y radius= 2pt,
			tqft/boundary separation=0.6cm, 
			tqft/view from=incoming,
		}
		]
		\pic[
		tqft/cup,
		name=d,
		every incoming lower boundary component/.style={draw},
		every outgoing lower boundary component/.style={draw},
		every incoming boundary component/.style={draw},
		every outgoing boundary component/.style={draw},
		cobordism edge/.style={draw},
		cobordism height= 1cm,
		];
	\end{tikzpicture})=G\times\mathcal{S}$, where $\mathcal{S}$ is a Kostant slice in the Lie algebra of $G$. Moore and Tachikawa's conjecture is that there exists a 2-dimensional TQFT $\Cob_2\longrightarrow\MT$ satisfying these two conditions. While formulatable as a problem in pure mathematics, a resolution of this conjecture would clearly have implications for string theory. An affirmative answer would be some evidence that the aforementioned SQFTs actually exist. We refer the reader to Tachikawa's ICM paper \cite{tac:18} for further context.
	
	The Moore--Tachikawa conjecture was formulated less than 15 years prior to our completing this manuscript. In this short time, it has witnessed progress and been found to make deep connections to adjacent subjects \cite{ara:18,bal-may:22,bie:23,bra-fin-nak:19,cal:14,cal:15,cro-may:22,dan-gri-mar-zon:24,dan-kir-mar:24,gin-kaz:23,tac:18}. Progress on the conjecture itself is made in unpublished work of Ginzburg--Kazhdan \cite{gin-kaz:23}. These authors first construct the \textit{open Moore--Tachikawa varieties}, thereby proving a relative of the Moore--Tachikawa conjecture. If the affinizations of such varieties were known to be of finite type, then the Moore--Tachikawa conjecture would follow. The work of Braverman--Finkelberg--Nakajima \cite{bra-fin-nak:19} implies that the relevant affine schemes are of finite type for $G$ of Lie type $A$. In particular, the Moore--Tachikawa conjecture holds in type $A$.
	
	Ginzburg and Kazhdan's approach to the Moore--Tachikawa conjecture involves a kind of Hamiltonian reduction by abelian group schemes. This is formalized and generalized in \cite{cro-may:22}, where we introduce Hamiltonian reduction by symplectic groupoids along pre-Poisson subvarieties. We thereby recover the open Moore--Tachikawa varieties constructed by Ginzburg--Kazhdan, as well as their affinizations. In very rough terms, this is achieved by replacing $G$ with its cotangent groupoid $T^*G\tto\g^*$. One might therefore suspect the following: our recovering the Ginzburg--Kazhdan construction is a shadow of a more general ``affinization process", in which the results of \cite{cro-may:22} are used to affinize TQFTs constructed via symplectic groupoids. This is the first of two main themes in our manuscript.
	
	An appealing feature of the Moore--Tachikawa conjecture is its apparent amenability to a wide variety of techniques, including some not used by Ginzburg--Kazhdan or Braverman--Finkelberg--Nakajima. It is in this context that we are led to \textit{shifted symplectic geometry}, as introduced by Pantev--T\"oen--Vaquie--Vezzosi \cite{ptvv:13}. In subsequent work \cite{cal:15}, Calaque suggests that shifted symplectic geometry should facilitate a rigorous construction of the Moore--Tachikawa TQFT. This is largely consistent with a result in \cite{bal-may:22}, where B\u{a}libanu and the second named author develop a Hamiltonian reduction theory for $1$-shifted symplectic groupoids (a.k.a. quasi-symplectic groupoids). One consequence is a multiplicative counterpart of the Ginzburg--Kazhdan construction described above. With this in mind, the second main theme of our manuscript is the relevance of shifted symplectic geometry to the Moore--Tachikawa conjecture.  
	
	\subsection{Overarching principle}
	The overarching principle of this manuscript is that the structure underlying certain TQFTs is encoded in multiplication and identity bisections in abelian symplectic groupoids. We thereby offer an algebraic explanation for the existence of the Moore--Tachikawa TQFT in this enlarged category, as well as that of other TQFTs. These ideas are made somewhat more precise in the next few subsections, and completely precise in the main body of the manuscript.   
	
	\subsection{TQFTs in a $1$-shifted Weinstein symplectic category}
	For the moment, we may work in the smooth or holomorphic categories.
	A first useful idea, suggested by Calaque \cite{cal:15}, is to temporarily replace the Moore--Tachikawa category $\MT$ of holomorphic symplectic varieties with a 1-shifted version $\WS$ of Weinstein symplectic category \cite{wei:82,wei:10,weh-woo:10}; see Section \ref{Section: Shifted}. The objects of $\WS$ are quasi-symplectic groupoids \cite{xu:04,bur-cra-wei-zhu:04}, interpreted as presentations of 1-shifted symplectic stacks.
	We then recall that a 2-dimensional TQFT in a symmetric monoidal category $\mathbf{C}$ is equivalent to a commutative Frobenius object in $\mathbf{C}$.
	We subsequently establish that any abelian symplectic groupoid gives rise to a commutative Frobenius object in $\WS$.
	The multiplication and unit of the aforementioned commutative Frobenius object are essentially given by groupoid multiplication and the identity section, respectively.
	
	 In the interest of associating TQFTs to a wider class of groupoids, we proceed as follows. A quasi-symplectic groupoid is called \textit{abelianizable} if it is Morita equivalent to abelian symplectic groupoid. A notion of Morita transfer then implies that every abelianizable quasi-symplectic groupoid determines a TQFT. This amounts to the following result, the underlying details of which are given in Section \ref{90msjexi}.

	\begin{mtheorem}\label{Theorem: Main 1}
	Every abelianizable quasi-symplectic groupoid $\G$ completes to a commutative Frobenius object in $\WS$ whose product $\G_\mu \in \Hom(\G \times \G, \G)$ and unit $\G_\eta \in \Hom(\star, \G)$ are induced by groupoid multiplication and identity bisection in the abelianization, respectively. It thereby determines a TQFT $\eta_{\G}:\Cob_2\longrightarrow\WS$.
	\end{mtheorem}
	
	In this context, one naturally seeks sufficient conditions for a quasi-symplectic groupoid $\G\tto X$ to be abelianizable. This leads us to introduce the notion of an \textit{admissible global slice}. We define it to be a submanifold $S\s X$ that intersects every $\G$ orbit in $X$ transversely in a singleton, such that the isotropy group $\G_x$ is abelian for all $x\in S$. We show $\G\tto X$ to be abelianizable if there exists an admissible global slice $S\s X$ with the property that the $3$-form on $X$ pulls back to an exact $3$-form on $S$.  
	
	\subsection{Affinizations of TQFTs} We now work exclusively over $\mathbb{C}$. It is useful to specialize Main Theorem \ref{Theorem: Main 1} as follows. Let $G$ be a connected semisimple affine algebraic group with Lie algebra $\g$. Write $T^*G\reg\tto\greg$ for the pullback of $T^*G\tto\g^*$ to the regular locus $\gdreg\s\g^*$. One may verify that $T^*G\reg\tto\greg$ is Morita equivalent to $\mathcal{Z}_{G}\longrightarrow\mathfrak{c}\coloneqq\mathrm{Spec}((\mathrm{Sym}\hspace{2pt}\g)^G)$, the universal centralizer of $G$. This fact implies that $T^*G\reg\tto\greg$ is abelianizable. By means of Main Theorem \ref{Theorem: Main 1}, $T^*G\reg\tto\greg$ turns out to determine the open Moore--Tachikawa TQFT. As mentioned above, Ginzburg and Kazhdan affinize some of the morphisms in the image of this TQFT\cite{gin-kaz:23}. These affinizations are subsequently shown to satisfy the relations of a TQFT, despite not necessarily being of finite type.
	
	Our second main result is that the aforementioned affinization process applies in far greater generality, and that the TQFT relations essentially follow from Main Theorem \ref{Theorem: Main 1}. To this end, we enlarge $\MT$ to what we call the \textit{algebraic Moore--Tachikawa category} $\AMT$; it is very loosely described as a category with affine symplectic groupoids as objects, and morphisms consisting of isomorphism classes of affine Poisson schemes with Hamiltonian actions of product groupoids. Our previously developed machinery of scheme-theoretic coisotropic reduction \cite{cro-may:24} allows one to compose morphisms in $\AMT$.
	We then devise an affinization process which sends 1-shifted Lagrangian correspondences in $\WS$ to Hamiltonian Poisson schemes in $\AMT$.
	
	In light of the above, we introduce the notion of a \textit{Hartogs abelianizable} affine symplectic groupoid $\G\tto X$. By this, we mean that $\G\tto X$ is abelianizable when restricted to an open subset $U \s X$ whose complement has codimension at least two in $X$.
	We then show that the composition of the TQFT associated to $\G|_U$ in Main Theorem \ref{Theorem: Main 1} with the affinization process gives $\G$ the structure of a commutative Frobenius object in $\AMT$.
	A key ingredient in the proof is that $U \s X$ satisfies Hartogs' theorem, enabling us to complete the relevant structures on $\G|_U$ to $\G$.
	In more detail, our second main result is as follows.
	
	\begin{mtheorem}\label{Theorem: Main 2}
	Let $\G \tto X$ be an affine symplectic groupoid that is abelianizable over an open subset $U \s X$ whose complement has codimension at least two in $X$.
	Then $\G$ completes to a commutative Frobenius object in $\AMT$ whose product $\G_\mu \in \Hom(\G \times \G, \G)$ and unit $\G_\eta \in \Hom(\star, \G)$ are the affinizations of $(\G|_U)_\mu$ and $(\G|_U)_\eta$ in Main Theorem \ref{Theorem: Main 1}, respectively.
	In this way, $\G$ determines a TQFT $\eta_{\G}:\Cob_2\longrightarrow\AMT$.
	\end{mtheorem}
	
	In analogy with the discussion following Main Theorem \ref{Theorem: Main 1}, one should seek sufficient conditions for an affine symplectic groupoid to be Hartogs abelianizable. This leads us to introduce the notion of an \textit{admissible Hartogs slice}; it is defined somewhat analogously to an admissible global slice. We show that an affine symplectic groupoid admitting an admissible Hartogs slice is Hartogs abelianizable. 
	
	Let us again consider the cotangent groupoid $T^*G\tto\g^*$ of a connected complex semisimple affine algebraic group $G$. The universal centralizer witnesses $T^*G\tto\g^*$ as Hartogs abelianizable. Main Theorem \ref{Theorem: Main 2} therefore applies and yields the Moore--Tachikawa TQFT. Perhaps surprisingly, Main Theorem \ref{Theorem: Main 2} turns out to yield a strictly larger class of TQFTs. The following are some details in this direction. 
	
	One might seek conditions on a complex affine algebraic group $G$ for $T^*G\tto\g^*$ to be Hartogs abelianizable. This leads us to introduce the notion of a \emph{Moore--Tachikawa group}, i.e. a complex affine algebraic group $G$ with Lie algebra $\g$ that satisfies the following properties:
	\begin{itemize}
	\item the set $\g^*\reg$ of regular elements has a complement of codimension at least two in $\g^*$;
	\item the restriction of the cotangent groupoid $T^*G$ to $\g^*\reg$ is abelianizable.
	This holds, in particular, if there exists an affine slice $S \s \g^*\reg$ for the coadjoint action on the regular locus such that the stabilizer subgroup $G_\xi$ is abelian for all $\xi \in S$.
	\end{itemize}
	Results of Kostant \cite{kos:59,kos:63} imply that every complex reductive group is Moore--Tachikawa. On the other hand, Main Theorem \ref{Theorem: Main 2} implies that every Moore--Tachikawa group induces a TQFT. These considerations motivate one to find examples non-reductive Moore--Tachikawa groups. We provide one such example in Section \ref{Section: Non-reductive}. We also obtain examples of TQFTs arising from certain Slodowy slices, e.g. a Slodowy slice to the minimal nilpotent orbit in $\mathfrak{sl}_n$.

	\subsection{Organization} Each section begins with a summary of its contents. In Section \ref{Section: 1-shifted}, we develop pertinent results and techniques regarding shifted Lagrangians. This leads to Section \ref{Section: Shifted}, where we define and discuss the $1$-shifted Weinstein symplectic ``category" $\WSQ$ and its completion $\WS$. We then prove Main Theorem 1 in Section \ref{90msjexi}. 
	
	Our attention subsequently turns to the matter of affinizing TQFTs. In Section \ref{Section: Affinization}, we use scheme-theoretic coisotropic reduction to define and study the category $\AMT$. Main Theorem \ref{Theorem: Main 2} is also proved in Section \ref{Section: Affinization}. Section \ref{Section: Special case} is then devoted to the implications of Main Theorem \ref{Theorem: Main 2} for constructing the Moore--Tachikawa TQFT. In Sections \ref{Section: Examples} and \ref{Section: Non-reductive}, we use Main Theorem \ref{Theorem: Main 2} to produce TQFTs different from those of Moore--Tachikawa. Section \ref{Section: Examples} provides a TQFT for an affine symplectic groupoid integrating a Slodowy slice to the minimal nilpotent orbit in $\mathfrak{sl}_n$. Section \ref{Section: Non-reductive} then provides an example of a non-reductive Moore--Tachikawa group.

	\subsection*{Acknowledgements}
	We thank Ana B\u{a}libanu, Damien Calaque, Zo\"ik Dubois, and Tom Gannon for highly useful discussions. The authors also thank the referee for insightful, interesting, and helpful comments that have led to an improved manuscript. P.C. acknowledges partial support from the National Science Foundation Grant DMS-2454103 and Simons Foundation Grant MPS-TSM-00002292. M.M. was partially supported by the NSERC Discovery Grant RGPIN-2023-04587.

%
%
%
	\section{1-shifted Lagrangian correspondences}\label{Section: 1-shifted}
	
	In this section, we review background material on quasi-symplectic groupoids, shifted Lagrangians, and Morita equivalence in a differential-geometric context, following \cite{ptvv:13,cal:14,get:14,cal:21,cue-zhu:23,may:23}. We also provide proofs of certain basic, technical facts which we have not been able to find in the literature. 
	
	We work in the smooth category throughout this section. If one makes the obvious adjustments to terminology (e.g. replacing \textit{real} and \textit{smooth} with \textit{complex} and \textit{holomorphic}, respectively), all results hold in the holomorphic category as well.
	
	\subsection{Notation and conventions on Lie groupoids}\label{xizv8xjy}
	 Given a Lie groupoid $\G$, the manifolds of arrows and objects are denoted $\G\arr$ and $\G\obj$, respectively. One has source and target maps $(\sss, \ttt) : \G\arr \tto \G\obj$, multiplication $\mmm : \G\arr \fp{\sss}{\ttt} \G\arr \longrightarrow \G\arr$, inversion $\iii : \G\arr \longrightarrow \G\arr$, and the identity section $\uuu : \G\obj \longrightarrow \G\arr$. The Lie algebroid of $\G$ is $A_\G \coloneqq \ker \sss_*$, together with $\rho_\G \coloneqq \ttt_*\big\vert_{A_\G}$ as its anchor map.
	
	Consider a Lie groupoid $\G$ and smooth map $\mu:N\longrightarrow\G\obj$. The \defn{pullback} of $\G$ by $\mu$ is denoted $\mu^*\G$; it is a groupoid with arrows $(\mu^*\G)\arr = N \fp{\mu}{\sss} \G \fp{\ttt}{\mu} N$ and objects $(\mu^*\G)\obj = N$. A sufficient condition for $\mu^*\G$ to be a Lie groupoid is for $\ttt \circ \pr_{\G\arr} : N \fp{\mu}{\sss} \G\arr \to \G\obj$ to be a surjective submersion, where $\pr_{\G\arr}:  N \fp{\mu}{\sss} \G\arr\longrightarrow\G\arr$ is the canonical map.

	The \defn{restriction} of a groupoid $\G$ to a subset $S \s \G\obj$ is the groupoid $\G|_S \coloneqq \sss^{-1}(S) \cap \ttt^{-1}(S)$, i.e.\ the pullback of $\G$ by the inclusion $S \too \G\obj$.

	For an action of a Lie groupoid $\G$ on a manifold $M$, the \defn{action groupoid} is denoted $\G \ltimes M$.
	It has source map $\sss(g, p) = p$ and target map $\ttt(g, p) = g \cdot p$.
	
	\subsection{Quasi-symplectic groupoids}
	Quasi-symplectic groupoids feature prominently in this manuscript. For the sake of completeness and self-containment, we recall their definition below.
	
	\begin{definition}[Xu \cite{xu:04}, Bursztyn--Crainic--Weinstein--Zhu \cite{bur-cra-wei-zhu:04}]
		A \defn{quasi-symplectic groupoid} is a triple $(\G, \omega, \phi)$ of a Lie groupoid $\G$, 2-form $\omega$ on $\G\arr$, and a 3-form $\phi$ on $\G\obj$, such that $\mathrm{d}\omega = \sss^*\phi - \ttt^*\phi$, $\dim \G\arr = 2 \dim \G\obj$, and $\ker \omega_x \cap \ker \sss_* \cap \ker \ttt_* = 0$ for all $x \in \G\obj$.
		The \defn{IM-form} of $(\G, \omega, \phi)$ is the map $$\sigma_\omega : A_{\G} \longrightarrow T^*\G\obj,\quad a\mapsto\uuu^*(i_a\omega).$$
We write $\G^-$ for the quasi-symplectic groupoid $(\G, -\omega, -\phi)$.
A \defn{symplectic groupoid} is a quasi-symplectic groupoid whose 2-form $\omega$ is non-degenerate and whose 3-form $\phi$ is trivial.
	\end{definition}

The central purpose of quasi-symplectic groupoids is that they are presentations of 1-shifted symplectic stacks \cite{get:14,cal:21,ptvv:13}.

	\subsection{$1$-shifted Lagrangians}\label{e3lhuxkc}
	Let us recall the following crucial definition.
	
	\begin{definition}[Pantev--To\"en--Vaqui\'e--Vezzosi \cite{ptvv:13}] \label{yhan52mv}
		A \defn{1-shifted Lagrangian} on a quasi-symplectic groupoid $(\G, \omega, \phi)$ is a Lie groupoid morphism $\mu : \L \longrightarrow \G$, together with a 2-form $\gamma$ on $\L\obj$, such that the following conditions are satisfied:
		\begin{enumerate}[label={(\roman*)}]
			\item \label{n6alyrle}
			$\mu^*\omega = \ttt^*\gamma - \sss^*\gamma$;
			\item \label{kopgj8mq}
			$\mu^*\phi = -\mathrm{d}\gamma$;
			\item \label{dhs45mwu}
			the map
			\begin{align*}
				A_\L &\too \{(v, a) \in T\L\obj \oplus \mu^*A_\G : \mu_*(v) = \rho_\G(a) \text{ and } i_v\gamma = \mu^* \sigma_\omega(a)\} \\
				b &\mtoo (\rho_\L(b), \mu_*(a))
			\end{align*}
			is an isomorphism.
		\end{enumerate}
		See \cite[Lemma 3.3]{may:23} for the equivalence between this definition and the original one in \cite{ptvv:13}.
	\end{definition}

\begin{example}\label{cb2gi3le}
The notion of 1-shifted Lagrangians is closely related to that of Hamiltonian $\G$-spaces, as defined by Xu \cite{xu:04}. In more detail, let a quasi-symplectic groupoid $\G$ act on a manifold $M$. A 1-shifted Lagrangian structure on the projection $\G \ltimes M \to \G$ is precisely a 2-form on $M$ for which $M$ is a Hamiltonian $\G$-space.
\end{example}

	Given two quasi-symplectic groupoids $\G_1$ and $\G_2$, a \defn{1-shifted Lagrangian correspondence} from $\G_1$ to $\G_2$ is a 1-shifted Lagrangian $\L \longrightarrow \G_1 \times \G_2^-$.
	We usually denote a 1-shifted Lagrangian correspondence by a span
	\[
	\begin{tikzcd}[row sep={1.5em,between origins},column sep={3em,between origins}]
		& \L \arrow{dl} \arrow{dr} & \\
		\G_1 & & \G_2.
	\end{tikzcd}
	\]
	Two 1-shifted Lagrangian correspondences
	\begin{equation}
		\begin{tikzcd}[row sep={1.5em,between origins},column sep={3em,between origins}]
			& \L_1 \arrow{dl} \arrow{dr} & & \L_2 \arrow{dl} \arrow{dr} & \\
			\G_1 & & \G_2 & & \G_3
		\end{tikzcd}
	\end{equation}
	are called \defn{transverse} if their \defn{homotopy fibre product} $\L_1 \htimes_{\G_2} \L_2$ (see, e.g.\ \cite[Subsection 5.3]{moe-mrc:03} or \cite[Subsection 4.2]{may:23}) is transverse, i.e.\ the two maps
	\[
	\begin{tikzcd}[row sep={1em},column sep={2em}]
		\L_1\obj \times \L_2\obj \arrow{dr} & & \G_2\arr \arrow[swap]{dl}{(\sss, \ttt)} \\
		& \G_2\obj \times \G_2\obj &
	\end{tikzcd}
	\]
	are transverse; see Subsection \ref{s0ja09pn} for the homotopical meaning of this fibre product and a justification of the transversality condition.
	In this case,
	\begin{equation}\label{0dmgaq1o}
		\begin{tikzcd}[row sep={2em,between origins},column sep={4em,between origins}]
			& \L_1 \htimes_{\G_2} \L_2 \arrow{dl} \arrow{dr} & \\
			\G_1 & & \G_3.
		\end{tikzcd}
	\end{equation}
	is a 1-shifted Lagrangian correspondence \cite[Theorem 4.10]{may:23}; it is called the \defn{composition} of $\L_1$ and $\L_2$, and denoted $\L_2 \circ \L_1$.
	More precisely, if $\gamma_i$ is the 1-shifted Lagrangian structure on $\L_i$ for $i = 1, 2$, then the 1-shifted Lagrangian structure on \eqref{0dmgaq1o} is given by $\pr_{\L_1}^*\gamma_1 + \pr_{\L_2}^*\gamma_2 - \pr_{\G_2\arr}^*\omega_2$, where $\pr_{\L_i} : \L_1 \htimes_{\G_2} \L_2 \too \L_i$ and
	\[
	\pr_{\G_2\arr} : (\L_1 \htimes_{\G_2} \L_2)\obj = \L_1\obj \fp{\mu_1}{\sss} \G_2\arr \fp{\ttt}{\mu_2} \L_2\obj \too \G_2\arr
	\]
	are the natural projections.

	\subsection{Morita equivalences}\label{3y574s4t}
	A morphism of Lie groupoids $\varphi : \H \longrightarrow \G$ is called a \defn{Morita morphism} \cite{beh-xu:03,beh:04,lau-sti-xu:09,beh-xu:11} (a.k.a.\ \emph{surjective equivalence} \cite{hoy:13} or \emph{hypercover} \cite{zhu:09,cue-zhu:23}) if the underlying map on objects $\varphi\obj : \H\obj \longrightarrow \G\obj$ is a surjective submersion, and the induced morphism $\H \longrightarrow (\varphi\obj)^*\G$ is a Lie groupoid isomorphism.
	A \defn{Morita equivalence} between two Lie groupoids $\G_1$ and $\G_2$ consists of a Lie groupoid $\H$ and Morita morphisms $\G_1 \longleftarrow \H \longrightarrow \G_2$.
	Morita equivalence is an equivalence relation on the set of Lie groupoids.
	
	A weaker notion is that of an \defn{essential equivalence}, consisting of a Lie groupoid morphism $\varphi : \H \too \G$ such that $\ttt \circ \pr_{\G\arr} : \H\obj \fp{\varphi}{\sss} \G\arr \to \G\obj$ is a surjective submersion and $\H \to (\varphi\obj)^*\G$ is a Lie groupoid isomorphism.
	Viewed as Lie $\infty$-groupoids, these essential equivalences are precisely the weak equivalences in the incomplete category of fibrant objects of Rogers--Zhu \cite{rog-zhu:20}.
    Two Lie groupoids $\G_1$ and $\G_2$ are Morita equivalent if and only if there is a Lie groupoid $\H$ together with essential equivalences $\G_1 \longleftarrow \H \longrightarrow \G_2$ (see e.g.\ \cite[Remark 5.2]{may:23}).
	Hence, Morita morphisms and essential equivalences induce the same equivalence relation on the set of Lie groupoids and can be used interchangeably.
	We use mainly Morita morphisms in this paper, following \cite{may:23}, but both approaches are equivalent.
    
	The notion of Morita equivalence gives rise to the notion of a \defn{symplectic Morita equivalence} between two quasi-symplectic groupoids $(\G_1, \omega_1, \phi_1)$ and $(\G_2, \omega_2, \phi_2)$, i.e.\ a Morita equivalence $\G_1 \longleftarrow \H \longrightarrow \G_2$ that is also a 1-shifted Lagrangian correspondence.
	Two quasi-symplectic groupoids are symplectically Morita equivalent if and only if they present isomorphic 1-shifted symplectic stacks \cite{get:14,cal:21,cue-zhu:23,may:23}.
    Thus symplectic Morita equivalence is the change-of-presentation relation under which the shifted symplectic structure descends to the underlying differentiable stack; see \cite{cue-zhu:23} for the higher-groupoid definition and examples, and \cite{weiers:25} for Morita invariance of shifted symplectic structures on Lie $n$-groupoids.
	By \cite[Lemma 6.3]{may:23}, the non-degeneracy condition \ref{dhs45mwu} in Definition \ref{yhan52mv} is automatic: a Morita equivalence $\G_1 \overset{\varphi_1}{\longleftarrow} \H \overset{\varphi_2}{\longrightarrow} \G_2$ is symplectic if and only if there is a 2-form $\gamma$ on $\H\obj$ such that $\varphi_1^*\omega_1 - \varphi_2^*\omega_2 = \ttt^*\gamma - \sss^*\gamma$ and $\varphi_1^*\phi_1 - \varphi_2^*\phi_2 = -\mathrm{d}\gamma$.
	The following example therefore follows automatically.

	\begin{example}\label{c00ufvui}
	For every quasi-symplectic groupoid $\G$, the identity maps $\G \longleftarrow \G \longrightarrow \G$ and trivial 2-form constitute a symplectic Morita equivalence.
	\end{example}

	\begin{remark}\label{aoe5a1sn}
	An alternative but equivalent \cite{cue-zhu:23} approach to symplectic Morita equivalence is via Hamiltonian bibundles \cite{xu:04,ale-mei:22}, i.e.\ two quasi-symplectic groupoids $\G_1$ and $\G_2$ are symplectically Morita equivalent if and only if there is manifold $M$ endowed with a 2-form and commuting Hamiltonian actions of $\G_1$ and $\G_2^-$ such that $M \to \G_2\obj$ is a principal $\G_1$-bundle and $M \to \G_1\obj$ is a principal $\G_2$-bundle.
	Such a bibundle gives rise to a symplectic Morita equivalence via the action groupoid $(\G_1 \times \G_2) \ltimes M$ together with the two projections to $\G_1$ and $\G_2$.
	\end{remark}
	
	A \defn{Lagrangian Morita equivalence} between two 1-shifted Lagrangians $(\L_1, \gamma_1) \longrightarrow (\G_1, \omega_1, \phi_1)$ and $(\L_2, \gamma_2) \longrightarrow (\G_2, \omega_2, \phi_2)$ is a 2-commutative diagram
	\begin{equation}
		\begin{tikzcd}[column sep={4em,between origins},row sep={2em,between origins}]
			& \M \arrow[swap]{dl}{\psi_1} \arrow{dd}{\nu} \arrow{dr}{\psi_2} & \\
			\L_1 \arrow[swap]{dd}{\mu_1} \arrow[Rightarrow,shorten=6pt,swap]{dr}{\theta_1} &  & \L_2 \arrow{dd}{\mu_2} \arrow[Rightarrow,shorten=6pt]{dl}{\theta_2} \\
			& \H \arrow{dl}{\varphi_1} \arrow[swap]{dr}{\varphi_2} & \\
			\G_1 & & \G_2,
		\end{tikzcd}
	\end{equation}
	such that $(\varphi_1, \varphi_2)$ is a symplectic Morita equivalence with respect to some 2-form $\delta$ on $\H\obj$, $(\psi_1, \psi_2)$ is a Morita equivalence of Lie groupoids, and
	\[
	\psi_1^*\gamma_1 - \psi_2^*\gamma_2 = \nu^*\delta - \theta_1^*\omega_1 + \theta_2^*\omega_2.
	\]
	In this case, we say that $\L_1$ and $\L_2$ are \defn{Lagrangially Morita equivalent}.
	Thus a Lagrangian Morita equivalence refines a symplectic Morita equivalence by compatible Morita data for the Lagrangians themselves. Equivalently, it identifies the 1-shifted symplectic stacks presented by $\G_1$ and $\G_2$ and, under that same identification, the 1-shifted Lagrangians presented by $\L_1\to\G_1$ and $\L_2\to\G_2$; see \cite[Section 10]{may:23}.
	In the special case where $(\G_1, \omega_1, \phi_1) = (\G_2, \omega_2, \phi_2)$, we call this a \defn{weak equivalence} of 1-shifted Lagrangians. This amounts to declaring two $1$-shifted Lagrangians $(\L_1, \gamma_1) \longrightarrow (\G, \omega, \phi)$ and $(\L_2, \gamma_2) \longrightarrow (\G, \omega, \phi)$ to be weakly equivalent if there is a 2-commutative diagram
	\begin{equation}\label{msge1bqw}
		\begin{tikzcd}[row sep={3em,between origins},column sep={3em,between origins}]
			& \M \arrow[swap]{dl}{\psi_1} \arrow{dr}{\psi_2} & \\
			\L_1 \arrow[swap]{dr}{\mu_1} \arrow[shorten=14pt,Rightarrow]{rr}{\theta} & & \L_2 \arrow{dl}{\mu_2} \\
			& \G &
		\end{tikzcd}
	\end{equation}
	such that $(\psi_1, \psi_2)$ is a Morita equivalence of Lie groupoids and $\psi_2^*\gamma_2 - \psi_1^*\gamma_1 = \theta^*\omega$.

\begin{example}\label{gesm9jcm}
Recall that the action of $\G \times \G^-$ on $\G$ by left and right multiplications is Hamiltonian in the sense of Xu \cite{xu:04}.
It follows from Example \ref{cb2gi3le} that $(\G \times \G) \ltimes \G$ is a symplectic Morita equivalence from $\G$ to $\G$.
Under the equivalence between Hamiltonian bibundles and symplectic Morita equivalences (Remark \ref{aoe5a1sn}), $\G \longleftarrow (\G \times \G) \ltimes \G \longrightarrow \G$ is weakly equivalent to the 1-shifted Lagrangian correspondence $\G \longleftarrow \G \longrightarrow \G$ of Example \ref{c00ufvui}.
An explicit essential equivalence can be obtained by the morphism $\G \too (\G \times \G) \ltimes \G$, $g \mtoo ((g, g), \uuu_{\sss(g)})$.
\end{example}

If $\G_1$ and $\G_2$ are symplectically Morita equivalent quasi-symplectic groupoids, then for every 1-shifted Lagrangian $\L_1 \to \G_1$ there is a 1-shifted Lagrangian $\L_2 \to \G_2$ together with a Lagrangian Morita equivalence between them \cite[Section 7]{may:23}.
Moreover, $\L_2 \to \G_2$ is unique up to weak equivalences.
We call this process of transferring 1-shifted Lagrangians from $\G_1$ to $\G_2$ \defn{Morita transfer}.
	
The next three results are important for the construction of the 1-shifted Weinstein symplectic category in the next section.

	\begin{lemma}\label{32j00jes}
		Consider a 2-commutative diagram of Lie groupoid morphisms
		\[
		\begin{tikzcd}[row sep={1.5em,between origins},column sep={3em,between origins}]
			\M_1 \arrow[swap]{dd} \arrow{dr} & & \M_2 \arrow[swap]{dl} \arrow{dd} \\
			& \H \arrow{dd}  &  \\
			\L_1 \arrow[swap]{dr} & & \L_2 \arrow{dl} \\
			& \G &
		\end{tikzcd},
		\]
		where the vertical arrows are Morita morphisms.
		Then $\L_1 \htimes_\G \L_2$ is transverse if and only if $\M_1 \htimes_\H \M_2$ is transverse.
		In this case, the induced map $\M_1 \htimes_\H \M_2 \longrightarrow \L_1 \htimes_\G \L_2$ is a Morita morphism.
	\end{lemma}
	
	\begin{proof}
		Introduce the following labels:
		\[
		\begin{tikzcd}[row sep={2em,between origins},column sep={4em,between origins}]
			\M_1 \arrow[swap]{dd}{\psi_1} \arrow{dr}{\nu_1} & & \M_2 \arrow[swap]{dl}{\nu_2} \arrow{dd}{\psi_2} \\
			& \H \arrow{dd}{\varphi} \arrow[Rightarrow, shorten=8pt,swap]{dl}{\theta_1} \arrow[Rightarrow, shorten=8pt]{dr}{\theta_2} &  \\
			\L_1 \arrow[swap]{dr}{\mu_1} & & \L_2 \arrow{dl}{\mu_2} \\
			& \G &
		\end{tikzcd},
		\]
		where $\theta_i$ is a natural transformation from $\varphi \nu_i$ to $\mu_i \psi_i$, i.e.\
		\[
		\sss \theta_i = \varphi \nu_i, \quad \ttt \theta_i = \mu_i \psi_i, \quad\text{and}\quad \theta_i(\ttt(g_i)) \cdot \varphi(\nu_i(g_i)) = \mu_i(\psi_i(g_i)) \cdot \theta_i(\sss(g_i))
		\]
		for all $g_i \in \M_i$ and $i=1,2$.
		Choose an Ehresmann connection $\tau$ on $\G$ \cite[Definition 2.8]{aba-cra:13}; this is a right splitting of the short exact sequence
		\begin{equation}\label{pu32toow}
			\begin{tikzcd}[column sep=3em]
				0 \arrow{r} & \ttt^* A_\G \arrow{r}{R} & T\G\arr \arrow{r}{\sss_*} & \sss^*T\G\obj \arrow{r} & 0,
			\end{tikzcd}
		\end{equation}
		where $R$ is right translation.
		Let $\Ad$ be the corresponding adjoint representation up to homotopy \cite[Definition 2.11]{aba-cra:13}; see also \cite[Subsection 7.2]{may:23} for an overview in the same notation as this proof. We have the quasi-action of $\G$ on $T\G\obj$ and $A_\G$ given by 
		\begin{align*}
			\Ad_g &: T_{\sss(g)}\G\obj \too T_{\ttt(g)}\G\obj, \quad \Ad_g(v) = \ttt_*(\tau_g(v)) \\
			\Ad_g &: (A_\G)_{\sss(g)} \too (A_\G)_{\ttt(g)}, \quad \Ad_g(a) = \check{\tau}(a^L_g)
		\end{align*}
		for all $g \in \G$, where $\check{\tau}$ is the left splitting of \eqref{pu32toow} corresponding to $\tau$.
		As in \cite[Subsection 7.2]{may:23}, we denote the basic curvature of $\tau$ by $K \in \Gamma(\G^2; \Hom(\sss^*T\G\obj, \ttt^*A_\G))$.
		We also set
		\[
		\dot{\theta}_i \coloneqq \check{\tau} \circ \theta_{i*} : T\M_i\obj \too \psi_2^*\mu_2^*A_\G,
		\]
		for $i=1,2$, so that the $\dot{\theta}_i$ provide chain homotopies for the maps of chain complexes from $A_{\M_i} \longrightarrow T\M_i\obj$ to $A_{\G} \longrightarrow T\G\obj$ \cite[Proposition 7.4]{may:23}.
		
		Suppose that $\M_1 \htimes_\H \M_2$ is transverse.
		Let $(x_1, g, x_2) \in (\L_1 \htimes_\G \L_2)\obj$, i.e.\ $(\mu_1(x_1), \mu_2(x_2)) = (\sss(g), \ttt(g))$.
		Take $(u_1, u_2) \in T_{\sss(g)}\G\obj \times T_{\ttt(g)}\G\obj$.
		Let $y_i \in \M_i\obj$ be such that $\psi_i(y_i) = x_i$.
		Then $\sss(g) = \mu_1(x_1) = \mu_1(\psi_1(y_1)) = \ttt(\theta_1(y_1))$ and $\ttt(g) = \mu_2(x_2) = \mu_2(\psi_2(y_2)) = \ttt(\theta_2(y_2))$.
		Moreover, $\sss(\theta_1(y_1)) = \varphi(\nu_1(y_1))$ and $\sss(\theta_2(y_2)) = \varphi(\nu_2(y_2))$.
		It follows that
		\[
		(\nu_1(y_1), \theta_2(y_2)^{-1} g \theta_1(y_1), \nu_2(y_2))
		\in
		(\varphi\obj)^*\G.
		\]
		Since $\H \cong (\varphi\obj)^*\G$, there is a unique $h \in \H$ such that $\sss(h) = \nu_1(y_1)$, $\ttt(h) = \nu_2(y_2)$, and $\varphi(h) = \theta_2(y_2)^{-1} g \theta_1(y_1)$.
		Noting that $u_1 \in T_{\sss(g)}\G\obj = T_{\ttt(\theta_1(y_1))}\G\obj$, there exists a unique $v_1 \in T_{\sss(\theta_1(y_1))}\G\obj = T_{\sss(\varphi(h))} \G\obj$ such that $\Ad_{\theta_1(y_1)}(v_1) = u_1$.
		A similar argument reveals the existence of a unique $v_2 \in T_{\sss(\theta_2(y_2))}\G\obj = T_{\ttt(\varphi(h))}\G\obj$ such that $\Ad_{\theta_2(y_2)}(v_2) = u_2$.
		We also have $\varphi(\nu_i(y_i)) = \sss(\theta_i(y_i))$, and know $\varphi$ to be a submersion. One therefore has $w_i \in T_{\nu_i(y_i)}\H\obj$ such that $\varphi_*(w_i) = v_i$.
		At the same time, $(y_1, h, y_2) \in (\M_1 \htimes_\H \M_2)\obj$ and $(w_1, w_2) \in T_{\sss(h)}\H\obj \times T_{\ttt(h)}\H\obj$.
		Since $\M_1 \htimes_\H \M_2$ is transverse, there exist $\delta_i \in T_{y_i}\M_i$ and $\gamma \in T_h\H\arr$ such that
		\[
		(w_1, w_2) = (\nu_{1*}(\delta_1), \nu_{2*}(\delta_2)) + (\sss_*(\gamma), \ttt_*(\gamma)).
		\]
		Let $\tilde{g} \coloneqq \theta_2(y_2)^{-1} g \theta_1(y_1)$.
		Using the fact that $T_{\tilde{g}}\G\arr = R_{\tilde{g}}((A_\G)_{\ttt(\tilde{g})}) \oplus \tau_{\tilde{g}}(T_{\sss(\tilde{g})}\G\obj)$, we can write 
		\begin{equation}\label{2wcbya94}
			\varphi_* \gamma = a^R_{\tilde{g}} + \tau_{\tilde{g}}(b)
		\end{equation}
		for $a \in (A_\G)_{\sss\theta_2(y_2)}$ and $b \in T_{\sss\theta_1(y_1)}\G\obj$.
		Applying $\sss_*$ to both sides of \eqref{2wcbya94} yields $\sss_* \varphi_* (\gamma) = b$.
		It follows that
		\begin{equation}\label{js1iulld}
			\varphi_* (\gamma) = a^R_{\tilde{g}} + \tau_{\tilde{g}} \sss_* \varphi_*(\gamma).
		\end{equation}
		Furthermore, applying $\ttt_*$ to both sides of \eqref{js1iulld} gives
		\begin{equation}\label{l35gw3ng}
			\rho(a) = \ttt_* \varphi_* (\gamma) - \Ad_{\tilde{g}} \sss_* \varphi_*(\gamma).
		\end{equation}

		We claim that
		\[
		(u_1, u_2) = (\mu_{1*} (\tilde{\delta}_1) + \sss_* (\tilde{\gamma}), \mu_{2*} (\tilde{\delta}_2) + \ttt_* (\tilde{\gamma})),
		\]
		where
		\begin{align*}
			\tilde{\gamma} &\coloneqq \tau_g \Ad_{\theta_1(y_1)} \sss \varphi_*(\gamma) + (\dot{\theta}_1(\delta_1))^L_g + (\Ad_{\theta_2(y_2)}(a) -\dot{\theta}_2 (\delta_2) - K(g, \theta_1(y_1))(\sss \varphi_*(\gamma)) + K(\theta_2(y_2), \tilde{g})(\sss \varphi_* (\gamma)))^R_g \\
			\tilde{\delta}_1 &\coloneqq \psi_{1*}(\delta_1) \\
			\tilde{\delta}_2 &\coloneqq \psi_{2*}(\delta_2).
		\end{align*}
		To this end, \cite[Proposition 7.4]{may:23} implies that
		\begin{align*}
			\mu_{1*}(\tilde{\delta}_1) + \sss_*(\tilde{\gamma})
			&=
			\mu_1\psi_{1*}(\delta_1) + \Ad_{\theta_1(y_1)} \sss \varphi_*(\gamma) - \rho\dot{\theta}_1(\delta_1) \\
			&= \Ad_{\theta_1(y_1)}\varphi_*\nu_{1*}(\delta_1) + \Ad_{\theta_1(y_1)} \sss \varphi_*(\gamma) \\
			&= \Ad_{\theta_1(y_1)} \varphi_*(w_1) \\
			&= u_1.
		\end{align*}
		At the same time, \cite[Eq.\ (7.9) and Proposition 7.4]{may:23}, \eqref{l35gw3ng}, and the equivariance of $\rho$ tell us that
		\begin{align*}
			&\mu_{2*}(\tilde{\delta}_2) + \ttt_*(\tilde{\gamma}) \\
			&=
			\mu_2 \psi_{2*} (\delta_2) + \Ad_g \Ad_{\theta_1(y_1)} \sss_* \varphi_*(\gamma) + \rho(\Ad_{\theta_2(y_2)}(a) - \dot{\theta}_2 (\delta_2) - K(g, \theta_1(y_1))(\sss_* \varphi_* (\gamma)) + K(\theta_2(y_2), \tilde{g})(\sss_* \varphi_* (\gamma))) \\
			&= \Ad_{\theta_2(y_2)} \varphi_* \nu_{2*} (\delta_2) + \Ad_g \Ad_{\theta_1(y_1)} \sss_* \varphi_*(\gamma) + \Ad_{\theta_2(y_2)} (\ttt_*\varphi_*(\gamma) - \Ad_{\tilde{g}} \sss_* \varphi_*(\gamma)) \\
			&\qquad - \rho(K(g, \theta_1(y_1))(\sss_* \varphi_* (\gamma))) + \rho(K(\theta_2(y_2), \tilde{g})(\sss_* \varphi_* (\gamma))) \\
			&= 
			\Ad_{\theta_2(y_2)}\varphi_*\nu_{2*}(\delta_2) + \Ad_{\theta_2(y_2)} \ttt_* \varphi_*(\gamma) \\
			&= u_2.
		\end{align*}
		It follows that $\L_1 \htimes_\G \L_2$ is transverse.
		
		Suppose now that $\L_1 \htimes_\G \L_2$ is transverse.
		Let $(y_1, h, y_2) \in (\M_1 \htimes_\H \M_2)\obj$, i.e.\ $(\nu_1(y_1), \nu_2(y_2)) = (\sss(h), \ttt(h))$.
		Let $(v_1, v_2) \in T_{\sss(h)}\H\obj \times T_{\ttt(h)}\H\obj$.
		Note that for all $(u_1, u_2) \in \ker(\varphi_*)_{\sss(h)} \times \ker(\varphi_*)_{\ttt(h)}$, we have $(u_1, u_2) = (\sss_*(w), \ttt_*(w))$, where $w = (u_1, 0, u_2) \in T_h\H\arr \cong T_{\sss(h)}\H\obj \times_{T\G\obj} T_{\varphi(h)}\G\arr \times_{T\G\obj} T_{\ttt(h)}\H\obj$.
		It therefore suffices to show that
		\[
		\varphi_* (v_1) = \varphi_*(\nu_{1*}(u_1) + \sss_*(w))
		\quad\text{and}\quad
		\varphi_* (v_2) = \varphi_*(\nu_{2*}(u_2) + \ttt_*(w))
		\]
		for some $u_i \in T_{y_i}\M_i\obj$ and $w \in T_h\H\arr$.
		In other words, it suffices to show that
		\[
		\varphi_*(v_1 - \nu_{1*}(u_1)) = \sss_* \gamma
		\quad\text{and}\quad
		\varphi_*(v_2 - \nu_{2*}(u_2)) = \ttt_* \gamma,
		\]
		for some $u_i \in T_{y_i}\M_i\obj$ and $\gamma \in T_{\varphi(h)}\G\arr$.
		
		Note that $\ttt(\varphi(h)) = \varphi(\nu_2(y_2)) = \sss(\theta_2(y_2))$ and $\sss(\varphi(h)) = \varphi(\nu_1(y_1)) = \sss(\theta_1(y_1))$, so that $g \coloneqq \theta_2(y_2) \cdot \varphi(h) \cdot \theta_1(y_1)^{-1} \in \G\arr$ is well-defined.
		Let $x_1 \coloneqq \psi_1(y_1)$ and $x_2 \coloneqq \psi_2(y_2)$, noting that $(\mu_1(x_1), \mu_2(x_2)) = (\sss(g), \ttt(g))$.
		We then have $(\varphi_*(v_1), \varphi_*(v_2)) \in T_{\sss\theta_1(y_1)}\G\obj \times T_{\sss\theta_2(y_2)}\G\obj$. It follows that $(\Ad_{\theta_1(y_1)}\varphi_*(v_1), \Ad_{\theta_2(y_2)}\varphi_*(v_2)) \in T_{\sss(g)}\G\obj \times T_{\ttt(g)}\G\obj$.
		Since $\L_1 \htimes_\G \L_2$ is transverse, there exist $\tilde{u}_1 \in T_{x_1}\L_1\obj$, $\tilde{u}_2 \in T_{x_2}\L_2\obj$, and $w \in T_g\G\arr$ such that
		\[
		(\Ad_{\theta_1(y_1)}\varphi_*(v_1), \Ad_{\theta_2(y_2)}\varphi_*(v_2))
		=
		(\mu_{1*}(\tilde{u}_1) + \sss_*(w), \mu_{2*}(\tilde{u}_2) + \ttt_*(w)).
		\]
		Choose $u_i \in T_{y_i}\M_i\obj$ such that $\psi_{i*}(u_i) = \tilde{u}_i$.
		As in the first part, we can write
		\[
		w = a^R_g + \tau_g \sss_* (w),
		\]
		where $a \in (A_\G)_{\ttt(g)}$ and $\rho(a) = \ttt_*(w) - \Ad_g \sss_* (w)$.
		Consider the element of $T_{\varphi(h)}\G\arr$ defined by
		\begin{align*}
			\gamma &\coloneqq \tau_{\varphi(h)}\Ad_{\theta_1(y_1)^{-1}}\sss_*(w) \\
			&\qquad + (\Ad_{\theta_2(y_2)^{-1}}(a + \dot{\theta}_2 (u_2)) + K(\theta_2(y_2)^{-1}, g)(\sss_* (w)) - K(\varphi(h), \theta_1(y_1)^{-1})(\sss_*(w)))^R_{\varphi(h)}\\
			&\qquad -K(\theta_2(y_2)^{-1}, \theta_2(y_2)(\varphi_*(v_2 - \nu_{2*}(u_2))))^R_{\varphi(h)} \\
			&\qquad + (K(\theta_1(y_1)^{-1}, \theta_1(y_1))(\varphi_*(v_1 - \nu_{1*}(u_1))) - \Ad_{\theta_1(y_1)^{-1}}\dot{\theta}_1(u_1))^L_{\varphi(h)}.
		\end{align*}
		We have
		\begin{align*}
			\sss_* \gamma &= \Ad_{\theta_1(y_1)^{-1}} \sss_*(w) - \rho(K(\theta_1(y_1)^{-1}, \theta_1(y_1))(\varphi_*(v_1 - \nu_{1*}(u_1)))) + \rho(\Ad_{\theta_1(y_1)^{-1}}\dot{\theta}_1(u_1)) \\
			&= \Ad_{\theta_1(y_1)^{-1}} \sss_* (w) - \Ad_{\theta_1(y_1)^{-1}} \Ad_{\theta_1(y_1)} \varphi_*(v_1 - \nu_{1*}(u_1)) + \varphi_*(v_1 - \nu_{1*}(u_1)) + \Ad_{\theta_1(y_1)^{-1}} \rho \dot{\theta}_1(u_1) \\
			&= \Ad_{\theta_1(y_1)^{-1}} \sss_* (w) - \Ad_{\theta_1(y_1)^{-1}} \Ad_{\theta_1(y_1)} \varphi_*(v_1 - \nu_{1*}(u_1)) + \varphi_*(v_1 - \nu_{1*}u_1) + \Ad_{\theta_1(y_1)^{-1}}(\mu_{1*}\psi_{1*}u_1 - \Ad_{\theta_1(y_1)} \varphi_*\nu_{1*}u_1) \\
			&= \Ad_{\theta_1(y_1)^{-1}} \Ad_{\theta_1(y_1)}\varphi_*(v_1) - \Ad_{\theta_1(y_1)^{-1}} \Ad_{\theta_1(y_1)} \varphi_*(v_1 - \nu_{1*}(u_1)) + \varphi_*(v_1 - \nu_{1*}(u_1)) - \Ad_{\theta_1(y_1)^{-1}}\Ad_{\theta_1(y_1)} \varphi_*\nu_{1*}(u_1) \\
			&= \varphi_*(v_1 - \nu_{1*}(u_1)) \\
		\end{align*}
		and
		\begin{align*}
			\ttt_*\gamma &= \Ad_{\varphi(h)}\Ad_{\theta_1(y_1)^{-1}} \sss_*(w) + \Ad_{\theta_2(y_2)^{-1}}(\ttt_*(w) - \Ad_g\sss_*(w) + \mu_{1*}\psi_{1*}(u_2) - \Ad_{\theta_2(y_2)}\varphi_*\nu_{2*}(u_2)) \\
			&\qquad +\Ad_{\theta_2(y_2)^{-1}}\Ad_g \sss_*(w)  - \Ad_{\theta_2(y_2)^{-1}g}\sss_*(w) - \Ad_{\varphi(h)}\Ad_{\theta_1(y_1)^{-1}}\sss_*(w) + \Ad_{\varphi(h)\theta_1(y_1)^{-1}}\sss_*(w) \\
			&\qquad - \Ad_{\theta_2(y_2)^{-1}}\Ad_{\theta_2(y_2)}(\varphi_*(v_2 - \nu_{2*}(u_2))) + \varphi_*(v_2 - \nu_{2*}(u_2)) \\
			&= \varphi_*(v_2 - \nu_{2*}(u_2)).
		\end{align*}
		It follows that $\M_1 \htimes_\H \M_2$ is transverse.
		
		In this case, we have a Lie groupoid morphism
		\begin{align*}
			\M_1 \htimes_\H \M_2 &\too \L_1 \htimes_\G \L_2 \\
			(l_1, h, l_2) &\mtoo (\psi_1(l_1), \theta_2(\sss(l_2)) \cdot \varphi(h) \cdot \theta_1(\sss(l_1))^{-1}, \psi_2(l_2)).
		\end{align*}
		To see that this is a Morita morphism, we first show that the map on objects is a surjective submersion. This is accomplished by establishing the existence of local sections. We first note that $\H \cong \varphi^*\G$. It follows that $(\M_1 \htimes_\H \M_2)\obj = \M_1\obj \times_{\varphi \nu_1, \sss} \G\arr \times_{\ttt, \varphi \nu_2} \M_2\obj$.
		The map on objects must therefore take the form
		\begin{align}
			\tilde{\varphi} : \M_1\obj \times_{\varphi \nu_1, \sss} \G\arr \times_{\ttt, \varphi \nu_2} \M_2\obj
			&\too
			\L_1\obj \times_{\mu_1, \sss} \G\arr \times_{\ttt, \mu_2} \L_2\obj \label{hadzso7n} \\
			(y_1, g, y_2) &\mtoo (\psi_1(y_1), \theta_2(y_2) \cdot g \cdot \theta_1(y_1)^{-1}, \psi_2(y_2)). \nonumber
		\end{align}
		Let $\sigma_i$ be local sections of $\psi_i : \M_i\obj \to \L_i\obj$.
		The map $$(x_1, g, x_2) \mto (\sigma_1(x_1), \theta_2(\sigma_2(x_2))^{-1} g \theta_1(\sigma_1(x_1)), \sigma_2(x_2))$$ is a local section of \eqref{hadzso7n}.
		A straightforward computation then shows that the induced map $\M_1 \htimes_\H \M_2 \longrightarrow \tilde{\varphi}^*(\L_1 \htimes_\G \L_2)$ is a diffeomorphism.
	\end{proof}
	
	\begin{proposition}\label{eiptq550}
		Consider 1-shifted Lagrangian correspondences
		\begin{equation}\label{ri7ihfoy}
			\begin{tikzcd}[row sep={2em,between origins},column sep={4em,between origins}]
				& \L_1 \arrow{dl} \arrow{dr} & & \L_2 \arrow{dl} \arrow{dr} & \\
				\G_1 & & \G_2 & & \G_3
			\end{tikzcd}
		\end{equation}
		and
		\begin{equation}\label{190c4zw2}
			\begin{tikzcd}[row sep={2em,between origins},column sep={4em,between origins}]
				& \L_1' \arrow{dl} \arrow{dr} & & \L_2' \arrow{dl} \arrow{dr} & \\
				\G_1' & & \G_2' & & \G_3'
			\end{tikzcd},
		\end{equation}
		where $\G_1 \longleftarrow \L_1 \longrightarrow \G_2$ is Lagrangially Morita equivalent to $\G_1' \leftarrow \L_1' \to \G_2'$, and $\G_2 \longleftarrow \L_2 \longrightarrow \G_3$ is Lagrangially Morita equivalent to $\G_2' \leftarrow \L_2' \to \G_3'$.
		Then \eqref{ri7ihfoy} is transverse if and only if \eqref{190c4zw2} is transverse.
		In this case, the compositions
		\begin{equation}
			\begin{tikzcd}[row sep={3em,between origins},column sep={5em,between origins}]
				& \L_1 \htimes_{\G_2} \L_2 \arrow{dl} \arrow{dr} & \\
				\G_1 & & \G_3
			\end{tikzcd}
			\qquad\text{and}\qquad
			\begin{tikzcd}[row sep={3em,between origins},column sep={5em,between origins}]
				& \L_1' \htimes_{\G_2'} \L_2' \arrow{dl} \arrow{dr} & \\
				\G_1' & & \G_3'
			\end{tikzcd}
		\end{equation}
		are Lagrangially Morita equivalent.
	\end{proposition}
	
	\begin{proof}
		Lemma \ref{32j00jes} implies the following: \eqref{ri7ihfoy} is transverse if and only if \eqref{190c4zw2} is transverse, in which case the homotopy fibre products $\L_1 \htimes_{\G_2} \L_2$ and $\L_1' \htimes_{\G_2'} \L_2'$ are Morita equivalent as Lie groupoids.
		It remains to check compatibility with the 1-shifted Lagrangian structures.
		We have a 2-commutative diagram
		\newcommand{\shortenlength}{30pt}
		\begin{equation}\label{for4dr3w}
			\begin{tikzcd}[row sep={7em,between origins},column sep={7em,between origins}]
				(\G_1, \omega_1, \phi_1) & (\L_1, \gamma_1) \arrow[swap]{l}{\mu_{11}} \arrow{r}{\mu_{12}} \arrow[Rightarrow,shorten=\shortenlength]{dl}{\theta_{11}} \arrow[Rightarrow,shorten=\shortenlength]{dr}{\theta_{12}} & (\G_2, \omega_2, \phi_2) & (\L_2, \gamma_2) \arrow[swap]{l}{\mu_{22}} \arrow{r}{\mu_{23}} \arrow[Rightarrow,shorten=\shortenlength]{dl}{\theta_{22}} \arrow[Rightarrow,shorten=\shortenlength]{dr}{\theta_{23}} & (\G_3, \omega_3, \phi_3) \\
				(\H_1, \delta_1) \arrow[swap]{u}{\varphi_1} \arrow{d}{\varphi_1'} & \M_1 \arrow[swap]{l}{\nu_{11}} \arrow{r}{\nu_{12}}  \arrow[swap]{u}{\psi_1} \arrow{d}{\psi_1'} & (\H_2, \delta_2) \arrow[swap]{u}{\varphi_2} \arrow{d}{\varphi_2'} & \M_2 \arrow[swap]{l}{\nu_{22}} \arrow{r}{\nu_{23}} \arrow[swap]{u}{\psi_2} \arrow{d}{\psi_2'} & (\H_3, \delta_3) \arrow[swap]{u}{\varphi_3} \arrow{d}{\varphi_3'} \\
				(\G_1', \omega_1', \phi_1') & (\L_1', \gamma_1') \arrow{l}{\mu_{11}'} \arrow[swap]{r}{\mu_{12}'} \arrow[Rightarrow,shorten=\shortenlength,swap]{ul}{\theta_{11}'} \arrow[Rightarrow,shorten=\shortenlength,swap]{ur}{\theta_{12}'} & (\G_2', \omega_2', \phi_2') & (\L_2', \gamma_2') \arrow{l}{\mu_{22}'} \arrow[swap]{r}{\mu_{23}'} \arrow[Rightarrow,shorten=\shortenlength,swap]{ul}{\theta_{22}'} \arrow[Rightarrow,shorten=\shortenlength,swap]{ur}{\theta_{23}'} & (\G_3', \omega_3', \phi_3')
			\end{tikzcd}
		\end{equation}
		of Lagrangian Morita equivalences.
		In other words, the vertical maps are Morita morphisms and the following hold:
		\begin{align}
			\varphi_i^*\omega_i - \varphi_i'^*\omega_i' &= \ttt^*\delta_i - \sss^*\delta_i \quad \text{for } i = 1, 2, 3 \nonumber; \\
			\varphi_i^*\phi_i - \varphi_i'^*\phi_i' &= -\mathrm{d}\delta_i \quad \text{for } i = 1, 2, 3 \nonumber; \\
			\psi_1^*\gamma_1 - \psi_1'^*\gamma_1' &= \nu_{11}^*\delta_1 - \nu_{12}^*\delta_2 - \theta_{11}^*\omega_1 + \theta_{12}^*\omega_2 + \theta_{11}'^*\omega_1' - \theta_{12}'^*\omega_2'; \label{upka972s} \\
			\psi_2^*\gamma_2 - \psi_2'^*\gamma_2' &= \nu_{22}^*\delta_2 - \nu_{23}^*\delta_3 - \theta_{22}^*\omega_2 + \theta_{23}^*\omega_3 + \theta_{22}'^*\omega_2' - \theta_{23}'^*\omega_3'. \label{yoxlfq77}
		\end{align}
		Taking homotopy fibre products in \eqref{for4dr3w} yields the 2-commutative diagram
		\renewcommand{\shortenlength}{46pt}
		\begin{equation}\label{m7fhcmlj}
			\begin{tikzcd}[row sep={6em,between origins},column sep={14em,between origins}]
				(\G_1, \omega_1, \phi_1) & (\L_1 \htimes_{\G_2} \L_2, \pr_{\L_1}^*\gamma_1 + \pr_{\L_2}^*\gamma_2) \arrow[swap]{l}{\mu_{11} \pr_{\L_1}} \arrow{r}{\mu_{23}\pr_{\L_2}} \arrow[Rightarrow,shorten=\shortenlength]{dl}{\theta_{11} \pr_{\M_1}} \arrow[Rightarrow,shorten=\shortenlength]{dr}{\theta_{23} \pr_{\M_2}} & (\G_3, \omega_3, \phi_3) \\
				(\H_1, \delta_1) \arrow[swap]{u}{\varphi_1} \arrow{d}{\varphi_1'} & \M_1 \htimes_{\H_2} \M_2 \arrow[swap]{l}{\nu_{11}\pr_{\M_1}} \arrow{r}{\nu_{23}\pr_{\M_2}}  \arrow[swap]{u}{\psi} \arrow{d}{\psi'} & (\H_2, \delta_2) \arrow[swap]{u}{\varphi_3} \arrow{d}{\varphi_3'} \\
				(\G_1', \omega_1', \phi_1') & (\L_1' \htimes_{\G_2'} \L_2', \pr_{\L_1'}\gamma_1' + \pr_{\L_2'}\gamma_2') \arrow{l}{\mu_{11}'\pr_{\L_1'}} \arrow[swap]{r}{\mu_{23}'\pr_{\L_2'}} \arrow[Rightarrow,shorten=\shortenlength,swap]{ul}{\theta_{11}' \pr_{\M_1}} \arrow[Rightarrow,shorten=\shortenlength,swap]{ur}{\theta_{23}' \pr_{\M_2}} & (\G_3', \omega_3', \phi_3')
			\end{tikzcd},
		\end{equation}
		where
		\begin{align*}
			\psi(m_1, h_2, m_2) &= (\psi_1(m_1), \theta_{22}(\sss(m_2))^{-1} \cdot \varphi_2(h_2) \cdot \theta_{12}(\sss(m_1)), \psi_2(m_2)) \\
			\psi'(m_1, h_2, m_2) &= (\psi_1'(m_1), \theta_{22}'(\sss(m_2))^{-1} \cdot \varphi_2'(h_2) \cdot \theta_{12}'(\sss(m_1)), \psi_2'(m_2))
		\end{align*}
		for all $(m_1, h_2, m_2) \in \M_1 \htimes_{\H_2} \M_2$.
		By Lemma \ref{32j00jes}, $\psi$ and $\psi'$ are Morita morphisms.
		Since $\pr_{\L_i} \psi = \psi_i \pr_{\M_i}$ and $\pr_{\L_i'} \psi' = \psi_i' \pr_{\M_i}$ for $i = 1, 2$, we have
		\begin{align*}
			&\psi^*(\pr_{\L_1}^*\gamma_1 + \pr_{\L_2}^*\gamma_2)
			-
			\psi'^*(\pr_{\L_1'}^*\gamma_1' + \pr_{\L_2'}^*\gamma_2')
			\\
			&=
			\pr_{\M_1}^*(\psi_1^*\gamma_1 - \psi_1'^*\gamma_1') + \pr_{\M_2}^*(\psi_2^*\gamma_2 - \psi_2'^*\gamma_2') \\
			&=
			\pr_{\M_1}^*(\nu_{11}^*\delta_1 - \nu_{12}^*\delta_2 - \theta_{11}^*\omega_1 + \theta_{12}^*\omega_2 + \theta_{11}'^*\omega_1' - \theta_{12}'^*\omega_2') \\
			&
			\,+\pr_{\M_2}^*(\nu_{22}^*\delta_2 - \nu_{23}^*\delta_3 - \theta_{22}^*\omega_2 + \theta_{23}^*\omega_3 + \theta_{22}'^*\omega_2' - \theta_{23}'^*\omega_3') \\
			&=
			(\nu_{11}\pr_{\M_1})^*\delta_1 - (\nu_{23}\pr_{\M_2})^*\delta_2
			-(\theta_{11}\pr_{\M_1})^*\omega_1
			+(\theta_{23}\pr_{\M_2})^*\omega_3
			+(\theta_{11}'\pr_{\M_1})^*\omega_1'
			-(\theta_{23}'\pr_{\M_2})^*\omega_3';
		\end{align*}
		The second equality follows from \eqref{upka972s} and \eqref{yoxlfq77}, and last equality follows from the fact that $\nu_{12} \pr_{\M_1} = \nu_{22} \pr_{\M_2}$ on $\M_1 \htimes_{\H_2} \M_2$.
		We conclude that \eqref{m7fhcmlj} is a Lagrangian Morita equivalence.
	\end{proof}
	
	\begin{lemma}\label{rw45fz6q}
		Let
		\begin{equation}\label{7h4kt9r9}
			\begin{tikzcd}[row sep={2em,between origins},column sep={4em,between origins}]
				& \L \arrow{dl} \arrow{dr} & \\
				\G_1 & & \G_2,
			\end{tikzcd}
		\end{equation}
		be a 1-shifted Lagrangian correspondence. Regard $\G_1 \longleftarrow \G_1 \longrightarrow \G_1$ as a 1-shifted Lagrangian correspondence via the trivial 2-form on $\G_1\obj$.
		Then the pair
		\begin{equation}\label{7qkor7bl}
			\begin{tikzcd}[row sep={2em,between origins},column sep={4em,between origins}]
				& \G_1 \arrow{dl} \arrow{dr} & & \L \arrow{dl} \arrow{dr} & &\\
				\G_1 & & \G_1 & & \G_2
			\end{tikzcd}
		\end{equation}
		is transverse, and the composition
		\begin{equation}\label{vcbwyrld}
			\begin{tikzcd}[row sep={3em,between origins},column sep={5em,between origins}]
				& \G_1 \htimes_{\G_1} \L \arrow{dl} \arrow{dr} & & \\
				\G_1 & & \G_2
			\end{tikzcd}
		\end{equation}
		is weakly equivalent to \eqref{7h4kt9r9}.
		The analogous statement for composition on the right by $\G_2 \longleftarrow \G_2 \longrightarrow \G_2$ also holds.
	\end{lemma}
	
	\begin{proof}
		Denote the quasi-symplectic structure on $\G_i$ by $(\omega_i, \phi_i)$ for $i = 1, 2$, and the 1-shifted Lagrangian structure on \eqref{7h4kt9r9} by $\gamma$.
		Let us also denote also the two morphisms in \eqref{7h4kt9r9} by $\mu_1:\L\longrightarrow\G_1$ and $\mu_2:\L\longrightarrow\G_2$.
		To establish composability in \eqref{7qkor7bl}, we need to show that the maps
		\begin{equation}\label{gnykqot1}
			\begin{tikzcd}
				\G_1\obj \times \L\obj \arrow[swap]{dr}{\Id \times \mu_1} & & \G_1\arr \arrow{dl}{(\sss, \ttt)} \\
				& \G_1\obj \times \G_1\obj &
			\end{tikzcd}
		\end{equation}
		are transverse.
		Let $g \in \G_1\arr$ be such that $(\sss(g), \ttt(g)) = (x, \mu_1(y))$ for some $(x, y) \in \G_1\obj \times \L\obj$, and let $(u, v) \in T_{\sss(g)}\G_1\obj \times T_{\ttt(g)}\G_1\obj$.
		Since the target map of a Lie groupoid is a submersion, we can write $v = \ttt_* (w)$ for some $w \in T_g\G_1\arr$.
		It follows that $(u, v) = (\Id_*(u - \sss_*(w)), \mu_{1*}(0)) + (\sss_*(w), \ttt_*(w))$, proving transversality.
		We also get that the projection
		\[
		\pr_\L : \G_1 \htimes_{\G_1} \L \too \L
		\]
		is a Morita morphism, fitting into a 2-commutative diagram
		\begin{equation}\label{yqkndi0s}
			\begin{tikzcd}[row sep={5em,between origins},column sep={5em,between origins}]
				\arrow[Rightarrow,shorten=42pt]{ddr}{\widetilde{\pr}_{\G_1}} & \G_1 \htimes_{\G_1} \L \arrow[swap]{dl}{\pr_{\G_1}} \arrow{dr}{\mu_2\pr_\L} \arrow{dd}{\pr_\L} & \\
				\G_1 & & \G_2 \\
				& \L \arrow{ul}{\mu_1} \arrow[swap]{ur}{\mu_2} &
			\end{tikzcd},
		\end{equation}
		where $\widetilde{\pr}_{\G_1} : \pr_{\G_1} \Rightarrow \mu_1 \pr_\L$ is the projection $(\G_1 \htimes_{\G_1} \L)\obj = \G_1\obj \fp{\Id}{\sss} \G_1\arr \fp{\ttt}{\mu_1} \L\obj \to \G_1\arr$.
		In the framework of weak equivalence of 1-shifted Lagrangians as in \eqref{msge1bqw}, we can write \eqref{yqkndi0s} as
		\begin{equation}\label{nyi4esjp}
			\begin{tikzcd}[row sep={5em,between origins},column sep={5em,between origins}]
				& \G_1 \htimes_{\G_1} \L \arrow[swap]{dl}{\Id} \arrow{dr}{\pr_\L} & \\
				\G_1 \htimes_{\G_1} \L \arrow[swap]{dr}{\pr_{\G_1} \times \mu_1 \pr_\L} \arrow[Rightarrow,shorten=20pt]{rr}{\widetilde{\pr}_{\G_1} \times \uuu \mu_2\pr_\L} & & \L \arrow{dl}{\mu_1 \times \mu_2} \\
				&\G_1 \times \G_2^-. &
			\end{tikzcd}
		\end{equation}
		By the definition of composition, the 1-shifted Lagrangian structure on the composition \eqref{vcbwyrld} is given by $\pr_{\G_1}^*0 + \pr_\L^*\gamma - \widetilde{\pr}_{\G_1}^*\omega_1$.
		The statement that \eqref{nyi4esjp} is a Lagrangian Morita equivalence then amounts to the identity
		\[
		\pr_\L^*\gamma - \Id^*(\pr_{\G_1}^*0 + \pr_\L^*\gamma - \widetilde{\pr}_{\G_1}^*\omega_1) = (\widetilde{\pr}_{\G_1} \times \uuu \mu_2 \pr_\L)^*(\pr_{\G_1}^*\omega_1 - \pr_{\G_2}^*\omega_2).
		\]
		This identity follows from the fact that $\uuu^*\omega_2 = 0$ for a multiplicative form $\omega_2$.
	\end{proof}

	\subsection{Strong fibre products and transversality}
	Consider 1-shifted Lagrangian correspondences
	\begin{equation}\label{e1rk3u0k}
		\begin{tikzcd}[row sep={2em,between origins},column sep={4em,between origins}]
			& \L_1 \arrow[swap]{dl} \arrow{dr} & & \L_2 \arrow[swap]{dl} \arrow{dr} & \\
			\G_1 & & \G_2 & & \G_3.
		\end{tikzcd}
	\end{equation}
	It is sometimes useful to consider the \defn{strong fibre product} $\L_1 \stimes_{\G_2} \L_2$, defined as the standard set-theoretical fibre product on both arrows and objects.
	We say that the 1-shifted Lagrangian correspondences \eqref{e1rk3u0k} are \defn{strongly transverse} if the fibre product on arrows is transverse.
	As in \cite[Proof of Lemma A.1.3]{bur-cab-hoy:16}, this implies that the fibre product on object is also transverse, and that the vector bundle morphisms
	\[
	\begin{tikzcd}[row sep={2em,between origins},column sep={4em,between origins}]
		A_{\L_1} \arrow{dr} &  & A_{\L_2} \arrow{dl} \\
		& A_{\G_2} &
	\end{tikzcd}
	\]
	are transverse.
	Indeed, fix
    $(x_1,x_2)\in
    \L_1\obj\fp{\mu_{12}}{\mu_{22}}\L_2\obj$
    and set
    $x\coloneqq\mu_{12}(x_1)=\mu_{22}(x_2)$.
    Strong transversality, applied to the pair of unit arrows
    $(\uuu_{x_1},\uuu_{x_2})$, says that
    \[
    T_{\uuu_{x_1}}(\L_1\arr)
    \times
    T_{\uuu_{x_2}}(\L_2\arr)
    \longrightarrow
    T_{\uuu_x}(\G_2\arr),
    \qquad
    (v_1,v_2)\longmapsto
    \mu_{12*}v_1+\mu_{22*}v_2
    \]
    is surjective. With respect to the canonical splittings
    \[
    T_{\uuu_{x_i}}(\L_i\arr)
    =
    \uuu_*T_{x_i}(\L_i\obj)
    \oplus
    (A_{\L_i})_{x_i},
    \qquad
    T_{\uuu_x}(\G_2\arr)
    =
    \uuu_*T_x(\G_2\obj)
    \oplus
    (A_{\G_2})_x,
    \]
    this map is the direct sum of
    \[
    T_{x_1}(\L_1\obj)\times T_{x_2}(\L_2\obj)
    \longrightarrow T_x(\G_2\obj)
    \]
    and
    \[
    (A_{\L_1})_{x_1}\times(A_{\L_2})_{x_2}
    \longrightarrow(A_{\G_2})_x.
    \]
    Surjectivity of the direct sum therefore implies surjectivity of both
    component maps.
	It follows that the strong fibre product is a Lie groupoid endowed with the structure of a 1-shifted Lagrangian correspondence
	\[
	\begin{tikzcd}[row sep={3em,between origins},column sep={5em,between origins}]
		& \L_1 \stimes_{\G_2} \L_2 \arrow[swap]{dl} \arrow{dr} & \\
		\G_1 & & \G_3
	\end{tikzcd}
	\]
	with respect to $\pr_{\L_1}^*\gamma_1 + \pr_{\L_2}^*\gamma_2$, where $\gamma_i$ are the 1-shifted Lagrangian structures on $\L_1$ and $\L_2$, respectively \cite[Theorem 4.1]{may:23}.
	The following proposition shows that the homotopy and strong fibre products are equivalent in many situations.
	
	\begin{proposition}\label{blagblt1}
		Suppose that the 1-shifted Lagrangian correspondences \eqref{e1rk3u0k} are transverse and strongly transverse.
		Suppose also that the map
		\begin{align}
			\L_1\arr \fp{\mu_{12} \sss}{\mu_{22} \sss} \L_2\arr
			&\too
			\L_1\obj \fp{\mu_{12}}{\sss} \G_2\arr \fp{\ttt}{\mu_{22}} \L_2\obj \label{nvwikwcm} \\
			(l_1, l_2) &\mtoo (\ttt(l_1), \mu_{22}(l_2)^{-1}\mu_{12}(l_1), \ttt(l_2)) \nonumber
		\end{align}
		is a surjective submersion.
		Then the homotopy fibre product $\L_1 \htimes_{\G_2} \L_2$ and the strong fibre product $\L_1 \stimes_{\G_2} \L_2$ are weakly equivalent as 1-shifted Lagrangian correspondences from $\G_1$ to $\G_3$.
	\end{proposition}
	
	\begin{proof}
		We have a commutative diagram
		\[
		\begin{tikzcd}[row sep={3em,between origins}, column sep={4em,between origins}]
			\L_1 \stimes_{\G_2} \L_2 \arrow{rr}{\psi} \arrow{dr} & & \L_1 \htimes_{\G_2} \L_2 \arrow{dl} \\
			& \G_1 \times \G_3^{-}, &
		\end{tikzcd}
		\]
		where
		\begin{align*}
			\psi\arr : \L_1\arr \fp{\mu_{12}}{\mu_{22}} \L_2\arr &\too \L_1\arr \fp{\mu_{12}\sss}{\sss} \G_2\arr \fp{\ttt}{\mu_{22}\sss} \L_2 \\
			(l_1, l_2) &\mtoo (l_1, \uuu_{\mu_{12}(\sss(l_1))}, l_2) \hspace{10pt} \text{and}\\
			\psi\obj : \L_1\obj \fp{\mu_{12}}{\mu_{22}} \L_2\obj &\too \L_1\obj \fp{\mu_{12}}{\sss} \G_2\arr \fp{\ttt}{\mu_{22}} \L_2\obj \\
			(x_1, x_2) &\mtoo (x_1, \uuu_{\mu_{12}(x_1)}, x_2).
		\end{align*}
		Note that $\psi^*(\pr_{\L_1}^*\gamma_1 + \pr_{\L_2}^*\gamma_2 - \pr_{\G_2}^*\omega_2) = \pr_{\L_1}^*\gamma_1 + \pr_{\L_2}^*\gamma_2$, as $\uuu^*\omega_2 = 0$.
		It suffices to check that $\psi$ is an essential equivalence.
		This amounts to checking the following.
		\begin{enumerate}[label={(\arabic*)}]
			\item\label{ijibfqjp}
			The map
			\begin{align}
				\ttt \circ \pr_{(\L_1 \htimes_{\G_2} \L_2)\arr} :
				(\L_1\obj \fp{\mu_{12}}{\mu_{22}} \L_2\obj) \fp{\psi\obj}{\sss} 
				(\L_1\arr \fp{\mu_{12}\sss}{\sss} \G_2\arr \fp{\ttt}{\mu_{22}\sss} \L_2) &\too \L_1\obj \fp{\mu_{12}}{\sss} \G_2\arr \fp{\ttt}{\mu_{22}} \L_2\obj \label{89rwy53w} \\
				((x_1, x_2), (l_1, g, l_2)) &\mtoo (\ttt(l_1), \mu_{22}(l_2) \cdot g \cdot \mu_{12}(l_1)^{-1}, \ttt(l_2)) \nonumber
			\end{align}
			is a surjective submersion.
			\item\label{dna6bfpu}
			The map $(\L_1 \stimes_{\G_2} \L_2)\arr \to (\psi\obj)^*(\L_1 \htimes_{\G_2} \L_2)$ is a diffeomorphism.
		\end{enumerate}
		To show \ref{ijibfqjp}, note that we have a map
		\begin{align*}
			\L_1\arr \fp{\mu_{12}\sss}{\mu_{22}\sss} \L_2\arr
			&\too
			(\L_1\obj \fp{\mu_{12}}{\mu_{22}} \L_2\obj) \fp{\psi\obj}{\sss} 
			(\L_1\arr \fp{\mu_{12}\sss}{\sss} \G_2\arr \fp{\ttt}{\mu_{22}\sss} \L_2) \\
			(l_1, l_2) &\mtoo ((\sss(l_1), \sss(l_2)), (l_1, \uuu_{\mu_{12}(\sss(l_1))}, l_2)),
		\end{align*}
		whose composition with \eqref{89rwy53w} is \eqref{nvwikwcm}.
		Since \eqref{nvwikwcm} is a surjective submersion, so is \eqref{89rwy53w}.
		For \ref{dna6bfpu}, we need to check that the map
		\begin{align*}
			\L_1\arr \fp{\mu_{12}}{\mu_{22}} \L_2\arr &\too (\L_1\obj \fp{\mu_{12}}{\mu_{22}} \L_2\obj) \fp{\psi\obj}{\sss} (\L_1\arr \fp{\mu_{12}\sss}{\sss} \G_2\arr \fp{\ttt}{\mu_{22}\sss} \L_2) \fp{\ttt}{\psi\obj} (\L_1\obj \fp{\mu_{12}}{\mu_{22}} \L_2\obj) \\
			(l_1, l_2) &\mtoo ((\sss(l_1), \sss(l_2)), (l_1, \uuu_{\mu_{12}(\sss(l_1))}, l_2), (\ttt(l_1), \ttt(l_2)))
		\end{align*}
		is a diffeomorphism.
		But it has an explicit inverse, given by $((x_1, x_2), (l_1, g, l_2), (y_1, y_2)) \mto (l_1, l_2)$.
	\end{proof}

\subsection{Homotopy fibre products and transversality}\label{s0ja09pn}

We now explain the homotopical meaning of the weak fibre product used above and justify the transversality condition introduced in Subsection \ref{e3lhuxkc}, including its relation to the path-object and fibration-replacement constructions.

Our overarching principle is that Lie groupoids should be viewed as representing stacks via their Morita equivalence class or, more precisely, using the equivalence of bicategories from the bicategory of Lie groupoids localized at essential equivalences to the 2-category of differentiable stacks \cite{beh-xu:11,pro:96}; see also \cite[Theorem 2.3]{hep:09}, \cite[Theorem I.2.1]{car:11}, \cite[\S4.2]{hof-sja:21}, and \cite{vil:18}.

Therefore, from this point of view, given two Lie groupoid morphisms
\[
\begin{tikzcd}
    \L_1 \arrow{r}{\mu_1} & \G & \L_2 \arrow[swap]{l}{\mu_2},
\end{tikzcd}
\]
the naive set-theoretic fibre product is not suitable. Indeed, it requires the equality $\mu_1(x_1)=\mu_2(x_2)$ in the object manifold of the chosen presentation $\G$ of the stack. This equality of representatives is not intrinsic to the underlying stacks: replacing the Lie groupoids and morphisms by Morita-equivalent presentations need not produce a Morita-equivalent set-theoretic fibre product, and hence need not present the same fibre product of differentiable stacks. Instead, the appropriate notion is the weak fibre product, whose objects are triples $(x_1,g,x_2)$, where $x_1 \in \L_1\obj$ and $x_2 \in \L_2\obj$ do not necessarily have the same image in $\G\obj$, but their images are related by an arrow $g \in \G\arr$, i.e.\
\[
    \sss(g)=\mu_1(x_1)
    \qquad\text{and}\qquad
    \ttt(g)=\mu_2(x_2).
\]
An arrow from $(x_1,g,x_2)$ to $(x'_1,g',x'_2)$ consists of arrows $l_i:x_i\to x'_i$ in $\L_i$ such that the corresponding square in $\G$ commutes, namely
\[
    \mu_2(l_2)g=g'\mu_1(l_1);
\]
see \cite[Section 5.3]{moe-mrc:03}, \cite{hoy:13}, or \cite[\S4.2]{may:23}. In particular, the set of objects is the set-theoretic fibre product of the maps
\[
\begin{tikzcd}[row sep={1em},column sep={2em}]
    \L_1\obj \times \L_2\obj \arrow[swap]{dr}{(\mu_1,\mu_2)} & & \G\arr \arrow{dl}{(\sss,\ttt)} \\
    & \G\obj \times \G\obj. &
\end{tikzcd}
\]
It is then natural to require transversality, so that this set-theoretic fibre product is a manifold of the expected dimension and we remain in the bicategory of Lie groupoids localised at essential equivalences, without having to enhance the framework to include derived differentiable stacks. Furthermore, this transversality condition on the object space guarantees that the corresponding arrow space is also a manifold \cite[Section 5.3]{moe-mrc:03}. Thus the condition imposed above suffices for the weak fibre product to be a Lie groupoid, and this is the condition that we use in this paper.

Moreover, there is the following important point. This transversality condition is well-behaved with respect to Morita equivalences by Lemma \ref{32j00jes}, and hence makes sense intrinsically as a stack-theoretic condition. It is also convenient in the present context, since it behaves well with respect to the composition of Lagrangian correspondences by Proposition \ref{eiptq550}.

We therefore interpret our transverse homotopy fibre product as an explicit Lie groupoid presentation of an intrinsic stack-theoretic transverse fibre product.

This construction also fits naturally into the general framework of homotopy categories. Indeed, Lie groupoids can be viewed as special cases of Lie $\infty$-groupoids, and the latter form an incomplete category of fibrant objects, as shown by Rogers--Zhu \cite[Theorem 7.1]{rog-zhu:20}. In this formalism, a homotopy fibre product over a given object can be constructed using a path object for that object: instead of requiring the two images to coincide, one keeps track of a path relating them. For an ordinary Lie groupoid $\G$, the relevant path object is the groupoid $\G^I$ whose objects are the arrows of $\G$ and whose arrows are the commutative squares in $\G$. It gives a factorisation
\[
\G\overset{\sim}{\longrightarrow}\G^I
\overset{(\sss,\ttt)}{\longrightarrow}\G\times\G
\]
of the diagonal, where the first morphism assigns to each object its
identity arrow and is a weak equivalence, while the second is a
fibration \cite[Proposition 7.13]{rog-zhu:20}. Whenever the corresponding ordinary fibre product exists as a Lie groupoid, the path-object construction of the homotopy fibre product is
\[
    (\L_1\times\L_2)\times_{\G\times\G}\G^I,
\]
which is precisely the weak fibre product described above. Our transversality condition guarantees that this ordinary fibre product exists. This concrete description for Lie groupoids was given by del Hoyo \cite[Sections 4.1 and 4.4]{hoy:13}.

The path object also provides the standard fibration replacement of either one of the two morphisms; see \cite[Lemma 2.4]{rog-zhu:20}. For example, it replaces $\mu_1:\L_1\to\G$ by
\[
    \L_1
    \overset{\sim}{\longrightarrow}
    \widetilde{\L}_1
    \coloneqq
    \L_1\times_{\mu_1,\G,\sss}\G^I
    \overset{\widetilde{\mu}_1}{\longrightarrow}
    \G,
    \qquad
    \widetilde{\mu}_1
    \coloneqq
    \ttt\circ\pr_{\G^I},
\]
where the second morphism is a Kan fibration. The key point is that this fibration-replacement procedure does not lead to a different transversality condition. Indeed, in the smooth setting the ordinary fibre product of $\widetilde{\mu}_1$ and $\mu_2$ need not exist automatically, and its candidate object space is
\[
    \bigl(
        \L_1\obj\fp{\mu_1}{\sss}\G\arr
    \bigr)
    \fp{\ttt\circ\pr_{\G\arr}}{\mu_2}
    \L_2\obj.
\]
Since $\sss:\G\arr\to\G\obj$ is a submersion, the first fibre product is automatically a manifold, and the tangent-space criterion for transversality gives
\[
    \ttt\circ\pr_{\G\arr}\pitchfork\mu_2
    \qquad\Longleftrightarrow\qquad
    (\mu_1,\mu_2)\pitchfork(\sss,\ttt).
\]
Thus the transversality condition arising from the fibration-replacement procedure is precisely the condition used in this paper. Under these equivalent conditions, the ordinary pullback exists as a Lie groupoid \cite[Proposition 7.10 and Remark 7.11]{rog-zhu:20}. Unpacking its objects and arrows gives a canonical isomorphism with the weak fibre product described above.
	
	\section{The 1-shifted Weinstein symplectic category}\label{Section: Shifted}
	
	We begin this section with a rough overview of the Weinstein symplectic ``category" and its Wehrheim--Woodward completion. This gives context for our subsequent definition of the $1$-shifted Weinstein symplectic ``category" $\WSQ$. Using the basic approach of Wehrheim--Woodward to the Weinstein symplectic category, we complete $\WSQ$ to a symmetric monoidal category $\WS$. One may work in the smooth or holomorphic categories, as with the previous section.
	
	\subsection{The Weinstein symplectic ``category"}
	Two morphisms in a category can be composed if and only if the source of one coincides with the target of the other. By weakening this to a necessary condition for composing morphisms, one obtains the definition of a \defn{``category''}.\footnote{Some authors would instead call this a \emph{precategory}.} Two morphisms in a ``category" are called \defn{composable} if their composition is defined. A prominent instance of this discussion is Weinstein's symplectic ``category'' \cite{wei:10}; its objects are symplectic manifolds, and its morphisms are Lagrangian correspondences. While one can compose two Lagrangian correspondences as relations between sets, the result need not be a Lagrangian correspondence. These correspondences are called composable if they satisfy transversality conditions sufficient to ensure that their set-theoretic composition is a Lagrangian correspondence.     
	While the Weinstein symplectic ``category'' is not a genuine category, Wehrheim--Woodward \cite{weh-woo:10,weh:16} show that it can be completed into one by defining morphisms as sequences of composable morphisms up to a certain equivalence relation. We implement a similar approach for the 1-shifted version of the Weinstein symplectic ``category".
	
	\subsection{The 1-shifted Weinstein symplectic ``category''.}
	We define a 1-shifted version of Weinstein's symplectic ``category'', denoted $\WSQ$, as follows.
	Quasi-symplectic groupoids constitute the objects of $\WSQ$. A morphism from $\G_1$ to $\G_2$ in $\WSQ$ is a weak equivalence class of 1-shifted Lagrangian correspondences from $\G_1$ to $\G_2$. 
%
	We say that two morphisms
	\[
	\begin{tikzcd}[row sep={2em,between origins},column sep={4em,between origins}]
		& \L_1 \arrow{dl} \arrow{dr} & & \L_2 \arrow{dl}\arrow{dr} & \\
		\G_1 & & \G_2 & & \G_3
	\end{tikzcd}
	\]
	in $\WSQ$ are \defn{composable} if their homotopy fibre product $\L_1 \htimes_{\G_2} \L_2$ is transverse.
	In this case, we define the composition $\L_2 \circ \L_1$ as the weak equivalence class of $\L_1 \htimes_{\G_2} \L_2$.
	Proposition \ref{eiptq550} implies that this morphism does not depend on the representatives chosen for the weak equivalence classes being composed.
	
	Let $\G$ be a quasi-symplectic groupoid. The identity morphism $\Id_\G:\G\longrightarrow\G$ is the canonical 1-shifted Lagrangian correspondence $\G \longleftarrow \G \longrightarrow \G$ (Example \ref{c00ufvui}). The content of this statement is that every morphism $\L : \G \longrightarrow \H$ in $\WSQ$ is composable with $\Id_\G$ and $\Id_\H$, and $\L \circ \Id_\G \simeq \L \simeq \Id_\H \circ \L$; see Lemma \ref{rw45fz6q}.
	
	\subsection{Extension of $\WSQ$ to a symmetric monoidal category}
	One may extend $\WSQ$ to a symmetric monoidal category $\WS$ in two ways. The first is a manifold-theoretic extension afforded by Wehrheim--Woodward; see \cite[Section 2]{weh-woo:10} and \cite[Definition 2.2.2]{weh:16}. The second, due to Calaque \cite{cal:15}, is algebro-geometric; one replaces algebraic quasi-symplectic groupoids by their associated 1-shifted symplectic stacks, and uses derived fibre products to form a symmetric monoidal category. 

The precise way in which we extend $\WSQ$ to a category $\WS$ is irrelevant since all computations, including the Frobenius-object computations in Section \ref{90msjexi} and the compatibility checks for the affinization process in Section \ref{Section: Affinization}, will be done on composable morphisms in $\WSQ$.
We therefore only need to prove the existence of an extension.
We do this in the differential-geometric context by adapting the Wehrheim--Woodward approach, as we now explain.
	
	The Wehrheim--Woodward approach is as follows.
	An object of $\WS$ is a quasi-symplectic groupoid.
	A morphism from $\G$ to $\G'$ is a sequence of quasi-symplectic groupoids $\G_0, \G_1, \ldots, \G_r$ with $\G_0 = \G$ and $\G_r = \G'$, together with 1-shifted Lagrangian correspondences $\G_{i-1} \longleftarrow \L_i \longrightarrow \G_{i}$ for $i = 1, \ldots, r$, up to the equivalence relation generated by
	\begin{equation}\label{kdi2q630}
		\left(
		\begin{tikzcd}[row sep={1.5em,between origins},column sep={2em,between origins}]
			& & \L_i \arrow{ddl} \arrow{ddr} & & \L_{i+1} \arrow{ddl} \arrow{ddr} & & \\
			\cdots & & & & & & \cdots \\
			& \G_{i-1} & & \G_{i} & & \G_{i+1} & 
		\end{tikzcd}
		\right)
		\quad\sim\quad
		\left(
		\begin{tikzcd}[row sep={1.5em,between origins},column sep={2em,between origins}]
			& & \L_i \htimes_{\G_i} \L_{i+1} \arrow{ddl} \arrow{ddr} & & \\
			\cdots & & & & \cdots \\
			& \G_{i-1} & & \G_{i+1} &
		\end{tikzcd}
		\right)
	\end{equation}
	if $\L_i \htimes_{\G_i} \L_{i+1}$ is transverse and
	\begin{equation}\label{m4aup4gt}
		\left(
		\begin{tikzcd}[row sep={1.5em,between origins},column sep={2em,between origins}]
			& & \L_i \arrow{ddl} \arrow{ddr} & & \\
			\cdots & & & & \cdots \\
			& \G_{i-1} & & \G_{i} & 
		\end{tikzcd}
		\right)
		\quad\sim\quad
		\left(
		\begin{tikzcd}[row sep={1.5em,between origins},column sep={2em,between origins}]
			& & \L_i' \arrow{ddl} \arrow{ddr} & & \\
			\cdots & & & & \cdots \\
			& \G_{i-1} & & \G_{i} & 
		\end{tikzcd}
		\right)
	\end{equation}
	if $\L_i$ and $\L_i'$ are weakly equivalent.
	As in \cite[Section 2]{weh-woo:10}, this forms a category with composition given by the concatenation of sequences.
	
	The category $\WS$ turns out to carry a symmetric monoidal structure \begin{equation}\label{k8avzbtp}
		\otimes : \WS \times \WS \too \WS.
	\end{equation}
	It is given by the Cartesian product on the level of objects.
	For two morphisms
	\[
	\left[
	\begin{tikzcd}[row sep={2em,between origins},column sep={2em,between origins}]
		& \L_1 \arrow{dl} \arrow{dr} & & \L_2 \arrow{dl} \arrow{dr} & & \L_r \arrow{dl} \arrow{dr} & \\
		\G_0 & & \G_1 & & \cdots & & \G_r
	\end{tikzcd}
	\right]
	\quad\text{and}\quad
	\left[
	\begin{tikzcd}[row sep={2em,between origins},column sep={2em,between origins}]
		& \M_1 \arrow{dl} \arrow{dr} & & \M_2 \arrow{dl} \arrow{dr} & & \M_s \arrow{dl} \arrow{dr} & \\
		\H_0 & & \H_1 & & \cdots & & \H_s
	\end{tikzcd}
	\right],
	\]
	one takes the tensor product as follows. Augment the morphism of smallest length with identity morphisms on its right, until to its length matches that of other morphism. Proceed to take Cartesian products of quasi-symplectic groupoids $\G_i \times \H_i$ and 1-shifted Lagrangians $\L_i \times \M_i$. Using the more concise notation $[\L_1, \ldots, \L_r]$ and $[\M_1, \ldots, \M_s]$ for the morphisms, assuming $r \le s$, and using juxtaposition to indicate Cartesian products, one has
	
	\begin{equation}\label{Equation: Tensor}
	[\L_1, \ldots, \L_r] \otimes [\M_1, \ldots, \M_s] = [\L_1 \M_1, \ldots, \L_r\M_r, \G_r\M_{r+1}, \ldots, \G_r\M_s].
	\end{equation}
	
	\begin{lemma}
	The right-hand side of \eqref{Equation: Tensor} does not depend on the representatives chosen for the two morphisms being composed.
	\end{lemma}
	
	\begin{proof}
	The right-hand side of \eqref{Equation: Tensor} it clearly invariant under \eqref{m4aup4gt}.
	To establish invariance under \eqref{kdi2q630}, let $1 \le i < r$ be such that $\L_i$ and $\L_{i+1}$ are transverse. It follows that
	\[
	[\L_1, \ldots, \L_i, \L_{i+1}, \ldots, \L_r]
	=
	[\L_1, \ldots, \L_i \circ \L_{i+1}, \ldots, \L_r].
	\]
	We need to check that both presentations of this morphism yield the same tensor product with $[\M_1, \ldots, \M_s]$, i.e.\ that
	\begin{align}
		&\quad [\L_1\M_1, \ldots, \L_i\M_i, \L_{i+1}\M_{i+1}, \ldots, \L_r\M_r, \G_r\M_{r+1}, \ldots, \G_r\M_s] \label{lght76j5} \\
		&= [\L_1\M_1, \ldots, (\L_i\circ\L_{i+1})\M_i, \L_{i+2}\M_{i+1}, \ldots, \L_r\M_{r-1}, \G_r\M_r, \ldots, \G_r\M_s]. \label{harcc7ph}
	\end{align}
	
	By inserting identities, \eqref{harcc7ph} is equal to
	\[
	[\L_1\M_1, \ldots, \underbrace{\L_i\M_i, \L_{i+1}\H_i}_{(\L_i \circ \L_{i+1})\M_i}, \underbrace{\G_{i+1}\M_{i+1}, \L_{i+2}\H_{i+1}}_{\L_{i+2}\M_{i+1}}, \ldots, \underbrace{\G_{r-1}\M_{r-1}, \L_r\H_{r-1}}_{\L_r\M_{r-1}}, \G_r\M_r, \ldots, \G_r\M_s].
	\]
	Subsequently performing composition in the other set of pairs
	\[
	[\L_1\M_1, \ldots, \L_i\M_i, \underbrace{\L_{i+1}\H_i, \G_{i+1}\M_{i+1}}_{\L_{i+1}\M_{i+1}}, \underbrace{\L_{i+2}\H_{i+1}, \ldots}_{\L_{i+2}\M_{i+2}} \underbrace{\ldots, \G_{r-1}\M_{r-1}}_{\L_{r-1}\M_{r-1}}, \underbrace{\L_r\H_{r-1}, \G_r\M_r}_{\L_r\M_r}, \ldots, \G_r\M_s],
	\]
	we get back \eqref{lght76j5}.
	This completes the proof.
	\end{proof}
	
	The unit object in $\WS$ is a point $\star$, viewed as a quasi-symplectic groupoid. Given quasi-symplectic groupoids $\G$, $\H$, and $\I$, the associator $\alpha_{\G,\H,\I} : \G \times (\H \times \I) \longrightarrow (\G \times \H) \times \I$ is the groupoid $\G \times \H \times \I$, regarded as a 1-shifted Lagrangian correspondence from $\G \times (\H \times \I)$ to $(\G \times \H) \times \I$.
	The left identity $\lambda_\G : \star \times \G \longrightarrow \G$ is $\G$, as a 1-shifted Lagrangian correspondence. A similar description applies for the right identity $\rho_\G : \G \times \star \longrightarrow \G$.
	
	The preceding discussion makes it clear that \eqref{k8avzbtp} is a monoidal structure on $\WS$.
	To address the symmetric structure, let $\G$ and $\H$ be quasi-symplectic groupoids. The braiding $\G \times \H \longrightarrow \H \times \G$ is the groupoid $\G \times \H$ with trivial 2-form on its base and the obvious morphisms to $\G \times \H$ and $\H \times \G$.
	The axioms of a symmetric monoidal category are immediate.
	
	\section{TQFTs valued in the 1-shifted Weinstein symplectic category}\label{90msjexi}

	This section is largely concerned with proving Main Theorem \ref{Theorem: Main 1}. We begin by recalling the equivalence between $2$-dimensional TQFTs in a symmetric monoidal category and commutative Frobenius objects in the same category. This framework allows us to prove that every abelian symplectic groupoid determines a commutative Frobenius object in $\WS$, whose product is induced by groupoid multiplication. To extend this result to a larger class of quasi-symplectic groupoids, we call a quasi-symplectic groupoid abelianizable if it Morita equivalent to an abelian symplectic groupoid. We then use a notion of Morita transfer to prove that every abelianizable quasi-symplectic groupoid determines a commutative Frobenius object in $\WS$. To conclude, we prove that a quasi-symplectic groupoid admitting an admissible global slice is necessarily abelianizable. All constructions work over $\mathbb{R}$ or $\mathbb{C}$, as in the previous two sections.   
	
	\subsection{2-dimensional TQFTs and commutative Frobenius objects}
	Let us first recall the equivalence between 2-dimensional TQFTs and commutative Frobenius objects. A standard reference for this material is the book \cite{koc:04}; see also \cite[Section 2]{kho-ost-kon:22}, \cite{eti-gel-nik-ost:15}, and \cite{sch:09}.

Let $\Cob_2$ be the category of 2-dimensional cobordisms, i.e.\ objects are compact 1-dimensional oriented manifolds and morphisms are cobordisms between them.
	Suppose that $(\mathbf{C}, \otimes, I, B)$ is a symmetric monoidal category, where $\otimes$ is the monoidal product, $I$ is the unit object, and $B$ is the braiding.
	A \defn{2-dimensional topological quantum field theory valued in $\mathbf{C}$} is a symmetric monoidal functor
	\[
	\Cob_2 \too \mathbf{C}.
	\]
	Such a functor is more easily described via the notion of commutative Frobenius object in $\mathbf{C}$, as we now recall. One finds that $\Cob_2$ is generated by the six morphisms
	
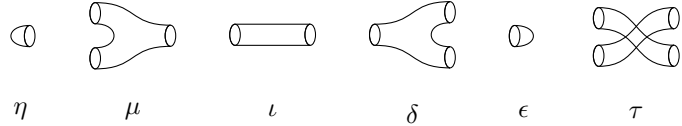
\[
	\begin{tikzpicture}[
		baseline=-2.5pt,
		every tqft/.append style={
			transform shape, rotate=90, tqft/circle x radius=4pt,
			tqft/circle y radius= 2pt,
			tqft/boundary separation=0.6cm, 
			tqft/view from=incoming,
		}
		]
		\pic[
		tqft/cap,
		name=d,
		every incoming lower boundary component/.style={draw},
		every outgoing lower boundary component/.style={draw},
		every incoming boundary component/.style={draw},
		every outgoing boundary component/.style={draw},
		cobordism edge/.style={draw},
		cobordism height= 1cm,
		];
		\draw (d-outgoing boundary)+(-0.12,-1) node {$\eta$};
	\end{tikzpicture}
	\qquad
	\begin{tikzpicture}[
		baseline=6pt,
		every tqft/.append style={
			transform shape, rotate=90, tqft/circle x radius=4pt,
			tqft/circle y radius= 2pt,
			tqft/boundary separation=0.6cm, tqft/view from=incoming,
		}
		]
		\pic[
		tqft/reverse pair of pants,
		name=d,
		every incoming lower boundary component/.style={draw},
		every outgoing lower boundary component/.style={draw},
		every incoming boundary component/.style={draw},
		every outgoing boundary component/.style={draw},
		cobordism edge/.style={draw},
		cobordism height= 1cm,
		];
		\draw (d-outgoing boundary)+(-0.5,-1) node {$\mu$};
	\end{tikzpicture}
	\qquad
	\begin{tikzpicture}[
		baseline=-3pt,
		every tqft/.append style={
			transform shape, rotate=90, tqft/circle x radius=4pt,
			tqft/circle y radius= 2pt,
			tqft/boundary separation=0.6cm, tqft/view from=incoming,
		}
		]
		\pic[
		tqft/cylinder,
		name=d,
		every incoming lower boundary component/.style={draw},
		every outgoing lower boundary component/.style={draw},
		every incoming boundary component/.style={draw},
		every outgoing boundary component/.style={draw},
		cobordism edge/.style={draw},
		cobordism height= 1cm,
		];
		\draw (d-outgoing boundary)+(-0.5,-1) node {$\iota$};
	\end{tikzpicture}
	\qquad
	\begin{tikzpicture}[
		baseline=-3pt,
		every tqft/.append style={
			transform shape, rotate=90, tqft/circle x radius=4pt,
			tqft/circle y radius= 2pt,
			tqft/boundary separation=0.6cm, tqft/view from=incoming,
		}
		]
		\pic[
		tqft/pair of pants,
		name=d,
		every incoming lower boundary component/.style={draw},
		every outgoing lower boundary component/.style={draw},
		every incoming boundary component/.style={draw},
		every outgoing boundary component/.style={draw},
		cobordism edge/.style={draw},
		cobordism height= 1cm,
		];
		\draw (d-outgoing boundary)+(-0.5,-0.75) node {$\delta$};
	\end{tikzpicture}
	\qquad
	\begin{tikzpicture}[
		baseline=-2.5pt,
		every tqft/.append style={
			transform shape, rotate=90, tqft/circle x radius=4pt,
			tqft/circle y radius=2pt,
			tqft/boundary separation=0.6cm, 
			tqft/view from=incoming,
		}
		]
		\pic[
		tqft/cup,
		name=d,
		every incoming lower boundary component/.style={draw},
		every outgoing lower boundary component/.style={draw},
		every incoming boundary component/.style={draw},
		every outgoing boundary component/.style={draw},
		cobordism edge/.style={draw},
		cobordism height=1cm,
		];
		\draw (d-incoming boundary)+(0.12,-1) node {$\epsilon$};
	\end{tikzpicture}
	\qquad
	\begin{tikzpicture}[
		baseline=5pt,
		every tqft/.append style={
			transform shape, rotate=90, tqft/circle x radius=4pt,
			tqft/circle y radius= 2pt,
			tqft/cobordism height=1cm, tqft/view from=incoming,
			tqft/boundary separation=1cm,
		}
		]
		\pic [
		tqft/cylinder to next,
		name=d,
		every incoming lower boundary component/.style={draw},
		every outgoing lower boundary component/.style={draw},
		cobordism edge/.style={draw}, 
		every incoming boundary component/.style={draw},
		every outgoing boundary component/.style={draw},
		];
		\pic [
		tqft/cylinder to prior,
		every incoming lower boundary component/.style={draw},
		every outgoing lower boundary component/.style={draw},
		every incoming boundary component/.style={draw},
		every outgoing boundary component/.style={draw},
		cobordism edge/.style={draw},
		at={($(d-incoming boundary)+(0, 0.5)$)}];
		\draw (d-outgoing boundary)+(-0.5,-1.25) node {$\tau$};
	\end{tikzpicture}
	\]
	read from left to right (e.g.\ $\eta : \emptyset \longrightarrow S^1$), subject to the relations
	\begin{align}
		&\begin{tikzpicture}[
			every tqft/.append style={
				transform shape, rotate=90, tqft/circle x radius=4pt,
				tqft/circle y radius=2pt,
				tqft/boundary separation=0.5cm, tqft/view from=incoming,
			}
			]
			\pic[
			tqft/cap,
			name=a,
			cobordism height=1cm,
			every incoming boundary component/.style={draw},
			every outgoing boundary component/.style={draw},
			cobordism edge/.style={draw},
			];
			\pic[
			tqft/cylinder,
			name=b,
			cobordism height=0.75cm,
			every incoming boundary component/.style={draw},
			every outgoing boundary component/.style={draw},
			cobordism edge/.style={draw},
			at={($ (a-outgoing boundary)+(-0.75,0.5) $) },
			];
			\pic[
			tqft/reverse pair of pants,
			name=c,
			cobordism height=0.75cm,
			every incoming boundary component/.style={draw},
			every outgoing boundary component/.style={draw},
			cobordism edge/.style={draw},
			at={($ (a-outgoing boundary)+(0,0) $) },
			];
			\draw (c-outgoing boundary)+(0.5,0)  node {\textup{=}};
			\pic[
			tqft/cylinder,
			name=e,
			cobordism height=0.75cm,
			every incoming boundary component/.style={draw},
			every outgoing boundary component/.style={draw},
			cobordism edge/.style={draw},
			at={($ (c-outgoing boundary)+(1,0) $) },
			];
			\draw (c-outgoing boundary)+(2.25,0)  node {\textup{=}};
			\pic[
			tqft/cap,
			name=ar,
			cobordism height=1cm,
			every incoming boundary component/.style={draw},
			every outgoing boundary component/.style={draw},
			cobordism edge/.style={draw},
			at={($(4.25, 0.5)$)},
			];
			\pic[
			tqft/cylinder,
			name=br,
			cobordism height=0.75cm,
			every incoming boundary component/.style={draw},
			every outgoing boundary component/.style={draw},
			cobordism edge/.style={draw},
			at={($ (ar-outgoing boundary)+(-0.75,-0.5) $) },
			];
			\pic[
			tqft/reverse pair of pants,
			name=cr,
			cobordism height=0.75cm,
			every incoming boundary component/.style={draw},
			every outgoing boundary component/.style={draw},
			cobordism edge/.style={draw},
			at={($ (ar-outgoing boundary)+(0,-0.5) $) },
			];
		\end{tikzpicture}
		\label{wzw7o5ta}\\
		&\begin{tikzpicture}[
			every tqft/.append style={
				transform shape, rotate=90, tqft/circle x radius=4pt,
				tqft/circle y radius=2pt,
				tqft/boundary separation=0.5cm, tqft/view from=incoming,
			}
			]
			\pic[
			tqft/pair of pants,
			name=a,
			cobordism height=0.75cm,
			every incoming boundary component/.style={draw},
			every outgoing boundary component/.style={draw},
			cobordism edge/.style={draw},
			];
			\pic[
			tqft/cylinder,
			name=b,
			cobordism height=0.75cm,
			every incoming boundary component/.style={draw},
			every outgoing boundary component/.style={draw},
			cobordism edge/.style={draw},
			at={($ (a-outgoing boundary)+(0,0.5)$) },
			];
			\pic[
			tqft/cup,
			name=c,
			cobordism height=1cm,
			every incoming boundary component/.style={draw},
			every outgoing boundary component/.style={draw},
			cobordism edge/.style={draw},
			at={($ (a-outgoing boundary)+(0,0)$) },
			];
			\draw (b-outgoing boundary)+(0.5,-0.25) node {\textup{=}};
			\pic[
			tqft/cylinder,
			name=e,
			cobordism height=0.75cm,
			every incoming boundary component/.style={draw},
			every outgoing boundary component/.style={draw},
			cobordism edge/.style={draw},
			at={($ (c-outgoing boundary)+(0.75,-0.25) $) },
			];
			\draw (e-outgoing boundary)+(0.5,0) node {\textup{=}};
			\pic[
			tqft/pair of pants,
			name=ar,
			cobordism height=0.75cm,
			every incoming boundary component/.style={draw},
			every outgoing boundary component/.style={draw},
			cobordism edge/.style={draw},
			at={($ (c-outgoing boundary)+(2.5,-0.25)$) },
			];
			\pic[
			tqft/cylinder,
			name=br,
			cobordism height=0.75cm,
			every incoming boundary component/.style={draw},
			every outgoing boundary component/.style={draw},
			cobordism edge/.style={draw},
			at={($ (ar-outgoing boundary)+(0,0)$) },
			];
			\pic[
			tqft/cup,
			name=cr,
			cobordism height=1cm,
			every incoming boundary component/.style={draw},
			every outgoing boundary component/.style={draw},
			cobordism edge/.style={draw},
			at={($ (ar-outgoing boundary)+(0,0.5)$) },
			];
		\end{tikzpicture}
		\label{xn8lyg5e}\\
		&\begin{tikzpicture}[
			every tqft/.append style={
				transform shape, rotate=90, tqft/circle x radius=4pt,
				tqft/circle y radius=2pt,
				tqft/boundary separation=0.5cm, tqft/view from=incoming,
			}
			]
			\pic[
			tqft/cylinder to prior,
			name=a,
			cobordism height=0.75cm,
			every incoming boundary component/.style={draw},
			every outgoing boundary component/.style={draw},
			cobordism edge/.style={draw},
			boundary separation=1cm,
			at={($ (c-outgoing boundary)+(-0.75,0.75) $)},
			];
			\pic[
			tqft/cylinder to next,
			name=b,
			cobordism height=0.75cm,
			every incoming boundary component/.style={draw},
			every outgoing boundary component/.style={draw},
			cobordism edge/.style={draw},
			boundary separation=1cm,
			at={($ (c-outgoing boundary)+(-0.75,0.25) $) },
			];
			\pic[
			tqft/reverse pair of pants,
			name=c,
			cobordism height=0.75cm,
			every incoming boundary component/.style={draw},
			every outgoing boundary component/.style={draw},
			cobordism edge/.style={draw},
			at={($ (a-outgoing boundary)+(0,0) $) },
			];
			\draw (c-outgoing boundary)+(0.5,0)  node {\textup{=}};
			\pic[
			tqft/reverse pair of pants,
			name=d,
			cobordism height=0.75cm,
			every incoming boundary component/.style={draw},
			every outgoing boundary component/.style={draw},
			cobordism edge/.style={draw},
			at={($(c-outgoing boundary)+(1,-0.25)$) },
			];
		\end{tikzpicture}
		\label{xbfqd14f}\\
		&\begin{tikzpicture}[
			every tqft/.append style={
				transform shape, rotate=90, tqft/circle x radius=4pt,
				tqft/circle y radius=2pt,
				tqft/boundary separation=0.5cm, tqft/view from=incoming,
			}
			]
			\pic[
			tqft/pair of pants,
			name=a,
			cobordism height=0.75cm,
			every incoming boundary component/.style={draw},
			every outgoing boundary component/.style={draw},
			cobordism edge/.style={draw},
			at={($ (a-outgoing boundary)+(0,0) $) },
			];
			\pic[
			tqft/cylinder to prior,
			name=b,
			cobordism height=0.75cm,
			every incoming boundary component/.style={draw},
			every outgoing boundary component/.style={draw},
			cobordism edge/.style={draw},
			boundary separation=1cm,
			at={($ (a-outgoing boundary)+(0,0.5) $)},
			];
			\pic[
			tqft/cylinder to next,
			name=c,
			cobordism height=0.75cm,
			every incoming boundary component/.style={draw},
			every outgoing boundary component/.style={draw},
			cobordism edge/.style={draw},
			boundary separation=1cm,
			at={($ (a-outgoing boundary)+(0,0) $) },
			];
			\draw (b-outgoing boundary)+(0.5,0.25)  node {\textup{=}};
			\pic[
			tqft/pair of pants,
			name=d,
			cobordism height=0.75cm,
			every incoming boundary component/.style={draw},
			every outgoing boundary component/.style={draw},
			cobordism edge/.style={draw},
			at={($(c-outgoing boundary)+(1,-0.25)$) },
			];
		\end{tikzpicture}
		\label{3p922agk}\\
		&\begin{tikzpicture}[
			every tqft/.append style={
				transform shape, rotate=90, tqft/circle x radius=4pt,
				tqft/circle y radius=2pt,
				tqft/boundary separation=0.5cm, tqft/view from=incoming,
			}
			]
			\pic[
			tqft/pair of pants,
			name=a,
			cobordism height=0.75cm,
			every incoming boundary component/.style={draw},
			every outgoing boundary component/.style={draw},
			cobordism edge/.style={draw},
			];
			\pic[
			tqft/cylinder to prior,
			name=b,
			cobordism height=0.75cm,
			every incoming boundary component/.style={draw},
			every outgoing boundary component/.style={draw},
			cobordism edge/.style={draw},
			at={($(a-outgoing boundary)+(-0.75,-0.25) $) },
			];
			\pic[
			tqft/cylinder to prior,
			name=c,
			cobordism height=0.75cm,
			every incoming boundary component/.style={draw},
			every outgoing boundary component/.style={draw},
			cobordism edge/.style={draw},
			at={($ (a-outgoing boundary)+(0,0.5) $) },
			];
			\pic[
			tqft/reverse pair of pants,
			name=d,
			cobordism height=0.75cm,
			every incoming boundary component/.style={draw},
			every outgoing boundary component/.style={draw},
			cobordism edge/.style={draw},
			at={($ (a-outgoing boundary)+(0,-0.5) $) },
			];
			\draw (c-outgoing boundary)+(0.5,-0.25)  node {\textup{=}};
			\pic[
			tqft/reverse pair of pants,
			name=e,
			cobordism height=0.75cm,
			every incoming boundary component/.style={draw},
			every outgoing boundary component/.style={draw},
			cobordism edge/.style={draw},
			at={($(a-outgoing boundary)+(1.75,-0.25) $) },
			];
			\pic[
			tqft/pair of pants,
			name=f,
			cobordism height=0.75cm,
			every incoming boundary component/.style={draw},
			every outgoing boundary component/.style={draw},
			cobordism edge/.style={draw},
			at={($(e-outgoing boundary)+(0,0) $) },
			];
			\draw (f-outgoing boundary)+(0.5,0.25)  node {\textup{=}};
			\pic[
			tqft/pair of pants,
			name=ar,
			cobordism height=0.75cm,
			every incoming boundary component/.style={draw},
			every outgoing boundary component/.style={draw},
			cobordism edge/.style={draw},
			at={($(f-outgoing boundary)+(1,0) $) },
			];
			\pic[
			tqft/cylinder to next,
			name=br,
			cobordism height=0.75cm,
			every incoming boundary component/.style={draw},
			every outgoing boundary component/.style={draw},
			cobordism edge/.style={draw},
			at={($(ar-outgoing boundary)+(-0.75,0.75) $) },
			];
			\pic[
			tqft/reverse pair of pants,
			name=cr,
			cobordism height=0.75cm,
			every incoming boundary component/.style={draw},
			every outgoing boundary component/.style={draw},
			cobordism edge/.style={draw},
			at={($(br-outgoing boundary)+(0,-0.5) $) },
			];
			\pic[
			tqft/cylinder to next,
			name=cr,
			cobordism height=0.75cm,
			every incoming boundary component/.style={draw},
			every outgoing boundary component/.style={draw},
			cobordism edge/.style={draw},
			at={($(ar-outgoing boundary)+(0,0) $) },
			];
		\end{tikzpicture}
		\label{vjvmwnz4}
	\end{align}
	and additional relations capturing the fact that the cylinder $\iota$ is the identity.
	 A 2-dimensional topological quantum field theory valued $\mathbf{C}$ is thereby equivalent to the following: a choice of object $X \in \mathbf{C}$ and four morphisms
	\begin{align*}
		X_\eta     &: I \to X && \text{(unit)},\\
		X_\mu      &: X \otimes X \to X && \text{(multiplication)}, \\
		X_\delta   &: X \to X \otimes X && \text{(comultiplication)}, \\
		X_\epsilon &: X \to I && \text{(counit)},
	\end{align*}
	subject to the relations indicated by \eqref{wzw7o5ta}--\eqref{vjvmwnz4}, where the cylinder $\iota$ is mapped to the identity $\Id_X : X \longrightarrow X$ and the twist $\tau$ is mapped to the braiding $B_X : X \otimes X \longrightarrow X \otimes X$ \cite[Theorem 3.6.17]{koc:04}.
	For example, \eqref{xbfqd14f} translates to the requirement that $X_\mu \circ B_X = X_\mu$. 
	A tuple $(X,X_{\eta},X_{\mu},X_{\delta},X_{\epsilon})$ of an object $X$ in $\mathbf{C}$ and morphisms $X_\eta, X_\mu, X_\delta, X_\epsilon$ satisfying the analogues of \eqref{wzw7o5ta}--\eqref{vjvmwnz4} in $\mathbf{C}$ is called a \defn{commutative Frobenius object in $\mathbf{C}$}.

	\subsection{The case of abelian symplectic groupoids}
	The following notion will be instrumental to constructing commutative Frobenius objects in $\WS$.
	
	\begin{definition}\label{rjg7zxym}
		A groupoid is \defn{abelian} if $\sss = \ttt$ and all isotropy groups are abelian.
	\end{definition}
	
	Our strategy is to first construct a commutative Frobenius object in $\WS$ for each abelian symplectic groupoid, and subsequently use Morita transfer to get a TQFT for every abelianizable quasi-symplectic groupoid.
	
	Let $(\A, \omega) \tto N$ be an abelian symplectic groupoid.
	The multiplication morphism $\A_\mu \in \Hom_{\WS}(\A \times \A, \A)$ is given by multiplication in $\A$, i.e.\ we define $\A_\mu$ to be the set $\A * \A \coloneqq \A \fp{\sss}{\ttt} \A$ of composable arrows, together with the morphisms
	\begin{equation}\label{4rsl53np}
		\begin{tikzcd}[row sep={3em,between origins},column sep={4em,between origins}]
			& \A * \A \arrow[hook']{dl} \arrow{dr}{\mmm} & \\
			\A \times \A & & \A.
		\end{tikzcd}
	\end{equation}
	Since $\A$ is abelian, the set $\A * \A$ is a Lie groupoid over $N$.
	The multiplicativity of $\omega$ implies that $\A_\mu$ is a 1-shifted Lagrangian correspondence from $\A \times \A$ to $\A$, with respect to the trivial 2-form on $N$.
	Similarly, we define comultiplication $\A_\delta \in \Hom_{\WS}(\A, \A \times \A)$ via
	\[
	\begin{tikzcd}[row sep={3em,between origins},column sep={4em,between origins}]
		& \A * \A \arrow[swap]{dl}{\mmm} \arrow[hook]{dr} & \\
		\A & & \A \times \A.
	\end{tikzcd}
	\]
	The unit $\A_\eta \in \Hom_{\WS}(\star, \A)$ is given by the identity section of $\A$, i.e.\ we define $\A_\eta$ to be the trivial groupoid $N \tto N$, together with the morphisms
	\[
	\begin{tikzcd}[row sep={3em,between origins},column sep={4em,between origins}]
		& N \arrow{dl} \arrow[hook]{dr}{\uuu} & \\
		\star & & \A.
	\end{tikzcd}
	\]
	It follows that $\A_\eta$ is a 1-shifted Lagrangian correspondence with respect to the trivial 2-form on $N$.
	Similarly, the counit $\A_\epsilon \in \Hom_{\WS}(\A, \star)$ is the trivial groupoid over $N$, together with
	\[
	\begin{tikzcd}[row sep={3em,between origins},column sep={4em,between origins}]
		& N \arrow[hook',swap]{dl}{\uuu} \arrow{dr} & \\
		\A & & \star.
	\end{tikzcd}
	\]
	
	We now verify that $\A_\eta, \A_\mu, \A_\delta, \A_\epsilon$ satisfy the relations \eqref{wzw7o5ta}--\eqref{vjvmwnz4} defining a commutative Frobenius object in $\WS$.
	Adopt the notation $\A^{*n} \coloneqq \A * \cdots * \A$ ($n$ times) for the set of $n$-composable arrows.
	
	\begin{lemma}[Sewing in discs]\label{9cx290iq}
		Relations \eqref{wzw7o5ta} and \eqref{xn8lyg5e} hold in $\WS$ for $\A_\eta, \A_\mu, \A_\delta, \A_\epsilon$. For the first identity in \eqref{wzw7o5ta}, we get that the 1-shifted Lagrangian correspondences
		\begin{equation}\label{rrcfu3yn}
			\begin{tikzcd}[row sep={3em,between origins},column sep={4em,between origins}]
				& N \times \A \arrow[swap]{dl}{\pr_\A} \arrow[hook]{dr}{\uuu \times \Id} & & \A * \A \arrow[hook']{dl} \arrow{dr}{\mmm} & \\
				\A & & \A \times \A & & \A
			\end{tikzcd}
		\end{equation}
		are transverse, and that their composition is weakly equivalent to the identity $\A \longleftarrow \A \longrightarrow \A$.
		Similar statements hold for the other three identities.
	\end{lemma}
	
	\begin{proof}
		We verify the first identity in \eqref{wzw7o5ta} only; the others are handled similarly.
		Transversality in \eqref{rrcfu3yn} is the statement that the maps
		\begin{equation}\label{swep5mxs}
			\begin{tikzcd}[row sep=0em]
				N^2 \times N \arrow{r} & N^2 \times N^2 & \A^2 \arrow{l} \\
				(x, y, z) \arrow[mapsto]{r} & (x, y, z, z) & \\
				& (\sss(a), \sss(b), \sss(a), \sss(b)) & (a, b) \arrow[mapsto]{l} \\
			\end{tikzcd}
		\end{equation}
		are transverse.
		We show that the intersection is weakly equivalent to the strong fibre product via Proposition \ref{blagblt1}.
		Strong transversality amounts to the maps
		\[
		\begin{tikzcd}[row sep=0em]
			N \times \A \arrow{r} & \A^2 & \A^{*2} \arrow{l} \\
			(x, a) \arrow[mapsto]{r} & (x, a), (b, c) & (b, c) \arrow[mapsto]{l}
		\end{tikzcd}
		\]
		being transverse; this is clear.
		Moreover, the map \eqref{nvwikwcm} is the map
		\begin{align*}
			\A^{*3} &\too \A^{*2} \\
			(a, b, c) &\mtoo (b^{-1}, c^{-1}a)
		\end{align*}
		which is also clearly a surjective submersion.
		The composition in \eqref{rrcfu3yn} is then weakly equivalent to the strong fibre product $(N \times \A) \stimes_{\A^2} (\A * \A) \cong \A$.
		%
		%
	\end{proof}
	
	\begin{lemma}[Commutativity and cocommutativity]\label{4900zjwk}
		Relations \eqref{xbfqd14f} and \eqref{3p922agk} hold in $\WS$ for $\A_\eta, \A_\mu, \A_\delta, \A_\epsilon$.
		More precisely, \eqref{xbfqd14f} is the statement that the 1-shifted Lagrangian correspondences
		\begin{equation}\label{tts3jim2}
			\begin{tikzcd}[row sep={3em,between origins},column sep={4em,between origins}]
				& \A \times \A \arrow[swap]{dl}{\Id} \arrow{dr}{\mathrm{swap}} & & \A * \A \arrow[hook']{dl} \arrow{dr}{\mmm} & \\
				\A \times \A & & \A \times \A & & \A
			\end{tikzcd}
		\end{equation}
		are transverse, and their composition is weakly equivalent to \eqref{4rsl53np}.
		A similar statement holds for \eqref{3p922agk}.
	\end{lemma}
	
	\begin{proof}
		We verify \eqref{xbfqd14f} only; the case of \eqref{3p922agk} is similar.
		Transversality in \eqref{tts3jim2} amounts to the fact that the maps in \eqref{swep5mxs} are transverse.
		Strong transversality is also clear.
		With a view to applying Proposition \ref{blagblt1}, we note that \eqref{nvwikwcm} corresponds to the surjective submersion
		\begin{align*}
			\A^{*4} &\too \A^{*2} \\
			(a, b, c, d) &\mtoo (c^{-1}b, d^{-1}a).
		\end{align*}
		We conclude that the composition in \eqref{tts3jim2} is weakly equivalent to the strong fibre product $(\A \times \A) \stimes_{\A^2} (\A * \A) = \{((a, b), (c, d)) \in (\A \times \A) \times (\A * \A) : (b, a) = (c, d)\}$.
		The latter is isomorphic to $\A * \A$, together with the maps $\A \times \A \leftarrow \A * \A \to \A$, $(a, b) \mapsfrom (a, b) \mapsto ba = ab$, i.e.\ to $\A_\mu$.
	\end{proof}
	
	\begin{lemma}[Frobenius relations]\label{jycyn8qs}
		Relations \eqref{vjvmwnz4} hold in $\WS$ for $\A_\eta, \A_\mu, \A_\delta, \A_\epsilon$.
		More precisely, the first relation in \eqref{vjvmwnz4} is the statement that the pairs of 1-shifted Lagrangian correspondences
		\begin{equation}\label{cy8fojir}
			\begin{tikzcd}[row sep={3em,between origins},column sep={4em,between origins}]
				& (\A * \A) \times \A \arrow[swap]{dl}{\mmm \times \Id} \arrow[hook]{dr} & & \A \times (\A * \A) \arrow[hook']{dl} \arrow{dr}{\Id \times \mmm} & \\
				\A \times \A & & \A \times \A \times \A & & \A \times \A
			\end{tikzcd}
		\end{equation}
		and
		\begin{equation}\label{8xxyokd5}
			\begin{tikzcd}[row sep={3em,between origins},column sep={4em,between origins}]
				& \A * \A \arrow[hook']{dl} \arrow{dr}{\mmm} & & \A * \A \arrow[swap]{dl}{\mmm} \arrow[hook]{dr} & \\
				\A \times \A & & \A & & \A \times \A
			\end{tikzcd}
		\end{equation}
		are both transverse, and that their compositions are weakly equivalent.
		A similar statement holds for the second identity.
	\end{lemma}
	
	\begin{proof}
		We show that both of these compositions are weakly equivalent to $\A^{2, 2}$ using Proposition \ref{blagblt1}. Let us begin with \eqref{cy8fojir}.
		Transversality is the statement that the maps
		\[
		\begin{tikzcd}[row sep=0em]
			N^2 \times N^2 \arrow{r} & N^3 \times N^3 & \A^3 \arrow{l} \\
			((x, y), (u, v)) \arrow[mapsto]{r} & ((x, x, y), (u, v, v)) & \\
			& ((\sss(a), \sss(b), \sss(c)), (\sss(a), \sss(b), \sss(c))) & (a, b, c) \arrow[mapsto]{l}
		\end{tikzcd}
		\]
		are transverse.
		Strong transversality is the statement that the maps
		\[
		\begin{tikzcd}[row sep=0em]
			\A^{*2} \times \A \arrow{r} & \A^3 & \A \times \A^{*2} \arrow{l} \\
			((a, b), c) \arrow[mapsto]{r} & (a, b, c) & (a, (b, c)) \arrow[mapsto]{l}
		\end{tikzcd}
		\]
		are transverse.
		The map \eqref{nvwikwcm} in Proposition \ref{blagblt1} corresponds to
		\begin{align*}
			\A^{*3} * \A^{*3} &\too \A^{*3} \\
			((a_1, b_1, c_1), (a_2, b_2, c_2)) &\mtoo (a_2^{-1}a_1, b_2^{-1}b_1, c_2^{-1}c_1),
		\end{align*}
		which is a surjective submersion.
		It follows that the composition in \eqref{cy8fojir} is weakly equivalent to the strong fibre product $((\A * \A) \times \A) \stimes_{\A^3} (\A \times (\A * \A))$.
		The latter can be identified with $\A * \A * \A$, together with the maps 
		\begin{equation}\label{1xp09ncc}
			\begin{tikzcd}[row sep={3em,between origins},column sep={4em,between origins}]
				& \A * \A * \A \arrow{dl} \arrow{dr} & \\
				\A \times \A & & \A \times \A
			\end{tikzcd}
			\qquad
			\begin{tikzcd}[row sep={3em,between origins},column sep={4em,between origins}]
				& (a, b, c) \arrow[mapsto]{dl} \arrow[mapsto]{dr} & \\
				(ab, c) & & (a, bc).
			\end{tikzcd}
		\end{equation}

		We now consider \eqref{8xxyokd5}.
		Transversality amounts to the maps
		\[
		\begin{tikzcd}
			N \times N \arrow{r}{\Id} & N \times N & \A \arrow{l}{(\sss, \sss)}
		\end{tikzcd}
		\]
		being transverse.
		Strong transversality amounts to the maps
		\[
		\begin{tikzcd}[row sep=0em]
			\A^{*2} \arrow{r} & \A & \A^{*2} \arrow{l} \\
			(a, b) \arrow[mapsto]{r} & ab & \\
			& cd & (c, d) \arrow[mapsto]{l}
		\end{tikzcd}
		\]
		being transverse.
		The map \eqref{nvwikwcm} in Proposition \ref{blagblt1} corresponds to
		\begin{align*}
			\A^{*4} &\too \A \\
			(a, b, c, d) &\mtoo abc^{-1}d^{-1},
		\end{align*}
		a surjective submersion.
		The composition in \eqref{8xxyokd5} is therefore weakly equivalent to the strong fibre product $(\A * \A) \stimes_{\A} (\A * \A) = \{(a_1, a_2, b_1, b_2) \in \A^{*4} : a_1 a_2 = b_1 b_2\}$.
		There is an isomorphism from \eqref{1xp09ncc} to the latter given by $(a, b, c) \mto (ab, c, a, bc)$.
	\end{proof}
	
	\begin{theorem}\label{9yp389oi}
		If $\A$ is an abelian symplectic groupoid, then $(\A,\A_{\eta},\A_{\mu},\A_{\delta},\A_{\epsilon})$ is a commutative Frobenius object in $\WS$. It thereby determines a $2$-dimensional TQFT $\Cob_2 \longrightarrow \WS$. \qed
	\end{theorem}
	
	\begin{proof}
	This is an immediate consequence of Lemmas \ref{9cx290iq}, \ref{4900zjwk}, and \ref{jycyn8qs}.
	\end{proof}
	
	\begin{remark}\label{vj7om4mk}
		By an induction argument, we see that for all $(m, n) \ne (0, 0)$, the TQFT in Theorem \ref{9yp389oi} maps the genus-0 cobordism from $m$ circles to $n$ circles to $\A^{m, n} \coloneqq \{(a, b) \in \A^{*m} * \A^{*n} : a_1 \cdots a_m = b_1 \cdots b_n\}$ together with the natural projections to $\A^m$ and $\A^n$. 
	\end{remark}
	
	\subsection{Abelianizations of quasi-symplectic groupoids}\label{Subsection: Abelianizable}
	
	We now consider the following definition.
	
	\begin{definition}
	An \defn{abelianization} of a quasi-symplectic groupoid $\G$ is an abelian symplectic groupoid $\A$, together with a symplectic Morita equivalence from $\A$ to $\G$. A quasi-symplectic groupoid is called \defn{abelianizable} if it admits an abelianization.
	\end{definition}
	
	Let $\G \tto M$ be a quasi-symplectic groupoid with an abelianization $\A \longleftarrow \H \longrightarrow \G$.
	By Morita transfer (see Subsection \ref{3y574s4t}), it follows that $\A_\eta, \A_\mu, \A_\delta, \A_\epsilon$ transfer in a unique way to morphisms in $\WS$ of the form $\G_\eta : \star \longrightarrow \G$, $\G_\mu : \G^2 \longrightarrow \G$, $\G_\delta : \G \longrightarrow \G^2$, and $\G_\epsilon : \G \longrightarrow \star$.
	For example, we can realize them as the homotopy fibre products of $\A_\eta$, $\A_\mu$, $\A_\delta$, and $\A_\epsilon$ with $\H^-$, $\H \times \times \H^-$, $\H \times \H^- \times \H^-$, and $\H$, respectively.
	Note that any other choice of abelianization $\A$ of $\G$ will give 1-shifted Lagrangians weakly equivalent to $\G_\eta, \G_\mu, \G_\delta, \G_\epsilon$.
	The corresponding morphisms in $\WS$ are therefore independent of the choice of abelianization.
	
	\begin{theorem}\label{jxre2dem}
		If $\G$ is an abelianizable quasi-symplectic groupoid, then $(\G,\G_\eta, \G_\mu, \G_\delta, \G_\epsilon)$ is a commutative Frobenius object in $\WS$.
		In this way, $\G$ determines a $2$-dimensional TQFT $\Cob_2 \longrightarrow \WS$.
	\end{theorem}
	
	\begin{proof}
		This follows from Proposition \ref{eiptq550} and Lemmas \ref{9cx290iq}, \ref{4900zjwk}, and \ref{jycyn8qs}.
	\end{proof}
	
	\subsection{Global slices}\label{Subsection: Global slice}
	
	We now discuss an application of Theorem \ref{9yp389oi}.
	
	\begin{definition}\label{Definition: Global slice}
	A \defn{global slice} to a Lie groupoid $\G\tto M$ is a submanifold $S\s M$ intersecting every orbit exactly once and transversely. The global slice is \defn{admissible} if for all $x \in S$, the isotropy group $\G_x$ is abelian.
	\end{definition}

\begin{remark}
Let $S\s X$ be an admissible global slice to a symplectic groupoid $\G\tto X$. The Lie algebra of $\G_x$ is necessarily abelian for all $x\in S$. In this way, asking the groups $\G_x$ to be abelian is a very mild constraint; it holds in many situations.   
\end{remark}
	
	\begin{proposition}\label{n7nca2ep}
	Let $(\G, \omega) \tto (M, \phi)$ be a quasi-symplectic groupoid. Suppose that $S\s M$ is an admissible global slice for which $i^*\phi=\mathrm{d}\gamma$ is exact, where $i:S\longrightarrow M$ is the inclusion. Then the restriction $\A \coloneqq \G\big\vert_S=\sss^{-1}(S)\cap\ttt^{-1}(S)\tto S$ is an abelianization of $\G$ with respect to $j^*\omega - \sss^*\gamma + \ttt^*\gamma$, where $j : \A \longrightarrow \G$ is the inclusion map.
	\end{proposition}
	
	\begin{proof}
		Note that $\A$ is the pullback of $\G$ by $i$.
		It follows that $\ttt \circ \pr_{\G} : S \fp{i}{\sss} \G \to M$ being a surjective submersion would force $\A$ to be a Lie groupoid. We would therefore like to prove that
		\[
		T_xS \fp{i_*}{\sss_*} T_g\G \too T_{\ttt(g)}M, \quad (v, w) \mtoo \ttt_*(w)
		\]
		is surjective for all $(x, g) \in S \fp{i}{\sss} \G$.
		In other words, we would like to show that the nullity of this map is $\dim(S \fp{i}{\sss} \G) - \dim M$.
		Its kernel is clearly isomorphic to $\{(v, w) \in T_xS \times (A_\G)_x : v = \rho(w)\}$.
		Since $T_{x}M = T_xS + \im \rho_{x}$, it has dimension $\dim S + \dim (A_\G)_x - \dim M = \dim(S \fp{i}{\sss} \G) - \dim M$.
		It follows that $\A$ is a Lie groupoid, and that the inclusion $j : \A \longrightarrow \G$ is an essential equivalence.
		
		It follows that $\A$ is a quasi-symplectic groupoid with respect to $(j^*\omega, i^*\phi)$.
		Since $i^*\phi = \mathrm{d}\gamma$ is exact, the gauge transformation \cite[Proposition 4.6]{xu:04} $(\A, j^*\omega - \sss^*\gamma + \ttt^*\gamma, 0)$ is a quasi-symplectic groupoid with trivial background 3-form symplectically Morita equivalent to $\G$.
		Moreover, since $S$ is a global slice, we have $\sss_\A = \ttt_\A$.
		It follows that $\rho_\A = 0$, so that $\A$ integrates the zero Poisson structure. This implies that $\A$ is indeed a symplectic groupoid.
		It follows that $(\A, j^*\omega - \sss^*\gamma + \ttt^*\gamma)$ is an abelianization of $(\G, \omega, \phi)$.
	\end{proof}
	
	It follows from Theorem \ref{jxre2dem} that $\G$ together with the slice $S$ determine a TQFT
	\begin{equation}\label{tqgl5jrm}
	\eta_\G : \Cob_2 \too \WS.
	\end{equation}

	Recall that the original Moore--Tachikawa TQFT associated to a complex semisimple group $G$ is characterised by the fact that $\begin{tikzpicture}[
		baseline=-2.5pt,
		every tqft/.append style={
			transform shape, rotate=90, tqft/circle x radius=4pt,
			tqft/circle y radius= 2pt,
			tqft/boundary separation=0.6cm, 
			tqft/view from=incoming,
		}
		]
		\pic[
		tqft/cup,
		name=d,
		every incoming lower boundary component/.style={draw},
		every outgoing lower boundary component/.style={draw},
		every incoming boundary component/.style={draw},
		every outgoing boundary component/.style={draw},
		cobordism edge/.style={draw},
		cobordism height= 1cm,
		];
	\end{tikzpicture}$ is sent to the Hamiltonian $G$-space $G \times \mathcal{S}$, where $\mathcal{S}$ is a Kostant slice.
	 This Hamiltonian space can also be described as $\sss^{-1}(\mathcal{S})$, where $\sss$ is the source map of the symplectic groupoid $T^*G \tto \g^*$.
	 Using the correspondence of Example \ref{cb2gi3le} between Hamiltonian spaces and 1-shifted Lagrangians, the TQFTs \eqref{tqgl5jrm} generalize the Moore--Tachikawa example in the following way.
	
	\begin{proposition}\label{wlgt0nx7}
		The morphism $\eta_\G(\begin{tikzpicture}[
		baseline=-2.5pt,
		every tqft/.append style={
			transform shape, rotate=90, tqft/circle x radius=4pt,
			tqft/circle y radius= 2pt,
			tqft/boundary separation=0.6cm, 
			tqft/view from=incoming,
		}
		]
		\pic[
		tqft/cup,
		name=d,
		every incoming lower boundary component/.style={draw},
		every outgoing lower boundary component/.style={draw},
		every incoming boundary component/.style={draw},
		every outgoing boundary component/.style={draw},
		cobordism edge/.style={draw},
		cobordism height= 1cm,
		];
	\end{tikzpicture}) = \G_\epsilon$ is the 1-shifted Lagrangian on $\G$ associated with the Hamiltonian $\G$-space $\sss^{-1}(S)$, where $\G$ acts by left multiplication.
	\end{proposition}
	
\begin{proof}
Let $\A \coloneqq \G|_S$ be the abelianization of $\G$.
Recall that $\A_\epsilon$ is the trivial groupoid $S \tto S$ viewed as a 1-shifted Lagrangian in $\A$.
A straightforward computation shows that $S \tto S$ is also a 1-shifted Lagrangian in $\G$.
Since the inclusion $\A \too \G$ is an essential equivalence, it follows that $\G_\epsilon$ is the 1-shifted Lagrangian $S \too \G$.
Recall from Example \ref{gesm9jcm} that the identity morphism $\G \longleftarrow \G \longrightarrow \G$ is weakly equivalent to the action groupoid $\G \longleftarrow (\G \times \G) \ltimes \G \longrightarrow \G$, where $\G \times \G$ acts on $\G$ by left and right multiplication.
Hence, by Lemma \ref{rw45fz6q}, the 1-shifted Lagrangian $S \too \G$ is weakly equivalent to its composition with $\G \longleftarrow (\G \times \G) \ltimes \G \longrightarrow \G$ on the left.
This composition is easily seen to be the 1-shifted Lagrangian $\G \ltimes \sss^{-1}(S) \too \G$.
\end{proof}

\section{The affinization process}\label{Section: Affinization}
In contrast to the preceding sections, we now work exclusively over $\mathbb{C}$. We first review some pertinent material from our manuscript on scheme-theoretic coisotropic reduction. This leads to a definition and study of Hartogs Morita morphisms. We next discuss a process by which to affinize $0$-shifted symplectic stacks and $1$-shifted Lagrangians. We then define what it means for an affine symplectic groupoid to be Hartogs abelianizable. By Main Theorem \ref{Theorem: Main 1}, such a groupoid determines a TQFT in $\WS$. 
On the other hand, we introduce the algebraic Moore--Tachikawa category $\AMT$ as well as an affinization process taking algebraic 1-shifted Lagrangian correspondences in $\WS$ to morphisms in $\AMT$. 
The proof of Main Theorem \ref{Theorem: Main 2} is subsequently given. To conclude, we prove that an affine symplectic groupoid admitting an admissible Hartogs slice is necessarily Hartogs abelianizable.

\subsection{Preliminaries} We briefly review some conventions from \cite{cro-may:24}. The term \defn{algebraic groupoid} is used for a groupoid object in the category of complex algebraic varieties.
An \defn{algebraic Lie groupoid} is an algebraic groupoid with smooth arrow and object varieties, as well smooth morphisms for its source and target maps.
By an \defn{affine algebraic Lie groupoid}, we mean an algebraic Lie groupoid whose object and arrow varieties are affine. Such a groupoid $\G$ is called an \defn{affine symplectic groupoid} if $\G$ carries an algebraic symplectic form for which the graph of multiplication is coisotropic in $\G\times\G\times\G^-$.

The \defn{affine quotient} associated to an algebraic groupoid $\G \tto X$ is the affine scheme
\[
X \sll{} \G \coloneqq \Spec \C[X]^\G,
\]
where $\C[X]$ is the algebra of morphisms $f : X \longrightarrow \C$, and $\C[X]^\G$ is the subalgebra of those satisfying $\sss^*f = \ttt^*f$. The \defn{pullback} of $\G \tto X$ by a morphism $\mu : Y \longrightarrow X$ is the fibre product
\[
\mu^*\G \coloneqq Y \fp{\mu}{\sss} \G \fp{\ttt}{\mu} Y.
\]
It is an algebraic groupoid over $Y$.
If $\G \tto X$ is an algebraic Lie groupoid and $\mu : Y \longrightarrow X$ is a smooth morphism, then $\mu^*\G$ is also an algebraic Lie groupoid.

\subsection{Morita and Hartogs Morita morphisms}
A \defn{Morita morphism} between algebraic Lie groupoids $\G \tto X$ and $\H \tto Y$ is a morphism of algebraic groupoids $(f, \mu) : (\H \tto Y) \longrightarrow (\G \tto X)$ for which $\mu : Y \longrightarrow X$ is a surjective smooth morphism, and the induced morphism
\[
\H \too \mu^*\G, \quad h \mtoo (\sss(h), f(h), \ttt(h))
\]
is an isomorphism. It is also advantegeous to introduce the weaker notion of a \defn{Hartogs Morita morphism} between affine algebraic Lie groupoids $\G \tto X$ and $\H \tto Y$; we define it to be morphism of algebraic groupoids $(f, \mu) : (\H \tto Y) \longrightarrow (\G \tto X)$ such that $\mu : Y \longrightarrow X$ is a smooth morphism, the open subset $U \coloneqq \im \mu \s X$ has a complement of codimension at least two in $X$, and the induced morphism
\[
\H \too \mu^*\G
\]
is an isomorphism.
The key property of Hartogs Morita morphisms is that they preserve affine quotients, as the next result shows.

\begin{proposition}\label{x2vodfb5}
	If $(\G \tto X) \longrightarrow (\H \tto Y)$ is a Hartogs Morita morphism, then the induced map $X \sll{} \G \longrightarrow Y \sll{} \H$ is an isomorphism of affine schemes.
\end{proposition}

\begin{proof}
	It suffices to consider the case in which $\H = \mu^*\G$ for $\G \tto X$ an affine algebraic Lie groupoid and $\mu : Y \longrightarrow X$ a smooth morphism whose image $U \coloneqq \im \mu$ has a complement of codimension at least two in $X$. Composing with $\mu$ then defines an injective algebra morphism $\C[X]^\G \longrightarrow \C[Y]^{\mu^*\G}$. We are reduced to showing that this morphism is surjective.
	
	Suppose that $f \in \C[Y]^{\mu^*\G}$.
	Since $\mu$ is smooth, there is an \'etale covering $\{\pi_i:U_i \longrightarrow U\}_{i\in I}$ with local sections $\{\sigma_i : U_i \longrightarrow Y\}_{i\in I}$ of $\mu$.
	We will constuct a function $F \in \C[X]^\G$ mapping to $f$ by considering the functions $F_i \coloneqq \sigma_i^*f : U_i \longrightarrow \C$, and subsequently using \'etale descent.
	To this end, fix $i,j\in I$, let $U_{ij} \coloneqq U_i \times_X U_j$, and write $p_1 : U_{ij} \longrightarrow U_i$ and $p_2 : U_{ij} \longrightarrow U_j$ for the projections. We have $$(p_1^*F_i)(x, y) = f(\sigma_i(x)) = f(\sigma_j(y)) = (p_2^*F_j)(x, y)$$
	for all $(x, y) \in U_{ij}$; the second equality follows from $\mu^*\G$-invariance of $f$, as $((\sigma_i(x), \sigma_j(y)), 1_{\pi_i(x)}) \in \mu^*\G$. We conclude that $p_1^* F_i = p_2^* F_j$.
	Since $\Hom(-, \C)$ is a sheaf in the \'etale topology, there is a function $F : U \longrightarrow \C$ satisfying $\pi_i^*F = F_i$ for all $i\in I$.
	Our assumption on codimension then implies that $F$ extends uniquely to an element $F \in \C[X]$.
	
	To see that $\mu^*F = f$, suppose that $y \in Y$.
	Since $\{\pi_i:U_i \longrightarrow U\}_{i\in I}$ is a cover, there exists a point $x$ in some $U_i$ with $\pi_i(x) = \mu(y)$.
	Note that $$(\mu^*F)(y) = F(\mu(y)) = F(\pi_i(x)) = F_i(x) = f(\sigma_i(x)) = f(y),$$ where the last equality uses the $\mu^*\G$-invariance of $f$ and the fact that $\mu(\sigma_i(x)) = \pi_i(x) = \mu(y)$. This establishes that $\mu^*F=f$.
	
	It remains only to establish that $F \in \C[X]^\G$. Given $g \in \G\big\vert_U\coloneqq\sss^{-1}(U)\cap\ttt^{-1}(U)$, we have $\sss(g) = \pi_i(x)$ and $\ttt(g) = \pi_j(y)$ for some $x \in U_i$ and $y \in U_j$. This implies that $(\sigma_i(x), \sigma_j(y), g) \in \mu^*\G$. By the $\mu^*\G$-invariance of $f$, we have $$F(\sss(g)) = F(\pi_i(x)) = f(\sigma_i(x)) = f(\sigma_j(y)) = F(\pi_j(y)) = F(\ttt(g)).$$ Since $\G\big\vert_U$ is open in $\G$, it follows that $F \in \C[X]^\G$, completing the proof.
\end{proof}

By replacing Morita morphisms with Hartogs Morita morphisms, we also obtain notions of \defn{Hartogs Morita equivalence}, \defn{Hartogs symplectic Morita equivalence}, and \defn{Hartogs Lagrangian Morita equivalence}.

\subsection{Affine Hamiltonian schemes}

An \defn{affine Poisson scheme} is the data of an affine scheme $X$ and a Poisson bracket on $\mathbb{C}[X]$. Recall that a closed subscheme of $X$ is called \defn{coisotropic} if its ideal is a Poisson subalgebra of $\mathbb{C}[X]$.

Let $\G \tto X$ be an affine symplectic groupoid acting algebraically on an affine Poisson scheme $M$, via a map $\mu : M \longrightarrow X$.
We say that the action is \defn{Hamiltonian} if the graph of the action morphism $\G \fp{\sss}{\mu} M \too M$ is coisotropic in $\G \times M \times M^-$.
We say that an affine Poisson scheme $M$ is an \defn{affine Hamiltonian $\G$-scheme} if it comes equipped with a Hamiltonian action of $\G$.
An \defn{isomorphism} of affine Hamiltonian $\G$-schemes is an equivariant isomorphism of affine Poisson schemes.


\subsection{Affinization of 0-shifted symplectic stacks}

Let $\G \tto X$ be an affine algebraic Lie groupoid.
A \defn{0-shifted symplectic structure} on $\G \tto X$ is a closed $2$-form $\omega$ of constant rank on $X$ satisfying $\sss^*\omega = \ttt^*\omega$ and $\ker \omega = \im \rho$, where $\rho$ is the anchor map.
We show that in this case, the affine quotient $X \sll{} \G$ is Poisson.

First note that we have a short exact sequence of vector bundles
\[
0 \too \ker \omega \too TX \too \im \omega \too 0.
\]
We define the Poisson structure explicitly as follows.
Note that for all $f \in \C[X]^\G$, $\mathrm{d}f \in (\im \rho)^\circ = \im \omega$.
Since every short exact sequence of vector bundles on an affine variety splits, we can choose a global vector field $X_f$ on $X$ such that $i_{X_f}\omega = \mathrm{d}f$, unique up to $\ker \omega$.
The Poisson bracket on $\C[X]^\G$ is then given by
\begin{equation}\label{911he928}
	\{f, g\} \coloneqq \omega(X_f, X_g).
\end{equation}

\begin{lemma}
	Equation \eqref{911he928} defines an affine Poisson scheme structure on $X \sll{} \G$.
\end{lemma}

\begin{proof}
	Let $R$ be a vector subbundle of $TX$ such that $TX = \im \rho \oplus R$.
	We then have an isomorphism $\omega : R \longrightarrow \im \omega = (\im \rho)^\circ$.
	It follows that every $\G$-invariant morphism $f : X \longrightarrow \C$ determines a unique section $X_f$ of $R$ that satisfies $\mathrm{d}f = \omega(X_f)$.
	The Poisson bracket $\{f, g\} = \omega(X_f, X_g)$ is independent of the choice of $R$, as any other choice gives vector fields differing from $X_f$ and $X_g$ by elements of $\ker \omega$.
	
	We need to show that $\{f, g\}$ is $\G$-invariant. To this end, fix $a \in \G$ and let $\hat{X}_f, \hat{X}_g \in T_a\G$ be mapped by $\ttt$ to $X_f$ and $X_g$, respectively.
	Then $$\sss^*(\omega(X_f)_{\sss(a)}) = (\sss^*\mathrm{d}f)_a = (\ttt^*\mathrm{d}f)_a = (\ttt^*(\omega(X_f)))_a = (\ttt^*\omega)_a(\hat{X}_f) = (\sss^*\omega)_a(\hat{X}_f) = \sss^*(\omega(\sss \hat{X}_f)_{\sss(a)}).$$
	It follows that $X_f - \sss\hat{X}_f \in \ker \omega$.
	We therefore have $$\sss^*\{f, g\}(a) = \omega_{\sss(a)}(X_f, X_g) = \omega_{\sss(a)}(\sss \hat{X}_f, \sss \hat{X}_g) = (\sss^*\omega)_a(\hat{X}_f, \hat{X}_g) = (\ttt^*\omega)_a(\hat{X}_f, \hat{X}_g) = \omega_{\ttt(a)}(X_f, X_g) = \ttt^*\{f, g\}(a).$$
	The Jacobi identity follows from the closedness of $\omega$.
\end{proof}

A \defn{symplectic Hartogs Morita morphism} between $0$-shifted symplectic affine algebraic Lie groupoids $\G \tto (X, \omega)$ and $\tilde{\G} \tto (\tilde{X}, \tilde{\omega})$ is a Hartogs Morita morphism $(f, \mu) : (\tilde{\G} \tto \tilde{X}) \to (\G \tto X)$ of the underlying affine algebraic Lie groupoids, such that $\mu^*\omega = \tilde{\omega}$.
We now observe that the affine Poisson scheme of a 0-shifted symplectic groupoid is invariant under symplectic Hartogs Morita morphisms.

\begin{proposition}\label{gln1n21i}
	Let $\G \tto X$ and $\tilde{\G} \tto \tilde{X}$ be 0-shifted affine symplectic groupoids. If there exists a symplectic Hartogs Morita morphism $\tilde{\G} \longrightarrow \G$, then the Poisson schemes $X \sll{} \G$ and $\tilde{X} \sll{} \tilde{\G}$ are isomorphic.
\end{proposition}

\begin{proof}
	Let $\mu : \tilde{X} \longrightarrow X$ be the map on objects coming from a Hartogs Morita morphism $\tilde{\G} \longrightarrow \G$.
	By Proposition \ref{x2vodfb5}, the map $\mu^* : \C[X]^\G \to \C[\tilde{X}]^{\tilde{\G}}$ is an algebra isomorphism. We also have $$\mu^*(i_{\mu(X_{\mu^*f})}\omega) = i_{X_{\mu^*f}} \mu^*\omega = i_{X_{\mu^*f}} \tilde{\omega} = \mathrm{d} \mu^*f = \mu^*\mathrm{d}f$$
	for all $f\in\mathbb{C}[X]$.
	Since $\mu$ is a submersion, we have $i_{\mu(X_{\mu^*f})} \omega = \mathrm{d}f$ for all $f\in\mathbb{C}[X]$. It follows that $\mu(X_{\mu^*f}) = X_f$ for all $f\in\mathbb{C}[X]$, implying that $\mu^*$ is also a Poisson algebra morphism.
\end{proof}

\begin{lemma}\label{4k6tno9w}
	Let $\G \tto (X, \omega)$ be an 0-shifted symplectic affine algebraic Lie groupoid and $Y \s X$ a smooth closed subvariety such that $TY^\omega = TY + \im \rho$.
	Let $I \s \C[X]$ be the ideal corresponding to $Y$ and set $J \coloneqq I \cap \C[X]^\G$.
	Then $J$ is a Poisson ideal and hence corresponds to a coisotropic subvariety of $X \sll{} \G$.
	The condition that $TY^\omega = TY + \im \rho$ holds, in particular, if $Y$ is isotropic, $\dim Y = \tfrac{1}{2} \dim X$, and $\dim(\ker \omega \cap TY) = \tfrac{1}{2} \dim \ker \omega$.
\end{lemma}

\begin{proof}
	Let $f, h \in J$.
	We have $X_f\big\vert_Y \in TY^\omega = TY + \im \rho$ and $\mathrm{d}h\big\vert_Y \in TY^\circ \cap (\im \rho)^\circ$, so that $$\{f, h\}\big\vert_Y = \omega(X_f, X_h)\big\vert_Y = -\mathrm{d}h(X_f)\big\vert_Y = 0.$$
	For the last statement, first note that $TY + \im \rho = TY + \ker \omega \s TY^\omega$ by isotropy of $Y$.
	We show that this inclusion is an equality by a dimension count. Indeed, we have $\dim(TY^\omega) = \dim X - \dim Y + \dim(\ker \omega \cap TY) = \dim Y + \dim \ker \omega - \dim(\ker \omega \cap TY)  = \dim(TY + \ker \omega)$.
\end{proof}

\subsection{Affinization of 1-shifted Lagrangians}
The following lemma is useful.
\begin{lemma}\label{95qjh5ie}
	Let $\G \tto X$ be an affine symplectic groupoid and $\H \tto Y$ an affine algebraic Lie groupoid.
	Suppose that $\G \times \H$ acts on a smooth affine variety $M$ with moment map $(\mu, \nu) : M \longrightarrow X \times Y$.
	Suppose also that the projection $(\G \times \H) \ltimes M \to \G$ has a 1-shifted Lagrangian structure $\omega \in \Omega_M^2$ for which $\ker \omega \s \ker \mu_*$ and the infinitesimal $\H$-action on $M$ has constant rank.
	Then $\omega$ is a 0-shifted symplectic form on $\H \ltimes M \tto M$, and the induced Poisson scheme $M \sll{} \H$ is a Hamiltonian $\G$-scheme.
\end{lemma}

\begin{proof}
	Let $\omega$ be the 1-shifted Lagrangian structure on $(\G \times \H) \ltimes M \longrightarrow \G$. Note that $\omega$ is a closed 2-form on $M$ with $\ttt^*\omega - \sss^*\omega = \pr_\G^*\Omega$, where $\Omega$ is the symplectic form on $\G$ and $(\sss, \ttt) : (\G \times \H) \ltimes M \tto M$ are the source and target maps. We also know that the map
	\begin{align*}
		\mu^*A_\G \oplus \nu^*A_\H &\too \{(v, a) \in TM \times A_\G : \mu_*v = \rho a \text{ and } i_v\omega = \mu^*\sigma_\Omega a\} \\
		(a, b) & \mtoo (\psi_*(a, b), a)
	\end{align*}
	is an isomorphism, where $\psi_*$ is the action map.
	Let $(\tilde{\sss}, \tilde{\ttt}) : \H \ltimes M \tto M$ be the source and target maps of the $\H$-action.
	Then $\tilde{\sss} = \sss \circ i$ and $\tilde{\ttt} = \ttt \circ i$, where $i : \H \ltimes M \to (\G \times \H) \ltimes M$ is given by $i(h, p) = (\uuu_{\mu(p)}, h, p)$.
	In particular, $\tilde{\ttt}^*\omega - \tilde{\sss}^*\omega = i^*(\ttt^*\omega - \sss^*\omega) = i^*\pr_\G^*\Omega = 0$, since $\pr_\G \circ i : \H \ltimes M \to \G$ has its image in the identity section.
	The non-degeneracy condition then shows that
	\[
	\ker \omega \cap \ker \mu_* = \im \rho_\H.
	\]
	for the anchor map $\rho_{\H}$ for $\H$. 
	Since we have $\ker \omega \s \ker \mu_*$ by assumption,
	\[
	\ker \omega = \im \rho_\H.
	\]
	It follows that $\omega$ is a $0$-shifted symplectic form, and hence descends to a Poisson bracket on $M \sll{} \H$.
	It remains to check that the residual action of $\G$ on $M \sll{} \H$ is Hamiltonian.
	By Lemma \ref{4k6tno9w}, it suffices to show that the graph $\Gamma \s \G \times M \times M^-$ of the action satisfies $\dim(\ker \eta \cap T\Gamma) = \tfrac{1}{2} \ker \eta$, where $\eta \coloneqq (\Omega, \omega, -\omega)$.
	But $\ker \eta = \{(0, u, v) : u, v \in \ker \omega\}$ and  $T\Gamma \cap \ker \eta = \{(0, v, v) : v \in \ker \omega\}$.
\end{proof}

The following generalization is natural and important.

\begin{theorem}[Affinization of 1-shifted Lagrangian correspondences]\label{sqyjpvje}
	Let $\G_1 \tto M_1$ and $\G_2 \tto M_2$ be affine symplectic groupoids and
	\begin{equation}\label{0uxd5xks}
		\begin{tikzcd}[row sep={3em,between origins},column sep={4em,between origins}]
			& \L \arrow[swap]{dl}{f_1} \arrow{dr}{f_2} \arrow[shift left=2pt]{d} \arrow[shift right=2pt]{d} & \\
			\G_1 \arrow[shift left=2pt]{d} \arrow[shift right=2pt]{d} & N \arrow{dl}{\mu_1} \arrow[swap]{dr}{\mu_2} & \G_2 \arrow[shift left=2pt]{d} \arrow[shift right=2pt]{d} \\
			M_1 & & M_2.
		\end{tikzcd}
	\end{equation}
	an affine 1-shifted Lagrangian correspondence.
	Then the affine quotient
	\begin{equation}\label{3eh6qc5x}
		\L\aff \coloneqq (\G_1 \fp{\sss}{\mu_1} N \fp{\mu_2}{\ttt} \G_2) \sll{} \L
	\end{equation}
	is a Hamiltonian $\G_1 \times \G_2^-$-scheme, where $\L$ acts via $l \cdot (g_1, n, g_2) = (g_1f_1(l)^{-1}, \ttt(l), f_2(l) g_2)$ if $\sss(l) = n$.
	If $\G_1 \longleftarrow \L' \longrightarrow \G_2$ is another affine 1-shifted Lagrangian correspondence that is Hartogs weakly equivalent to \eqref{0uxd5xks}, then $\L\aff$ and $(\L')\aff$ are isomorphic Hamiltonian $\G_1 \times \G_2^-$-schemes.
\end{theorem}

\begin{proof}
	Consider the 1-shifted Lagrangian correspondences
	\begin{equation}\label{iidsle4p}
		\begin{tikzcd}[row sep={3em,between origins},column sep={4em,between origins}]
			& (\G_1 \times \G_1) \ltimes \G_1 \arrow{dl} \arrow{dr} 
			& & \L \arrow{dl} \arrow{dr} & & (\G_2 \times \G_2) \ltimes \G_2 \arrow{dl} \arrow{dr} & \\
			\G_1 & & \G_1 & & \G_2 & & \G_2;
		\end{tikzcd}
	\end{equation}
	see Example \ref{gesm9jcm}.
	Their composition is the action groupoid
	\[
	\begin{tikzcd}[row sep={3em,between origins},column sep={4em,between origins}]
		& (\G_1 \times \L \times \G_2) \ltimes (\G_1 \fp{\sss}{\mu_1} N \fp{\mu_2}{\ttt} \G_2) \arrow{dl} \arrow{dr} & \\
		\G_1 & & \G_2
	\end{tikzcd}
	\]
	together with the 1-shifted Lagrangian structure given by the 2-form
	\[
	\omega \coloneqq \pr_{\G_1}^*\omega_1 + 2 \pr_N^*\gamma + \pr_{\G_2}^*\omega_2
	\]
	on $\G_1 \fp{\sss}{\mu_1} N \fp{\mu_2}{\ttt} \G_2$, where $\gamma$ is the 1-shifted Lagrangian structure on $\L$.
	
	To show that the affine quotient \eqref{3eh6qc5x} is a Hamiltonian $\G_1 \times \G_2^-$-scheme, we apply Lemma \ref{95qjh5ie}. A first step is to check that $\ker \omega \s \ker \mu_*$, where
	\[
	\mu : \G_1 \fp{\sss}{\mu_1} N \fp{\mu_2}{\ttt} \G_2 \too M_1 \times M_2, \qquad \mu(g_1, n, g_2) = (\ttt(g_1), \sss(g_2))
	\]
	is the moment map for the action of $\G_1 \times \G_2$ on $\G_1 \fp{\sss}{\mu_1} N \fp{\mu_2}{\ttt} \G_2$.
	Let $(v_1, u, v_2) \in T_{(g_1, n, g_2)}(\G_1 \fp{\sss}{\mu_1} N \fp{\mu_2}{\ttt} \G_2)$ be such that $(v_1, u, v_2) \in \ker \omega$.
	It follows that $v_1 \in (\ker \sss_*)^{\omega_1} = \ker \ttt_*$ and $v_2 \in (\ker \ttt_*)^{\omega_2} = \ker \sss_*$, so that $\mu_*(v_1, u, v_2) = 0$.
	A second step is to check that the infinitesimal $\L$-action on $\G_1 \fp{\sss}{\mu_1} N \fp{\mu_2}{\ttt} \G_2$ has constant rank.
	The kernel of this infinitesimal action is given by $\ker f_{1*} \cap \ker \rho_\L \cap \ker f_{2*}$; it is trivial by the definition of 1-shifted Lagrangians, i.e.\ the injectivity part of Definition \ref{yhan52mv}\ref{dhs45mwu}.
	By Lemma \ref{95qjh5ie}, \eqref{3eh6qc5x} is a Hamiltonian $\G_1 \times \G_2^-$-scheme.

	Consider another affine 1-shifted Lagrangian correspondence
	\begin{equation}\label{2qs7zlsf}
		\begin{tikzcd}[row sep={3em,between origins},column sep={4em,between origins}]
			& \L' \arrow[swap]{dl}{f_1'} \arrow{dr}{f_2'} \arrow[shift left=2pt]{d} \arrow[shift right=2pt]{d} & \\
			\G_1 \arrow[shift left=2pt]{d} \arrow[shift right=2pt]{d} & N' \arrow{dl}{\mu_1'} \arrow[swap]{dr}{\mu_2'} & \G_2 \arrow[shift left=2pt]{d} \arrow[shift right=2pt]{d} \\
			M_1 & & M_2.
		\end{tikzcd}
	\end{equation}
	that is Hartogs Morita equivalent to \eqref{0uxd5xks}.
	To show that our two correspondences yield isomorphic Hamiltonian schemes, it suffices to consider the following case: there is a Hartogs Morita morphism $\psi : \L' \longrightarrow \L$ fitting into a 2-commutative diagram
	\[
	\begin{tikzcd}[row sep={3em,between origins},column sep={3em,between origins}]
		& \L' \arrow[swap]{dl}{f_1'} \arrow{dr}{f_2'} \arrow{dd}{\psi} & \\
		\G_1 & & \G_2 \\
		& \L \arrow{ul}{f_1} \arrow[swap]{ur}{f_2} & 
	\end{tikzcd}
	\]
	with respect to some natural transformations $\theta_1 : f_1\psi \Longrightarrow f_1'$ and $\theta_2 : f_2\psi \Longrightarrow f_2'$.
	In this case, $\psi$ induces a Hartogs Morita morphism
	\begin{align*}
		\L' \ltimes (\G_1 \fp{\sss}{\mu_1'} N' \fp{\mu_2'}{\ttt} \G_2) 
		&\too
		\L \ltimes (\G_1 \fp{\sss}{\mu_1} N \fp{\mu_2}{\ttt} \G_2) \\
		(l, (g_1, n, g_2)) &\mtoo (\psi(l), (g_1 \theta_1(n), \psi(n), \theta_2(n)^{-1} g_2)).
	\end{align*}
	Proposition \ref{gln1n21i} implies that the induced Poisson schemes are isomorphic.
	Since the isomorphism is $\G_1 \times \G_2^-$ equivariant, this provides an isomorphism of Hamiltonian $\G_1 \times \G_2^-$-schemes.
\end{proof}

\subsection{The algebraic Moore--Tachikawa category}
We now 
define a category $\AMT$ in which the aforementioned affinizations of TQFTs take values. Our category turns out to enlarge Moore and Tachikawa's category $\MT$ of holomorphic symplectic varieties; see Section \ref{Section: Special case}. These considerations explain our choosing the notation $\AMT$. 

We begin with precise definitions. The objects of $\AMT$ are affine symplectic groupoids. To define morphisms, suppose that $\G$ and $\I$ are affine symplectic groupoids. Suppose also that $M$ and $N$ are affine Poisson schemes, equipped with commuting Hamiltonian actions of $\G$ and $\I^-$. Declare $M$ and $N$ to be \textit{isomorphic} if there exists an affine Poisson scheme isomorphism $M\longrightarrow N$ that intertwines the actions of $\G$ and $\I^-$. Morphisms from $\G$ to $\I$ are then defined to be isomorphism classes of affine Poisson schemes carrying commuting Hamiltonian actions of $\G$ and $\I^-$. To define morphism composition, consider affine symplectic groupoids $\G\tto X$, $\I\tto Y$, and $\mathcal{K}\tto Z$. Let us also consider $[M]\in\mathrm{Hom}(\G,\I)$ and $[N]\in\mathrm{Hom}(\I,\mathcal{K})$.
It follows that $\I \tto Y$ acts diagonally on the fibre product $M \times_Y N$.
By \cite[Subsection 4.3]{cro-may:24}, the affine quotient $(M \times_Y N) \sll{} \I$ is a Poisson scheme with commuting Hamiltonian actions of $\G$ and $\K^-$.
The Poisson structure is obtained by reduction, i.e. it is characterized by the relation $\{j^*f_1, j^*f_2\} = j^*\{f_1, f_2\}$ for all $f_1, f_2 \in \C[M \times N]$ such that $j^*f_1, j^*f_2 \in \C[M \times_Y N]^\I$, where $j : M \times_Y N \too M \times N$ is inclusion.


The previous paragraph and a routine exercise reveal that $\AMT$ is a category. It turns out that $\AMT$ also carries a symmetric monoidal structure. One defines the tensor product of affine symplectic groupoids $\G$ and $\I$ to be the affine symplectic groupoid $\G\otimes\I\coloneqq\G\times\I$. Given affine symplectic groupoids $\G$, $\I$, $\mathcal{K}$, and $\mathcal{L}$, and morphisms $[M]\in\mathrm{Hom}(\G,\I)$ and $[N]\in\mathrm{Hom}(\mathcal{K},\mathcal{L})$, note that $M\times N$ is an affine Hamiltonian $(\G\times\mathcal{K})\times(\mathcal{I}^- \times\mathcal{L}^-)$-scheme. One thereby has $$[M]\otimes[N]\coloneqq[M\times N]\in\mathrm{Hom}(\G\times\mathcal{K},\I\times\mathcal{L})$$ The unit object in $\AMT$ is the singleton groupoid $\{e\}\longrightarrow\{*\}$. The associator, left unitor, right unitor, and braiding are then exactly as one would expect. A routine exercise reveals that $\AMT$ is indeed a symmetric monoidal category.

\subsection{A ``functor"} Consider the sub-``category'' $\AWSQ$ of $\WSQ$ consisting of affine symplectic groupoids and affine algebraic 1-shifted Lagrangian correspondences.
We construct a ``functor''
\[
\AWSQ \too \AMT
\]
as follows.
Consider a morphism
\[
\begin{tikzcd}[row sep={2em,between origins},column sep={4em,between origins}]
	& \L \arrow{dl} \arrow{dr} & \\
	\G_1 & & \G_2.
\end{tikzcd}
\]
in $\AWSQ$.
Theorem \ref{sqyjpvje} implies that $\L\aff$ is a morphism from $\G_1$ to $\G_2$ in $\AMT$.
The next proposition shows that $\L \mto \L\aff$ behaves like a functor from $\AWSQ$ to $\AMT$.

\begin{proposition}\label{6x7y2609}
	The association $\L \mto \L\aff$ has the following properties.
	\begin{enumerate}[label={\textup{(\roman*)}}]
		\item\label{1afgzrs6}
		It is well-defined: if $\L_1$ and $\L_2$ are weakly equivalent 1-shifted Lagrangian correspondences from $\G_1$ to $\G_2$, then $\L_1\aff$ and $\L_2\aff$ are isomorphic Hamiltonian $\G_1 \times \G_2^-$-schemes.
		\item\label{yrvc6ed6}
		It preserves composition: if $\G_1 \longleftarrow \L_1 \longrightarrow \G_2$ and $\G_2 \longleftarrow \L_2 \longrightarrow \G_3$ are composable 1-shifted Lagrangian correspondences, then $(\L_2 \circ \L_1)\aff$ and $\L_2\aff \circ \L_1\aff$ are isomorphic Hamiltonian $\G_1 \times \G_3^-$-schemes.
		\item\label{venft4n2}
		It preserves identities: for every affine symplectic groupoid $\G$, the affinization of the identity $\G \longleftarrow \G \longrightarrow \G$ in $\WS$ is isomorphic to the identity of $\G$ in $\AMT$.
	\end{enumerate}
\end{proposition}

\begin{proof}
	Part \ref{1afgzrs6} follows from the last part of Theorem \ref{sqyjpvje}.
	For Part \ref{yrvc6ed6}, let $M_i$ and $N_i$ be the objects of $\G_i$ and $\L_i$, respectively.
	Let $\mu_{ij} : N_i \longrightarrow M_j$ be the base maps, and let $\mu_1 = (\mu_{11}, \mu_{12})$, $\mu_2 = (\mu_{22}, \mu_{23})$.
	We have a diagram of Lie groupoid morphisms
	\[
	\begin{tikzcd}[column sep={4em,between origins}, row sep={4em,between origins}]
		& & \L_2 \circ \L_1 \arrow[equal]{d} & & \\[-16pt]
		& & \L_1 \fp{\mu_{12}\sss}{\sss} \G_2 \fp{\ttt}{\mu_{22}\sss} \L_2 \arrow{dl} \arrow{dr} & & \\
		& \L_1 \arrow{dl} \arrow{dr} & & \L_2 \arrow{dl} \arrow{dr} \\
		\G_1 & & \G_2 & & \G_3
	\end{tikzcd}
	\]
	over the objects
	\[
	\begin{tikzcd}[column sep={4em,between origins}, row sep={4em,between origins}]
		& & N_1 \fp{\mu_{12}}{\sss} \G_2 \fp{\ttt}{\mu_{22}} N_2 \arrow[swap]{dl}{\pr_{N_1}} \arrow{dr}{\pr_{N_2}} & & \\
		& N_1 \arrow[swap]{dl}{\mu_{11}} \arrow{dr}{\mu_{12}} & & N_2 \arrow[swap]{dl}{\mu_{22}} \arrow{dr}{\mu_{23}} \\
		M_1 & & M_2 & & M_3.
	\end{tikzcd}
	\]
	Consider the map
	\[
	\begin{tikzcd}
		\L_2\aff \circ \L_1\aff \arrow[equal]{r} \arrow{d}
		& \Big((\G_1 \fp{\sss}{\mu_{11}} N_1 \fp{\mu_{12}}{\ttt} \G_2) \sll{} \L_1 \fp{\sss \circ \pr_{\G_2}}{\ttt \circ \pr_{\G_2}} (\G_2 \fp{\sss}{\mu_{22}} N_2 \fp{\mu_{23}}{\ttt} \G_3) \sll{} \L_2 \Big) \sll{} \G_2 \\
		(\L_2 \circ \L_1)\aff \arrow[equal]{r} 
		& \Big(\G_1 \fp{\sss}{\mu_{11} \circ \pr_{N_1}} (N_1 \fp{\mu_{12}}{\sss} \G_2 \fp{\ttt}{\mu_{22}} N_2) \fp{\mu_{23} \circ \pr_{N_2}}{\ttt} \G_3\Big) \sll{} (\L_1 \fp{\mu_{12} \circ \sss}{\sss} \G_2 \fp{\ttt}{\mu_{22} \circ \sss} \L_2)
	\end{tikzcd}
	\]
	which descends from the map
	\[
	\begin{tikzcd}
		(\G_1 \fp{\sss}{\mu_{11}} N_1 \fp{\mu_{12}}{\ttt} \G_2) \fp{\sss \circ \pr_{\G_2}}{\ttt \circ \pr_{\G_2}} (\G_2 \fp{\sss}{\mu_{22}} N_2 \fp{\mu_{23}}{\ttt} \G_3) \arrow{d} & (g_1, n_1, g_2, g_2', n_2, g_3) \arrow[mapsto]{d} \\
		\G_1 \fp{\sss}{\mu_{11} \circ \pr_{N_1}} (N_1 \fp{\mu_{12}}{\sss} \G_2 \fp{\ttt}{\mu_{22}} N_2) \fp{\mu_{23} \circ \pr_{N_2}}{\ttt} \G_3 & (g_1, n_1, (g_2g_2')^{-1}, n_2, g_3).
	\end{tikzcd}
	\]
	Note that this map descends from the morphism of action groupoids
	\[
	\begin{tikzcd}
		(\L_1 \times \G_2 \times \L_2) \ltimes \Big( (\G_1 \fp{\sss}{\mu_{11}} N_1 \fp{\mu_{12}}{\ttt} \G_2) \fp{\sss \circ \pr_{\G_2}}{\ttt \circ \pr_{\G_2}} (\G_2 \fp{\sss}{\mu_{22}} N_2 \fp{\mu_{23}}{\ttt} \G_3) \Big) \arrow{d} \\
		(\L_1 \fp{\mu_{12} \circ \sss}{\sss} \G_2 \fp{\ttt}{\mu_{22} \circ \sss} \L_2)
		\ltimes \Big(\G_1 \fp{\sss}{\mu_{11} \circ \pr_{N_1}} (N_1 \fp{\mu_{12}}{\sss} \G_2 \fp{\ttt}{\mu_{22}} N_2) \fp{\mu_{23} \circ \pr_{N_2}}{\ttt} \G_3\Big)
	\end{tikzcd}
	\]
	given by
	\[
	\begin{tikzcd}
		((l_1, a, l_2), (g_1, n_1, g_2, g_2', n_2, g_3)) \arrow[mapsto]{d} & \text{where $\sss(l_1) = n_1$, $\sss(l_2) = n_2$, and $\sss(a) = \sss(g_2) = \ttt(g_2')$} \\
		((l_1, (g_2 g_2')^{-1}, l_2), (g_1, n_1, (g_2 g_2')^{-1}, n_2, g_3)) & \text{where $\sss(l_1) = n_1$ and $\sss(l_2) = n_2$}.
	\end{tikzcd}
	\]
	One sees that this morphism of action groupoids is a Morita morphism, i.e.\ a pullback by a surjective smooth morphism. Moreover, the multiplicativity of the symplectic form on $\G_2$ implies that the base map preserves the 0-shifted symplectic 2-forms on both sides.
	Proposition \ref{gln1n21i} then implies that this morphism descends to an isomorphism of Poisson schemes $\L_2\aff \circ \L_1\aff \to (\L_2 \circ \L_1)\aff$.
	It is also clearly $\G_1 \times \G_3$-equivariant, and hence an isomorphism of Hamiltonian $\G_1 \times \G_3^-$-schemes.
	
	We now show Part \ref{venft4n2}. Note that the affinization of the identity of $\G$ in $\WS$ is the affine quotient $(\G \fp{\sss}{\ttt} \G) \sll{} \G$, where $\G$ acts by $g \cdot (a, b) = (ag^{-1}, gb)$; its Poisson structure is induced by the 0-shifted symplectic form on $\G \ltimes (\G \fp{\sss}{\ttt} \G) \tto \G \fp{\sss}{\ttt} \G$, i.e.\ the 2-form $\pr_1^*\omega + \pr_2^*\omega$ on $\G \fp{\sss}{\ttt} \G$. 
	At the same time, we have a Morita morphism from $\G \ltimes (\G \fp{\sss}{\ttt} \G) \tto \G \fp{\sss}{\ttt} \G$ to the trivial groupoid $\G \tto \G$; it is given by $(g, (a, b)) \mto ab$ on arrows and $\mmm : \G \fp{\sss}{\ttt} \G \to \G$ on objects.
	By the multiplicativity of $\omega$, this is a Morita morphism of 0-shifted symplectic affine algebraic Lie groupoids.
	Proposition \ref{gln1n21i} then shows that $(\G \fp{\sss}{\ttt} \G) \sll{} \G$ is isomorphic to $\G$ as a Hamiltonian $\G \times \G^-$-scheme.
\end{proof}

\subsection{Hartogs abelianizations of affine symplectic groupoids}\label{Subsection: Affinization}
The following definition is useful.

\begin{definition}
A \defn{Hartogs abelianization} of an affine symplectic groupoid $\G \tto X$ is an abelian affine symplectic groupoid $(\A, \omega_\A) \tto Y$, together with a Hartogs symplectic Morita equivalence from $\A$ to $\G$. We call $\G$ \textbf{Hartogs abelianizable} if it admits a Hartogs abelianization.
\end{definition}

The following is a more explicit statement of the definition. Let $(\G, \omega_\G) \tto X$ be an affine symplectic groupoid. A Hartogs abelianization of $\G$ consists of an abelian affine symplectic groupoid $(\A, \omega_\A) \tto Y$, an affine Lie groupoid $\H$ endowed with an algebraic closed 2-form $\gamma$ on $\H\obj$, and Hartogs Morita morphisms
\[
\begin{tikzcd}[row sep={2em,between origins}, column sep={3em,between origins}]
	& \H \arrow[swap]{dl}{\varphi} \arrow{dr}{\psi} & \\
	\A & & \G
\end{tikzcd}
\]
satisfying $\varphi^*\omega_\A - \psi^*\omega_\G = \ttt^*\gamma - \sss^*\gamma$.


As in Subsection \ref{Subsection: Abelianizable}, $\A_\eta$, $\A_\mu$, $\A_\delta$, and $\A_\epsilon$ transfer to affine 1-shifted Lagrangian correspondences
\begin{equation}\label{ozt3ju0f}
	\begin{tikzcd}[row sep={2em,between origins},column sep={3em,between origins}]
		& \G_\eta \arrow{dl} \arrow{dr} & \\
		\star & & \G
	\end{tikzcd}
	\qquad
	\begin{tikzcd}[row sep={2em,between origins},column sep={3em,between origins}]
		& \G_\mu \arrow{dl} \arrow{dr} & \\
		\G^2 & & \G
	\end{tikzcd}
	\qquad
	\begin{tikzcd}[row sep={2em,between origins},column sep={3em,between origins}]
		& \G_\delta \arrow{dl} \arrow{dr} & \\
		\G & & \G^2
	\end{tikzcd}
	\qquad
	\begin{tikzcd}[row sep={2em,between origins},column sep={3em,between origins}]
		& \G_\epsilon \arrow{dl} \arrow{dr} & \\
		\G & & \star
	\end{tikzcd}
\end{equation}
by taking the homotopy fibre products of $\A_\eta$, $\A_\mu$, $\A_\delta$, and $\A_\epsilon$ with $\H^-$, $\H \times \H \times \H^-$, $\H \times \H^- \times \H^-$, and $\H$, respectively.
Consider the corresponding morphisms $\G_\eta\aff$, $\G_\mu\aff$, $\G_\delta\aff$, and $\G_\epsilon\aff$ in $\AMT$ obtained via affinization.

\begin{theorem}\label{2vumxgkk}
If an affine symplectic groupoid $\G$ is Hartogs abelianizable, then $(\G,\G_\eta\aff,\G_\mu\aff,\G_\delta\aff,\G_\epsilon\aff)$ is a commutative Frobenius object in $\AMT$. In this way, $\G$ determines a TQFT $\Cob_2\longrightarrow\AMT$.
\end{theorem}

\begin{proof}
	Let $U \s X$ be the image of the base map of the Hartogs Morita morphism $\psi : \H \longrightarrow \G$. By assumption, $U$ is an open subset with a complement of codimension at least two in $X$. We also know that the correspondences \eqref{ozt3ju0f} restrict to 1-shifted Lagrangian correspondences on powers of $\G\big\vert_U\coloneqq\sss^{-1}(U)\cap\ttt^{-1}(U)$.
	Proposition \ref{eiptq550} and Lemma \ref{9cx290iq}, \ref{4900zjwk}, and \ref{jycyn8qs} now tell us that $\G\big\vert_U$ is a commutative Frobenius object in $\WS$ with respect to these restricted morphisms.
	The identity $\Id_{\G|_U}$ in $\WS$ can be identified with $\H$, and the braiding $B_{\G|_U}$ with $\H \times \H$.
	Let $\G_\iota \coloneqq \H$, viewed as a 1-shifted Lagrangian correspondence from $\G$ to $\G$.
	We similarly let $\G_\tau \coloneqq \H \times \H$, viewed as a 1-shifted Lagrangian correspondence from $\G \times \G$ to $\G \times \G$ via $(\psi(a), \psi(b)) \mapsfrom (a, b) \mapsto (\psi(b), \psi(a))$.
	It follows that $\G_\eta, \G_\mu, \G_\iota, \G_\delta, \G_\epsilon, \G_\tau$ satisfy the identities indicated by \eqref{wzw7o5ta}--\eqref{vjvmwnz4} viewed as 1-shifted Lagrangian correspondences on powers of $\G$.
	Proposition \ref{6x7y2609} then implies that their affinizations satisfy the same relations in $\AMT$.
	It therefore suffices to show that $\G_\iota\aff$ is the identity morphism from $\G$ to $\G$ in $\MT$ and, similarly, that $\G_\tau\aff$ is the braiding on $\G$ in $\AMT$.
	This follows from Proposition \ref{6x7y2609}\ref{venft4n2} and the last part of Theorem \ref{sqyjpvje}, as we have a Hartogs weak equivalence
	\[
\begin{tikzcd}[row sep={2em,between origins},column sep={3em,between origins}]
& \H \arrow{dl} \arrow{dr} \arrow{dd} & \\
\G & & \G \\
& \G \arrow{ul} \arrow{ur} &
\end{tikzcd}
	\]
	of 1-shifted Lagrangian correspondences from $\G$ to $\G$.
	%
\end{proof}

\subsection{Hartogs slices}

In analogy with Subsection \ref{Subsection: Global slice}, one might expect certain slices to induce Hartogs abelianizations of affine symplectic groupoids. This turns out to be true via the following algebro-geometric counterpart to Definition \ref{Definition: Global slice}.

\begin{definition}
A \defn{Hartogs slice} to an affine algebraic Lie groupoid $\G\tto X$ is a smooth closed affine subvariety $S \s X$ that is a global slice to $\G\big\vert_U$ for some open subset $U \s X$ whose complement has codimension at least 2.
The Hartogs slice is \defn{admissible} if the isotropy group $\G_x$ is abelian for all $x \in S$.
\end{definition}

\begin{proposition}\label{Proposition: Hartogs}
	Let $\G \tto X$ be an affine symplectic groupoid together with an admissible Hartogs slice $S \s X$.
	Then the restriction $\G\big\vert_S$ is a Hartogs abelianization of $\G$.
\end{proposition}

\begin{proof}
The proof is essentially the same as that of Proposition \ref{n7nca2ep}.
\end{proof}

Consider a Hartogs slice $S\s X$ to an affine symplectic groupoid $\G\tto X$. Note that $S$ is a Poisson transversal in $X$. Since the source map $\sss : \G \longrightarrow X$ is anti-Poisson, $\sss^{-1}(S) \s \G$ is a symplectic subvariety of $\G$. The action of $\G$ on $\sss^{-1}(S)$ by left multiplication is Hamiltonian. We may thereby view $\sss^{-1}(S)$ as a morphism $\G \to \star$ in $\AMT$.

\begin{theorem}\label{60ijaoc7}
	Let $\G \tto X$ be an affine symplectic groupoid with an admissible Hartogs slice $S \s X$.
	Write $F_\G : \Cob_2 \longrightarrow \AMT$ for the TQFT induced from Theorem \ref{2vumxgkk} by the Hartogs abelianization $\G\big\vert_S$.
	Then $F_\G$ maps $\begin{tikzpicture}[
		baseline=-2.5pt,
		every tqft/.append style={
			transform shape, rotate=90, tqft/circle x radius=4pt,
			tqft/circle y radius= 2pt,
			tqft/boundary separation=0.6cm, 
			tqft/view from=incoming,
		}
		]
		\pic[
		tqft/cup,
		name=d,
		every incoming lower boundary component/.style={draw},
		every outgoing lower boundary component/.style={draw},
		every incoming boundary component/.style={draw},
		every outgoing boundary component/.style={draw},
		cobordism edge/.style={draw},
		cobordism height= 1cm,
		];
	\end{tikzpicture}$ to the Hamiltonian $\G$-scheme $\sss^{-1}(S)$.
\end{theorem}

\begin{proof}
	By the proof of Proposition \ref{wlgt0nx7}, $F_\G(\begin{tikzpicture}[
		baseline=-2.5pt,
		every tqft/.append style={
			transform shape, rotate=90, tqft/circle x radius=4pt,
			tqft/circle y radius= 2pt,
			tqft/boundary separation=0.6cm, 
			tqft/view from=incoming,
		}
		]
		\pic[
		tqft/cup,
		name=d,
		every incoming lower boundary component/.style={draw},
		every outgoing lower boundary component/.style={draw},
		every incoming boundary component/.style={draw},
		every outgoing boundary component/.style={draw},
		cobordism edge/.style={draw},
		cobordism height= 1cm,
		];
	\end{tikzpicture})$ is the affinization of the 1-shifted Lagrangian $(S \tto S) \too (\G \tto X)$.
	It then follows directly from \eqref{3eh6qc5x} that this affinization is the Hamiltonian $\G$-scheme $\sss^{-1}(S)$.
\end{proof}

%
\begin{remark}
Consider a pair of integers $(m, n) \ne (0, 0)$.
One can describe the images of the genus-0 cobordism from $m$ circles to $n$ circles under the TQFT $F_\G$ of Theorem \ref{60ijaoc7}. To this end, let $\A \coloneqq \G\big\vert_S$.
As in Remark \ref{vj7om4mk}, the corresponding morphism from $\A^m$ to $\A^n$ in $\WS$ is $\A^{m, n} \coloneqq \{(a, b) \in \A^{*m} * \A^{*n} : a_1 \cdots a_m = b_1 \cdots b_n\}$, together with the projections to $\A^m$ and $\A^n$.
Since the inclusion $\A \longrightarrow \G\big\vert_U$ is an essential equivalence, the corresponding 1-shifted Lagrangian correspondences on powers of $\G$ are also given by $\A^{m, n}$ and the natural maps
\[
\begin{tikzcd}[row sep={2em,between origins},column sep={4em,between origins}]
	& \A^{m, n} \arrow{dl} \arrow{dr} & \\
	\G^m & & \G^n.
\end{tikzcd}
\]
Taking the affinizations of these morphisms, we get that $F_\G$ maps the genus-0 cobordism from $m$ circles to $n$ circles to the affine quotient 
\[
\{(g, h) \in \G^m \times \G^n : \sss(g_1) = \cdots = \sss(g_m) = \ttt(h_1) = \cdots = \ttt(h_n) \in S\} \sll{} \A^{m, n},
\]
where $\A^{m, n}$ acts by $(a, b) \cdot (g, h) = (ga^{-1}, bh)$.
The action of $\G^m \times (\G^-)^n$ giving it the structure of a Hamiltonian scheme descends from the action $(a, b) \cdot (g, h) = (ag, hb^{-1})$.
\end{remark}

\section{The special case of the Moore--Tachikawa conjecture}\label{Section: Special case}
This section is devoted to the implications of Main Theorem \ref{Theorem: Main 2} for constructing the Moore--Tachikawa TQFT. As with the previous section, we work exclusively over $\C$. We begin by recalling Moore and Tachikawa's category $\MT$ of holomorphic symplectic varieties with Hamiltonian actions. We then state the Moore--Tachikawa conjecture regarding the existence of a TQFT in $\MT$. This conjecture turns out to have a natural cousin in $\AMT$, i.e. a conjecture about the existence of a TQFT in $\AMT$. We deduce the latter conjecture as an immediate corollary of Main Theorem \ref{Theorem: Main 2}. 


\subsection{The Moore--Tachikawa category}\label{Subsection: MT category}
Moore and Tachikawa's conjectural TQFT would take values in the so-called \textit{category of holomorphic symplectic varieties with Hamiltonian actions} \cite{moo-tac:11}. We denote this category by $\MT$, and briefly recall its construction. 

The objects of $\MT$ are complex semisimple affine algebraic groups. A precise description of morphisms in $\MT$ hinges on the exact meaning of \textit{holomorphic symplectic variety}. In this context, one means an affine Poisson variety for which the Poisson structure is non-degenerate on an open dense subset of the smooth locus. A morphism in $\MT$ from $G$ to $I$ is an isomorphism class of holomorphic symplectic varieties carrying algebraic $G\times I$-actions, in such a way that the actions of $G=G\times\{e\}\s G\times I$ (resp. $I=\{e\}\times I\s G\times I$) are Poisson (resp. anti-Poisson). To define morphism composition, suppose that $[M]\in\Hom(G,I)$ and $[N]\in\Hom(I,K)$ for complex semisimple affine algebraic groups $G$, $I$, and $K$. One defines $$[N]\circ[M]\coloneqq[(M\times N^{-})\sll{0}I]\in\Hom(G,K).$$ By setting $G\otimes I\coloneqq G\times I$ and $[M]\otimes[N]=[M\times N]$, one can realize $\MT$ as a symmetric monoidal category.

We define a functor $\mathcal{F}:\MT\longrightarrow\AMT$ as follows. Given a complex semisimple affine algebraic group $G$, let $\F(G)$ be the cotangent groupoid $T^*G\tto\g^*$. Now suppose that $[M]\in\Hom(G,I)$ for complex semisimple affine algebraic groups $G$ and $I$. It follows that $M$ is an affine Hamiltonian $T^*G\times (T^*I)^{-}$-scheme. This fact allows us to let $\mathcal{F}([M])=[M]$, where the right-hand side is the isomorphism class of $M$ as an affine Hamiltonian $T^*G\times (T^*I)^{-}$-scheme. It follows that $\mathcal{F}:\MT\longrightarrow\AMT$ includes $\MT$ as a symmetric monoidal subcategory of $\AMT$.

\subsection{The Moore--Tachikawa conjecture}\label{Subsection: MT}
Let $G$ be a connected complex semisimple linear algebraic group with Lie algebra $\g$ and rank $\ell$. The Killing form determines an isomorphism between the adjoint and coadjoint representations of $G$; we denoted it by $(\cdot)^{\vee}:\g\longrightarrow\g^*$, $x\mapsto x^{\vee}$. Write $G_x\s G$ and $G_{\xi}\s G$ for the $G$-centralizers of $x\in\g$ and $\xi\in\g^*$, respectively. Their respective Lie algebras are denoted $\g_x\s\g$ and $\g_{\xi}\s\g$. It is known that $\dim\g_x\geq\ell$ (resp. $\dim\g_{\xi}\geq\ell$) for all $x\in\g$  (resp. $\xi\in\g^*$). The regular loci in $\g$ and $\g^*$ are the $G$-invariant open subvarieties given by
$$\g_{\text{reg}}\coloneqq\{x\in\g:\dim\g_x=\ell\}\quad\text{and}\quad\g^*_{\text{reg}}\coloneqq\{\xi\in\g^*:\dim\g_{\xi}=\ell\},$$
respectively. Using the $G$-equivariance of $(\cdot)^{\vee}:\g\longrightarrow\g^*$, one deduces that $\g^*_{\text{reg}}=(\g_{\text{reg}})^{\vee}$.

Let $G\times G$ act on $G$ via $(g_1,g_2)\cdot h=g_1hg_{2}^{-1}$. The cotangent lift of this action is a Hamiltonian $G\times G$-variety structure on $T^*G$. If we use the left trivialization to identify $T^*G$ with $G\times\g^*$, then this lifted action admits
$$(\mu_1,\mu_2):T^*G\longrightarrow\g^*\oplus\g^*,\quad (g,\xi)\mapsto(-\mathrm{Ad}_g^*(\xi),\xi)$$ as a moment map.

Let $(e,h,f)\in\g^{\times 3}$ be an $\mathfrak{sl}_2$-triple with $e,h,f\in\g_{\text{reg}}$, i.e.\ a principal $\mathfrak{sl}_2$-triple. This triple determines a Slodowy slice $\mathcal{S}\coloneqq e+\g_f$. One knows that $\mathcal{S}^{\vee}= (e+\g_{f})^{\vee}\s\g^*$ is a Poisson transversal in $\g^*$ \cite{gan-gin:02}. It follows that $G\times\mathcal{S}^{\vee}=\mu_2^{-1}(\mathcal{S}^{\vee})\s T^*G$ is a symplectic subvariety. We also observe that the Hamiltonian action of $G=G\times\{e\}\s G\times G$ on $T^*G$ preserves $G\times\mathcal{S}^{\vee}$. In this way, $G\times\mathcal{S}^{\vee}$ is a holomorphic symplectic variety with a Hamiltonian action of $G$. The Moore--Tachikawa conjecture is then stated as follows.
\begin{conjecture}[Moore--Tachikawa]\label{Conjecture: MT}
Let $G$ be a connected semisimple affine algebraic group with Lie algebra $\g$. Consider a principal $\mathfrak{sl}_2$-triple $(e,h,f)\in\g^{\times 3}$, and set $\mathcal{S}\coloneqq e+\g_f$. There exists a TQFT $\eta_G:\Cob_2\longrightarrow\MT$ satisfying $\eta_G(S^1)=G$ and $\eta_G(\begin{tikzpicture}[
	baseline=-2.5pt,
	every tqft/.append style={
		transform shape, rotate=90, tqft/circle x radius=4pt,
		tqft/circle y radius= 2pt,
		tqft/boundary separation=0.6cm, 
		tqft/view from=incoming,
	}
	]
	\pic[
	tqft/cup,
	name=d,
	every incoming lower boundary component/.style={draw},
	every outgoing lower boundary component/.style={draw},
	every incoming boundary component/.style={draw},
	every outgoing boundary component/.style={draw},
	cobordism edge/.style={draw},
	cobordism height= 1cm,
	];
\end{tikzpicture})=G\times\mathcal{S}^{\vee}$.
\end{conjecture}

We discuss the status of this conjecture in the next subsection.

\subsection{The Moore--Tachikawa conjecture in $\AMT$}\label{Subsection: MT in AMT}
Let $G$ be a connected complex semisimple affine algebraic group with Lie algebra $\g$. On the other hand, recall the functor $\mathcal{F}:\MT\longrightarrow\AMT$. Composing the conjectural TQFT $\eta_G:\Cob_2\longrightarrow\MT$ in Conjecture \ref{Conjecture: MT} with $\mathcal{F}$ yields the following conjecture.   

\begin{conjecture}[Moore--Tachikawa conjecture in $\AMT$]\label{Conjecture: Modified}
Let $G$ be a connected semisimple affine algebraic group with Lie algebra $\g$. Consider a principal $\mathfrak{sl}_2$-triple $(e,h,f)\in\g^{\times 3}$, and set $\mathcal{S}\coloneqq e+\g_f$. There exists a TQFT $\eta_{T^*G}:\Cob_2\longrightarrow\AMT$ satisfying $\eta_{T^*G}(S^1)=(T^*G\tto\g^*)$ and $\eta_{T^*G}(\begin{tikzpicture}[
	baseline=-2.5pt,
	every tqft/.append style={
		transform shape, rotate=90, tqft/circle x radius=4pt,
		tqft/circle y radius= 2pt,
		tqft/boundary separation=0.6cm, 
		tqft/view from=incoming,
	}
	]
	\pic[
	tqft/cup,
	name=d,
	every incoming lower boundary component/.style={draw},
	every outgoing lower boundary component/.style={draw},
	every incoming boundary component/.style={draw},
	every outgoing boundary component/.style={draw},
	cobordism edge/.style={draw},
	cobordism height= 1cm,
	];
\end{tikzpicture})=G\times\mathcal{S}^{\vee}$.
\end{conjecture}

With some effort, unpublished results of Ginzburg--Kazhdan \cite{gin-kaz:23} can be understood as implying Conjecture \ref{Conjecture: Modified}. Conjecture \ref{Conjecture: MT} then reduces to whether $\eta_{T^*G}:\Cob_2\longrightarrow\AMT$ takes values in the subcategory $\MT$, i.e. whether certain algebras are finitely generated. In Lie type $A$, results of Braverman--Finkelberg--Nakajima \cite{bra-fin-nak:19} imply that the relevant algebras are indeed finitely generated.

\subsection{Proof of Moore--Tachikawa conjecture in $\AMT$}
We now show Conjecture \ref{Conjecture: Modified} to be a straightforward consequence of our shifted symplecto-geometric approach. Retain the notation and objects used in the previous subsection.

\begin{lemma}\label{Lemma: Hartogs}
The subvariety $\mathcal{S}^{\vee}\s\g^*$ is an admissible Hartogs slice to $T^*G\tto\g^*$.
\end{lemma}

\begin{proof}
Write $G\cdot Y\s\g$ and $G\cdot Z\s\g^*$ for the $G$-saturations of $Y\s\g$ and $Z\s\g^*$, respectively. Our task is to verify the following:
\begin{itemize}
\item[\textup{(i)}] $G\cdot\mathcal{S}^{\vee}$ is open and has a complement of codimension at least two in $\g^*$;
\item[\textup{(ii)}] the restriction of $G\cdot \mathcal{S}^{\vee}\longrightarrow (G\cdot \mathcal{S}^{\vee})/G$ to $\mathcal{S}^{\vee}$ is a bijection;
\item[\textup{(iii)}] $\g^*=T_{\xi}(\mathcal{S}^{\vee})\oplus T_{\xi}(G\cdot \xi)$ for all $\xi\in\mathcal{S}^{\vee}$;
\item[\textup{(iv)}] $G_{\xi}$ is abelian for all $\xi\in\mathcal{S}^{\vee}$.
\end{itemize}
At the same time, recall that the isomorphism $(\cdot)^{\vee}:\g\longrightarrow\g^*$ is $G$-equivariant. Parts (i), (ii), and (iii) may therefore be rephrased as follows:
\begin{itemize}
\item[\textup{(i)}] $G\cdot\mathcal{S}$ is open and has a complement of codimension at least $2$ in $\g$;
\item[\textup{(ii)}] the restriction of $G\cdot S\longrightarrow (G\cdot\mathcal{S})/G$ to $\mathcal{S}$ is a bijection;
\item[\textup{(iii)}] $\g=T_{x}\mathcal{S}\oplus T_{x}(G\cdot x)$ for all $x\in\mathcal{S}$;
\item[\textup{(iv)}] $G_{x}$ is abelian for all $x\in\mathcal{S}$.
\end{itemize}
It is known that $G\cdot\mathcal{S}=\greg$, and that $\g\setminus\greg$ has codimension three in $\g$ \cite{kos:59,kos:63}. Parts (ii), (iii), and (iv) are then immediate and well-known implications of Kostant's work \cite{kos:59,kos:63}.
\end{proof}

\begin{theorem}\label{Theorem: Extended Moore-Tachikawa}
Conjecture \ref{Conjecture: Modified} is true.
\end{theorem}

\begin{proof}
Recall that the target morphism $\ttt:T^*G\longrightarrow\g^*$ is given by $\ttt(g,\xi)=\xi$. It follows that $\ttt^{-1}(\mathcal{S}^{\vee})=G\times\mathcal{S}^{\vee}$. Theorem \ref{60ijaoc7} and Lemma \ref{Lemma: Hartogs} then imply the desired result.
\end{proof}

\begin{remark}
There is a multiplicative counterpart to Theorem \ref{Theorem: Extended Moore-Tachikawa}. To obtain it, one replaces $T^*G$ and $\mathcal{S}^{\vee}\s\g^*$ with the quasi-Hamiltonian $G$-variety $D(G)\coloneqq G\times G$ and a Steinberg slice in $G$, respectively. This perspective is explored in the second named author's joint work with B\u{a}libanu \cite{bal-may:22}.
\end{remark}

\section{Examples involving Slodowy slices}\label{Section: Examples}
As with the previous two sections, we work exclusively over $\mathbb{C}$. It is tempting to wonder if Main Theorem \ref{Theorem: Main 2} gives examples of TQFTs that are fundamentally different from those conjectured by Moore--Tachikawa. Perhaps reassuringly, such examples exist. The current section is devoted to examples of this sort that involve Slodowy slices. We begin by addressing some Poisson-geometric properties of Slodowy slices. Particular attention is paid to the degeneracy locus of a Slodowy slice, i.e. the locus on which the Poisson bivector has non-maximal rank. The discussion subsequently turns to specific Slodowy slices. We restrict $T^*\SL_n\tto\mathfrak{sl}_n^*$ to a Slodowy slice to the minimal nilpotent orbit in $\mathfrak{sl}_n^*$. This restriction is shown to be Hartogs abelianizable. By Main Theorem \ref{Theorem: Main 2}, it determines a TQFT. 

\subsection{Some general comments}\label{Subsection: General}
Recall that the \textit{rank} of a smooth Poisson variety $X$ is the supremum of the dimensions of its symplectic leaves. Write $\mathrm{rk}\hspace{1pt}X$ for this quantity, and $X_{\text{reg}}\subset X$ for the union of the $\mathrm{rk}\hspace{1pt}X$-dimensional symplectic leaves of $X$. It turns out that $X\setminus X_{\text{reg}}$ is the vanishing locus of $\pi^{\frac{1}{2}\mathrm{rk}\hspace{1pt}X}$, where $\pi$ is the Poisson bivector field. This implies that $X_{\text{reg}}$ is an open subvariety of $X$. Let us call a symplectic leaf of $X$ \textit{regular} if it is contained in $X_{\text{reg}}$.

\subsection{Slodowy slices}
Let $G$ be a connected semisimple affine algebraic group with Lie algebra $\g$ and rank $\ell$. Suppose that $(e,h,f)\in\g^{\times 3}$ is an $\mathfrak{sl}_2$-triple with Slodowy slice $\mathcal{S}\coloneqq e+\g_f\subset\g$. Recall the definitions of $\greg$ and $\gdreg$ from Subsection \ref{Subsection: MT}. Note that the definition of $\gdreg$ coincides with the Poisson-theoretic one, obtained by setting $X=\g^*$ in Subsection \ref{Subsection: General}.

Let us call an adjoint (resp. coadjoint) orbit of $G$ \textit{regular} if it is contained in $\greg$ (resp. $\gdreg$).

\begin{lemma}\label{Lemma: Adjoint}
If $\mathcal{O}\s\g^*$ is a regular coadjoint orbit, then $\mathcal{O}\cap\mathcal{S}^{\vee}\neq\emptyset$.
\end{lemma}

\begin{proof}
	Consider the adjoint quotient $\pi:\g^*\longrightarrow\mathrm{Spec}(\mathbb{C}[\g^*]^G)\eqqcolon\mathfrak{c}$. The closure $\overline{\mathcal{O}}\subset\g^*$ is a fiber of $\pi$ \cite[Theorem 3]{kos:63}. On the other hand, the restriction $\pi\big\vert_{\mathcal{S}^{\vee}}:\mathcal{S}^{\vee}\longrightarrow\mathfrak{c}$ is known to be faithfully flat with irreducible fibers of dimension $\dim\mathcal{S}-\ell$ \cite{slo:80,slo2:80}. It follows that $\mathcal{S}^{\vee}\cap\overline{\mathcal{O}}$ is a non-empty, $(\dim\mathcal{S}-\ell)$-dimensional fiber of $\pi\big\vert_{\mathcal{S}^{\vee}}$. We also know that $\overline{\mathcal{O}}\setminus\mathcal{O}$ is a union of finitely many coadjoint orbits \cite[Theorem 3]{kos:63}, all of dimension strictly less than $\dim\mathcal{O}$. The intersection of each such orbit with $\mathcal{S}^{\vee}$ must therefore have dimension strictly less than $\dim\mathcal{O}+\dim\mathcal{S}-\dim\g=\dim\mathcal{S}-\ell$. Since $\mathcal{S}^{\vee}\cap\overline{\mathcal{O}}$ is non-empty and $(\dim\mathcal{S}-\ell)$-dimensional, we must have $\mathcal{S}^{\vee}\cap\mathcal{O}\neq\emptyset$.
\end{proof}

Recall the $G$-module isomorphism $(\cdot)^{\vee}:\g\longrightarrow\g^*$ discussed in Subsection \ref{Subsection: MT}. The subvariety $\mathcal{S}^{\vee}=(e+\g_f)^{\vee}$ is a Poisson transversal in $\g^*$ \cite{gan-gin:02}. As such, $\mathcal{S}^{\vee}$ is a Poisson variety. We may therefore consider the number $\mathrm{rank}\hspace{2pt}\mathcal{S}^{\vee}$ and locus $(\mathcal{S}^{\vee})_{\text{reg}}\s\mathcal{S}^{\vee}$. 

\begin{lemma}\label{Lemma: Intersections with regular orbits}
	We have $\mathrm{rank}\hspace{2pt}\mathcal{S}^{\vee}=\dim\mathcal{S}-\ell$ and $(\mathcal{S}^{\vee})_{\emph{reg}}=\mathcal{S}^{\vee}\cap\g^*_{\emph{reg}}$.	
\end{lemma}

\begin{proof}
	Suppose that $\xi\in\mathcal{S}^{\vee}$. Write $L_{\xi}\subset\mathcal{S}^{\vee}$ for the symplectic leaf of $\mathcal{S}^{\vee}$ containing $\xi$. It follows that $\dim T_{\xi}L_{\xi}=\dim (G\cdot\xi)+\dim\mathcal{S}-\dim\g$. This implies that $\dim T_{\xi}L_{\xi}\geq \dim T_{\eta}L_{\eta}$ for all $\eta\in\mathcal{S}^{\vee}$ if and only if $\dim (G\cdot\xi)\geq\dim(G\cdot\eta)$ for all $\eta\in\mathcal{S}^{\vee}$. A rephrased version is that $x\in(\mathcal{S}^{\vee})_{\text{reg}}$ if and only if $\dim\g_{\xi}\leq\dim\g_{\eta}$ for all $\eta\in\mathcal{S}^{\vee}$. Lemma \ref{Lemma: Adjoint} tells us that the latter condition holds if and only if $\xi\in\gdreg$. We conclude that $(\mathcal{S}^{\vee})_{\text{reg}}=\mathcal{S}^{\vee}\cap\g^*_{\text{reg}}$.
	
	Now suppose that $\xi\in (\mathcal{S}^{\vee})_{\text{reg}}=\mathcal{S}^{\vee}\cap\g^*_{\text{reg}}$. We have $$\dim T_{\xi}L_{\xi}=\dim (G\cdot\xi)+\dim\mathcal{S}-\dim\g=(\dim\g-\ell)+\dim\mathcal{S}-\ell=\dim\mathcal{S}-\ell.$$ It follows that $\mathrm{rank}\hspace{2pt}\mathcal{S}^{\vee}=\dim\mathcal{S}-\ell$. 
\end{proof}

\begin{proposition}\label{Proposition: Symplectic leaves in a Slodowy slice}
	The following statements are true.
	\begin{itemize}
		\item[\textup{(i)}] The regular symplectic leaves of $\mathcal{S}^{\vee}$ are the intersections of $\mathcal{S}^{\vee}$ with the regular coadjoint orbits.
		\item[\textup{(ii)}] The codimension of $\mathcal{S}^{\vee}\setminus(\mathcal{S}^{\vee})_{\emph{reg}}$ in $\mathcal{S}^{\vee}$ is at least three.
	\end{itemize}
\end{proposition}

\begin{proof}
	We begin by proving (i). Consider the adjoint quotient $\pi:\g^*\longrightarrow\mathrm{Spec}(\mathbb{C}[\g^*]^G)\eqqcolon\mathfrak{c}$. The restriction $\pi\big\vert_{\mathcal{S}^{\vee}}:\mathcal{S}^{\vee}\longrightarrow\mathfrak{c}$ is known to be faithfully flat with irreducible fibers of dimension $\dim\mathcal{S}-\ell$ \cite{slo:80,slo2:80}. Since each fiber of $\pi$ is a finite union of coadjoint orbits \cite[Theorem 3]{kos:63}, each fiber of $\pi\big{\vert_{\mathcal{S}^{\vee}}}$ must be a finite union of symplectic leaves of $\mathcal{S}^{\vee}$. It also follows that each fiber of $\pi\big\vert_{\mathcal{S}^{\vee}}$ is the closure of a $(\dim\mathcal{S}-\ell)$-dimensional symplectic leaf of $\mathcal{S}^{\vee}$.
	
	Suppose that $\alpha\in\mathfrak{c}$. There exists a unique regular coadjoint orbit $\mathcal{O}_{\alpha}\subset\g$ with the property that $\overline{\mathcal{O}_{\alpha}}=\pi^{-1}(\alpha)$ \cite[Theorem 3]{kos:63}. Each irreducible component of $\mathcal{S}^{\vee}\cap\mathcal{O}_{\alpha}$ is a $(\dim\mathcal{S}-\ell)$-dimensional symplectic leaf of $\mathcal{S}^{\vee}$ contained in $\mathcal{S}^{\vee}\cap\pi^{-1}(\alpha)$. Since $\mathcal{S}^{\vee}\cap\pi^{-1}(\alpha)$ is the closure of a $(\dim\mathcal{S}-\ell)$-dimensional symplectic leaf of $\mathcal{S}^{\vee}$, this leaf must be $\mathcal{S}^{\vee}\cap\mathcal{O}_{\alpha}$. We also know that $\alpha\mapsto\mathcal{O}_{\alpha}$ defines a bijection from $\mathfrak{c}$ to the set of regular coadjoint orbits \cite[Theorem 2]{kos:63}. It follows that each regular coadjoint orbit intersects $\mathcal{S}^{\vee}$ in a symplectic leaf. It is also clear that the regular symplectic leaves of $\mathcal{S}^{\vee}$ are the irreducible components of the intersections of $\mathcal{S}^{\vee}$ with the regular coadjoint orbits. These last two sentences combine to imply (i).
	
	We now verify (ii). Recall that a coadjoint orbit $\mathcal{O}\s\g^*$ is called \textit{semisimple} if it corresponds to a semisimple adjoint orbit under $(\cdot)^{\vee}:\g\longrightarrow\g^*$. Consider the locus $\mathfrak{c}^{\circ}\coloneqq\{\alpha\in\mathfrak{c}:\mathcal{O}_{\alpha}\text{ is semisimple}\}$ and its complement $\mathfrak{d}\coloneqq\mathfrak{c}\setminus\mathfrak{c}^{\circ}$. One knows that $\mathcal{O}_{\alpha}$ is closed for all $\alpha\in\mathfrak{c}^{\circ}$. Note also that $(\mathcal{S}^{\vee}\setminus(\mathcal{S}^{\vee})_{\text{reg}})\cap\pi^{-1}(\alpha)=\mathcal{S}^{\vee}\cap(\overline{\mathcal{O}_{\alpha}}\setminus\mathcal{O}_{\alpha})$ for all $\alpha\in\mathfrak{c}$, as follows from the second paragraph of this proof. We conclude that \begin{align*}\mathcal{S}^{\vee}\setminus(\mathcal{S}^{\vee})_{\text{reg}} & =\bigcup_{\alpha\in\mathfrak{c}}((\mathcal{S}^{\vee}\setminus(\mathcal{S}^{\vee})_{\text{reg}})\cap\pi^{-1}(\alpha)) \\ & =\bigcup_{\alpha\in\mathfrak{c}}(\mathcal{S}^{\vee}\cap(\overline{\mathcal{O}_{\alpha}}\setminus\mathcal{O}_{\alpha}))\\
	&=\bigcup_{\alpha\in\mathfrak{d}}(\mathcal{S}^{\vee}\cap(\overline{\mathcal{O}_{\alpha}}\setminus\mathcal{O}_{\alpha}))\\
	&=\bigcup_{\alpha\in\mathfrak{d}}((\mathcal{S}^{\vee}\setminus(\mathcal{S}^{\vee})_{\text{reg}})\cap\pi^{-1}(\alpha)).\end{align*} Let us also observe that $(\mathcal{S}^{\vee}\setminus(\mathcal{S}^{\vee})_{\text{reg}})\cap\pi^{-1}(\alpha)$ has codimension at least $2$ in $\mathcal{S}^{\vee}\cap\pi^{-1}(\alpha)$ for all $\alpha\in\mathfrak{c}$; this is implied by the first paragraph of the proof. The desired result now follows from $\mathfrak{d}$ having codimension $1$ in $\mathfrak{c}$.
\end{proof}

\subsection{Slodowy slices to the minimal nilpotent orbit in $\mathfrak{sl}_n$}
Let us specialize to the case $\g=\mathfrak{sl}_n$. Consider the $\mathfrak{sl}_2$-triple $(e,h,f)\in\mathfrak{sl}_n^{\times 3}$ given by
$$e=\begin{bmatrix}0 & 1 & 0 & \cdots & 0\\ 0 & 0 & 0 & \cdots & 0 \\ \vdots & \vdots & \vdots & \ddots  & \vdots \\ 0 & 0 & 0 & \cdots & 0\end{bmatrix},\quad h=\begin{bmatrix}1 & 0 & 0 & \cdots & 0\\ 0 & -1 & 0 & \cdots & 0 \\ 0 & 0 & 0 & \cdots & 0\\ \vdots & \vdots & \vdots & \ddots  & \vdots \\ 0 & 0 & 0 & \cdots & 0\end{bmatrix},\quad\text{and}\quad f=\begin{bmatrix}0 & 0 & \cdots & 0\\ 1 & 0 & \cdots & 0 \\ 0 & 0 & \cdots & 0\\ \vdots & \vdots & \ddots  & \vdots \\ 0 & 0 & \cdots & 0\end{bmatrix}.$$ Let us also consider the closed subvariety
$$\mathcal{T}\coloneqq\left\{\begin{bmatrix}0 & 1 & 0 & 0 & \cdots & 0\\ a_{n-2} & 0 & 1 & 0 & \cdots & 0 \\ \vdots & \vdots & \vdots & \ddots & \ddots  & \vdots \\ a_2 & 0 & 0 & 0 & \ddots & 0\\  a_1 & 0 & 0 & 0 & \cdots & 1\\  a_0 & 0 & 0 & 0 & \cdots & 0\end{bmatrix}:a_0,\ldots,a_{n-2}\in\mathbb{C}\right\}$$ of $\mathfrak{sl}_n$; it consists of the transposes of the trace-free $n\times n$ \textit{companion matrices}. A straightforward exercise reveals that $\mathcal{T}\subset\mathcal{S}\coloneqq e+(\mathfrak{sl}_n)_f$. It follows that $\mathcal{T}^{\vee}\s\mathcal{S}^{\vee}$.

\begin{proposition}\label{Proposition: Hartogs slice}
The subvariety $\mathcal{T}^{\vee}$ is a Hartogs slice to $\mathcal{S}^{\vee}$.
\end{proposition}

\begin{proof}
	One may use \cite[Lemma 10]{kos:63} to conclude that $\mathcal{T}\subset(\mathfrak{sl}_n)_{\text{reg}}$. In other words, $\mathcal{T}\s\mathcal{S}\cap(\mathfrak{sl}_n)_{\text{reg}}$. It follows that $\mathcal{T}^{\vee}\s\mathcal{S}^{\vee}\cap(\mathfrak{sl}_n^*)_{\text{reg}}=(\mathcal{S}^{\vee})_{\text{reg}}$, where the last instance of equality comes from Lemma \ref{Lemma: Intersections with regular orbits}.
	It remains only to prove that each regular symplectic leaf $L\subset\mathcal{S}^{\vee}$ intersects $\mathcal{T}^{\vee}$ in a single point. Proposition \ref{Proposition: Symplectic leaves in a Slodowy slice}(i) makes this the task of proving that each regular coadjoint orbit intersects $\mathcal{T}^{\vee}$ in a single point. This is equivalent to each regular adjoint orbit intersecting $\mathcal{T}$ in a single point. In other words, it suffices to prove that $\mathcal{T}$ is a section of the adjoint quotient $\pi:\mathfrak{sl}_n\longrightarrow\mathrm{Spec}(\mathbb{C}[\mathfrak{sl}_n]^{\operatorname{SL}_n})$.
	
	Write $\det(tI_n-x)=t^n+f_{n-2}(x)t^{n-2}+\cdots+f_1(x)t+f_0(x)$ for the characteristic polynomial of $x\in\mathfrak{sl}_n$. The polynomials $f_0,\ldots,f_{n-2}$ are algebraically independent generators of $\mathbb{C}[\mathfrak{sl}_n]^{\operatorname{SL}_n}$. The adjoint quotient of $\mathfrak{sl}_n$ is thereby the map $f=(f_0,\ldots,f_{n-2}):\mathfrak{sl}_n\longrightarrow\mathbb{C}^{n-1}$. It is also straightforward to check that $$f\left(\begin{bmatrix}0 & 1 & 0 & 0 & \cdots & 0\\ a_{n-2} & 0 & 1 & 0 & \cdots & 0 \\ \vdots & \vdots & \vdots & \ddots & \ddots & \vdots \\ a_2 & 0 & 0 & 0 & \ddots & 0\\  a_1 & 0 & 0 & 0 & \cdots & 1\\  a_0 & 0 & 0 & 0 & \cdots & 0\end{bmatrix}\right)=(-a_0,\ldots,-a_{n-2})$$ for all $a_0,\ldots,a_{n-2}\in\mathbb{C}$. This makes it clear that the restriction of the adjoint quotient $\pi:\mathfrak{sl}_n\longrightarrow\mathfrak{c}$ to $\mathcal{T}$ is an isomorphism of varieties. The proof is therefore complete. 
\end{proof}

One symplectic groupoid integrating $\mathcal{S}^{\vee}$ is the restriction $\G\tto\mathcal{S}^{\vee}$ of $T^*\SL_n\tto\mathfrak{sl}_n^*$ to $\mathcal{S}^{\vee}$.  It is also clear that $\mathcal{T}^{\vee}\subset(\mathfrak{sl}_n)^*_{\text{reg}}$ \cite[Lemma 10]{kos:63}, so that $(\SL_n)_{\xi}$ is abelian for all $\xi\in\mathcal{T}^{\vee}$. These considerations combine with Proposition \ref{Proposition: Hartogs slice} to imply that $\mathcal{T}^{\vee}$ is an admissible Hartogs slice to $\G\tto\mathcal{S}^{\vee}$. In light of Theorem \ref{2vumxgkk} and Proposition \ref{Proposition: Hartogs}, $\G\tto\mathcal{S}^{\vee}$ determines a TQFT.

\section{Examples arising from non-reductive groups}\label{Section: Non-reductive}

We continue to work exclusively over $\mathbb{C}$. In this section, we begin by defining the notion of a \textit{Moore--Tachikawa group} $G$. The definition provides sufficient conditions for $T^*G\tto\g^*$ to be Hartogs abelianizable, and includes all reductive groups as examples. In this way, Moore--Tachikawa groups give rise to TQFTs. In order to obtain TQFTs beyond those conjectured by Moore--Tachikawa, we must give examples of non-reductive Moore--Tachikawa groups. Most of this section is devoted to the construction of such groups.

\subsection{Moore--Tachikawa groups}
Let $G$ be an affine algebraic group with Lie algebra $\g$. Given $\xi\in\g^*$, let $G_{\xi}\s G$ and $\g_{\xi}\s\g$ denote the centralizers of $\xi$ under the coadjoint representations of $G$ and $\g$, respectively. We define
$$\g^*_{\text{reg}}\coloneqq\{\xi\in\g^*:\dim\g_{\xi}\leq\dim\g_{\eta}\text{ for all }\eta\in\g^*\}.$$ 

\begin{definition}\label{Definition: MT group}
A \textit{Moore--Tachikawa group} is an affine algebraic group $G$ with the following properties:
\begin{itemize}
\item[\textup{(i)}] $\g^*\reg$ has a complement of codimension at least two in $\g^*$;
\item[\textup{(ii)}] the pullback of the cotangent groupoid $T^*G\tto\g^*$ to $\g^*\reg$ is abelianizable, i.e. Morita equivalent to an abelian symplectic groupoid.
\end{itemize}
\end{definition}

A sufficient condition for (ii) to hold is the existence of a smooth, closed subvariety $S\s\g^*$ with the following properties:
\begin{itemize}
\item $S\s\gdreg$;
\item $S$ intersects every coadjoint orbit in $\gdreg$ transversely in a singleton;
\item $G_{\xi}$ is abelian for all $\xi\in S$. 
\end{itemize}

One shows this condition to be sufficient in a manner analogous to the proof of Proposition \ref{Proposition: Hartogs}.

Definition \ref{Definition: MT group} provides sufficient conditions for $T^*G\tto\g^*$ to be Hartogs abelianizable. By combining this observation with Theorem \ref{2vumxgkk}, one concludes that every Moore--Tachikawa group induces a TQFT in $\AMT$. On the other hand, results of Kostant \cite{kos:59,kos:63} imply that every reductive group is Moore--Tachikawa. The TQFTs induced by reductive groups are essentially those conjectured by Moore--Tachikawa. This motivates us to find examples of non-reductive Moore--Tachikawa groups. We devote the next two subsections to this task.

\subsection{The semidirect product of $\SL_2$ and its standard representation}

Let $G$ be an affine algebraic group and $\rho : G \longrightarrow \mathrm{GL}(V)$ a finite-dimensional, algebraic representation.
Consider the group semidirect product $H \coloneqq G \ltimes_\rho V$, i.e.\ we consider $V$ as a group with addition, so that multiplication in $H$ is given by
\[
(g_1, v_1) \cdot (g_2, v_2) = (g_1 g_2, v_1 + \rho_{g_1}(v_2)).
\]
The Lie algebra of $H$ is then $\h \coloneqq \g \ltimes_{\rho} V$, with bracket
\[
[(x_1, v_1), (x_2, v_2)] = ([x_1, x_2], \rho_{x_1}(v_2) - \rho_{x_2}(v_1)).
\]

\begin{proposition}
If $\rho$ is the standard representation of $\SL_2$, then the non-reductive group $H \coloneqq\SL_2\ltimes_\rho \C^2$ is Moore--Tachikawa. In particular, $H$ induces a TQFT $\Cob_2 \too \AMT$.
\end{proposition}

\begin{proof}
Let $V = \C^2$ be the standard representation of $\SL_2$.
We identify $V^*$ with $V = \C^2$ in an $\SL(2, \C)$-equivariant way via the standard symplectic form, i.e.
\[
\C^2 \overset{\cong}{\too} V^*, \quad (\eta_1, \eta_2) \mtoo \omega((\eta_1, \eta_2), \cdot),
\]
where $\omega((\eta_1, \eta_2), (v_1, v_2)) = \eta_1 v_2 - \eta_2 v_1$.
We also identify $\fsl_2^*$ with $\fsl_2$ via the invariant bilinear form $(x, y) \mto \tr(xy)$.
A straightforward computation then reveals that the coadjoint representation of $\h$ on $\h^* \cong \fsl_2 \times \C^2$ is
\[
\ad_{(x, u)}^*(\xi, \eta) =
\left(
\begin{pmatrix}
	(x_2 \xi_3 - x_3 \xi_2) - \tfrac{1}{2}(\eta_1 u_2 + \eta_2 u_1)
	& 2 (x_1 \xi_2 - x_2 \xi_1) + \eta_1 u_1 \\
	2 (x_3 \xi_1 - x_1 \xi_3) - \eta_2 u_2 & (x_3 \xi_2 - x_2 \xi_3) + \tfrac{1}{2}(\eta_1 u_2 + \eta_2 u_1)
\end{pmatrix},
\begin{pmatrix}
	x_1 \eta_1 + x_2 \eta_2 \\
	x_3 \eta_1 - x_1 \eta_2
\end{pmatrix}
\right),
\]
where $x = \left(\begin{smallmatrix} x_1 & x_2 \\ x_3 & -x_1\end{smallmatrix}\right)$, $u = \left(\begin{smallmatrix} u_1 \\ u_2 \end{smallmatrix}\right)$, $\xi = \left(\begin{smallmatrix} \xi_1 & \xi_2 \\ \xi_3 & -\xi_1\end{smallmatrix}\right)$, and $\eta = \left(\begin{smallmatrix} \eta_1 \\ \eta_2 \end{smallmatrix}\right)$. If $\eta \ne 0$, then $\ad_{(x, u)}^*(\xi, \eta) = 0$ if and only if there exists a constant $c \in \C$ such that
\[
x = c\begin{pmatrix} \eta_1 \eta_2 & -\eta_1^2 \\ \eta_2^2 & -\eta_1 \eta_2\end{pmatrix}
\quad\text{and}\quad
u = 
-2c\begin{pmatrix}
	\eta_1 \xi_1 + \eta_2 \xi_2 \\
	\eta_1 \xi_3 - \eta_2 \xi_1
\end{pmatrix}.
\]
It follows that the centralizer $\h_{(\xi,\eta)}$ is $1$-dimensional in this case. If $\eta = 0$, then $\ad_{(0, u)}^*(\xi, \eta) = 0$ for all $u\in\C^2$. The centralizer $\h_{(\xi,\eta)}$ must therefore have dimension at least two in this case. We conclude that
\[
\h^*\reg = \{(\xi, \eta) \in \fsl(2, \C) \times \C^2 : \eta \ne 0\}.
\]
This locus clearly has a complement of codimension two in $\h^*$.

It remains to verify (ii) in Definition \ref{Definition: MT group}. We accomplish this by verifying the sufficient condition mentioned immediately after that definition. To this end, let $S\s\h^*$ be the image of
\[
\sigma : \C \too \h^*\reg, \quad 
z \mtoo
\left(
\begin{pmatrix} 0 & 0 \\ z & 0 \end{pmatrix}
,
\begin{pmatrix} 1 \\ 0 \end{pmatrix}
\right).
\]
Suppose that $(\xi, \eta) \in \h^*\reg$.
Since $\eta \ne 0$, there exists $g \in \SL_2$ such that $\rho_g^*(\eta) = (1, 0)$.
We may therefore assume that $\eta = (1, 0)$.
Note that the stabilizer of $(1, 0)$ is the set of $g \in \SL_2$ of the form
\[
g = \begin{pmatrix} 1 & a \\ 0 & 1 \end{pmatrix},\quad a\in\mathbb{C}.
\]
For such $g$ and $\eta = (1, 0)$, we have
\begin{equation}\label{1prdw1s0}
\Ad_{(g, u)}^*(\xi, \eta) = 
\left(
\begin{pmatrix}
	\xi_1 + a \xi_3 - \tfrac{u_2}{2} &
	-2a \xi_1 + \xi_2 - a^2 \xi_3 + u_1 \\
	\xi_3 &  - \xi_1 - a \xi_3 + \tfrac{u_2}{2}
\end{pmatrix},
\begin{pmatrix}
	1 \\ 0
\end{pmatrix}
\right).
\end{equation}
It follows that there is a unique $u = (u_1, u_2)$ such that $\Ad_{(g, u)}^*(\xi, \eta)$ is $S$.
In other words, $S$ intersects every regular coadjoint orbit exactly once.
Moreover, the intersection is transverse.
Equation \eqref{1prdw1s0} also shows the following: for all $z \in \C$, the stabilizer of $\sigma(z) \in \h^*\reg$ is the set of elements of $H$ of the form $(\left(\begin{smallmatrix}1&a\\0&1\end{smallmatrix}\right),\left(\begin{smallmatrix}a^2z\\2az\end{smallmatrix}\right))$, $a \in \C$; this is abelian.
Our proof is therefore complete.
\end{proof}


\subsection{The centralizer of a minimal nilpotent element in $\fsl_3$}
One may weaken Definition \ref{Definition: MT group} by replacing $\gdreg$ with an arbitrary $G$-invariant open subset $U\s\g^*$. This open subset would be required to satisfy the following properties: 
\begin{itemize}
\item[\textup{(i)}] $U$ has a complement of codimension at least two in $\g^*$;
\item[\textup{(ii)}] the pullback of the cotangent groupoid $T^*G\tto\g^*$ to $U$ is abelianizable, i.e. Morita equivalent to an abelian symplectic groupoid.
\end{itemize}

A sufficient condition for (ii) to hold would be the existence of a smooth, closed subvariety $S\s\g^*$ with the following properties:
\begin{itemize}
	\item $S\s U$;
	\item $S$ intersects every coadjoint orbit in $U$ transversely in a singleton;
	\item $G_{\xi}$ is abelian for all $\xi\in S$. 
\end{itemize}

Suppose that an affine algebraic group $G$ admits a $G$-invariant open subset $U\s\g^*$ satisfying (i)--(iii). The cotangent groupoid $T^*G\tto\g^*$ is then clearly Hartogs abelianizable. As such, it determines a TQFT $\Cob_2\longrightarrow\AMT$.

\begin{remark}
The term \textit{Moore--Tachikawa group} could have been reserved for this generalization of Definition \ref{Definition: MT group}, i.e. for an affine algebraic group $G$ admitting a $G$-invariant open subset $U\s\g^*$ satisfying (i) and (ii) above. One of the reasons for preserving Definition \ref{Definition: MT group} is the bridge it provides to pure Lie theory. A Lie theorist could find meaningful examples satisfying Definition \ref{Definition: MT group} without reading our manuscript in detail.      
\end{remark}

\begin{proposition}\label{Proposition: sl3}
If $G \s \SL_3$ is the stabilizer of
\[
e \coloneqq \begin{pmatrix} 0 & 0 & 1 \\ 0 & 0 & 0 \\ 0 & 0 & 0 \end{pmatrix} \in \fsl(3, \C),
\]
then $G$ satisfies Conditions (i) and (ii) above. It thereby determines a TQFT $\Cob_2\longrightarrow\AMT$.
\end{proposition}

\begin{proof}
The group $G$ can be written explicitly as
\[
G = \left\{ \begin{pmatrix} r & a & c \\ 0 & r^{-2} & b \\ 0 & 0 & r \end{pmatrix} : (r, a, b, c) \in \C^{\times} \times \C^3\right\},
\]
with Lie algebra
\begin{equation}\label{92o57qah}
	\g = \left\{\begin{pmatrix} t & x & z \\ 0 & -2t & y \\ 0 & 0 & t \end{pmatrix} : (t, x, y, z) \in \C^4 \right\}.
\end{equation}
Identify $\g^*$ with $\C^4$ via the coordinates dual to those in \eqref{92o57qah}.
For $g = (r, a, b, c) \in G$ and $(s, u, v, w) \in \g^*$, we have
\[
\Ad_g^*(s, u, v, w) =
\left(
s + 3 \left(\frac{a}{r} u - b r^2 v + abr w\right),
\frac{u}{r^3} + \frac{wb}{r},
v r^3 - war^2,
w
\right).
\]
One then sees that we have two transverse slices, given by
\begin{align*}
\sigma_1 :	\C \times \C &\too \g^*, \quad (v, w) \mtoo (0, 1, v, w) \\
\sigma_2 :	\C \times \C &\too \g^*, \quad (u, w) \mtoo (0, u, 1, w). 
\end{align*}
They are slices for the set of regular elements with $(u, w) \ne (0, 0)$ and $(v, w) \ne (0, 0)$, respectively, which both have complements of codimension two.
If $w \ne 0$, the stabilizer of $\sigma_1(v, w) = (0, 1, v, w) \in \g^*$ is the set of $g = (r, a, b, c)$ such that $a = \frac{v}{w}(r - 1/r^2)$ and $b = \frac{1}{w}(r - 1/r^2)$, which is abelian.
If $w = 0$, the stabilizer is the set of $g = (r, a, b, c)$ such that $r^3 = 1$ and $a = bv$, which is also abelian.
It follows that $\sigma_1$ is an admissible Hartogs slice.
A similar argument shows that $\sigma_2$ has the same property.
\end{proof}

\begin{remark}
More generally, let $\g$ be a simple Lie algebra of type $A$ or $C$. Consider the centralizers $G_e\s G$ and $\g_e\s\g$ of a nilpotent element $e\in\g$ under the adjoint representations of $G$ and $\g$, respectively. By \cite[Theorem 3.14]{pan-pre-yak:07}, $(\g_e)^*_{\text{reg}}$ has a complement of codimension at least two in $(\g_e)^*$. It is therefore reasonable to expect $G_e$ to satisfy Conditions (i) and (ii) from the beginning of this subsection, and for a more general version of Proposition \ref{Proposition: sl3} to hold.
\end{remark}

%
%
%
%

\bibliographystyle{acm} 
\bibliography{tqfts}

\end{document}